\documentclass[a4paper,10pt]{article}
\usepackage[margin=0.80in,bottom=0.90in,top=0.85in]{geometry}

\usepackage[english]{babel}
\usepackage{chngcntr}
\usepackage[utf8]{inputenc}
\usepackage{amsmath,mathrsfs}
\usepackage{indentfirst}
\usepackage{fancyhdr}
\usepackage{amssymb}
\usepackage{amsthm}
\usepackage[dvips]{graphicx}
\usepackage{color}
\usepackage{xcolor}
\usepackage{cancel}
\usepackage{subcaption}
\usepackage[dvipsnames]{xcolor}
\usepackage{eucal}
\usepackage{latexsym}
\usepackage{chngcntr}
\usepackage{faktor}
\usepackage{microtype}
\usepackage{newlfont}
\usepackage{enumitem}
\usepackage{geometry}
\usepackage{nicematrix}
\usepackage{tikz}
\usepackage{todonotes}

\usepackage{amsmath, amsfonts, amsthm,amssymb,  bbm}

\usepackage{tikz-cd}
\usepackage{multicol}
\usepackage{cancel}
\usepackage{mathtools}
\usepackage{graphicx}

\usepackage[colorlinks=true,    
            linkcolor=blue,    
            citecolor=blue,     
            urlcolor=blue]{hyperref}

\newcommand{\e}{\epsilon}
\newcommand{\cO}{\mathcal{O}}
\newcommand{\ch}{\mathtt{c}_{\tth}}
\newcommand{\chk}{\mathtt{c}_{\tth, \kappa}}
\newcommand{\tth}{\mathtt{h}}
\newcommand{\ttf}{\mathtt{f}_\e}
\newcommand{\pa}{\partial}

\newcommand{\de}{\textrm{d}}
\newcommand{\tf}{\mathtt{f}}
\newcommand{\cA}{{\mathcal A}}
\newcommand{\cJ}{{\mathcal J}}

\newcommand{\N}{{\mathbb{N}}}
\newcommand{\R}{{\mathbb{R}}}
\newcommand{\Z}{{\mathbb{Z}}}
\newcommand{\T}{{\mathbb{T}}}

\newcommand{\vet}[2]{\begin{bmatrix}#1 \\ #2 \end{bmatrix}}
\newcommand{\cB}{\mathcal{B}}

\newcommand{\fR}{\mathfrak{R}}

\newcommand{\im}{\textup{i}}

\newcommand{\und}[1]{\underline{#1}}

\newcommand{\cL}{\mathcal{L}}

\newcommand{\vertiii}[1]{{\left\vert\kern-0.25ex\left\vert\kern-0.25ex\left\vert #1 
    \right\vert\kern-0.25ex\right\vert\kern-0.25ex\right\vert}}
\newcommand{\opnorm}[1]{{\vert\kern-0.25ex\vert\kern-0.25ex\vert #1 
    \vert\kern-0.25ex\vert\kern-0.25ex\vert}}

\theoremstyle{plain}
\newtheorem{lemma}{Lemma}
\newtheorem{theorem}[lemma]{Theorem}

\newtheorem{remark}[lemma]{Remark}
\newtheorem{definition}[lemma]{Definition}

\theoremstyle{definition} 

\numberwithin{equation}{section}
\counterwithin{lemma}{section}

\title{Full Benjamin-Feir instability \\ of  capillary-gravity Stokes waves in finite depth}

\begin{document}

\author{Ting-Yang Hsiao\footnote{International School for Advanced Studies (SISSA), Via Bonomea 265, 34136, Trieste, Italy. \newline
	\textit{Emails:} \texttt{thsiao@sissa.it}, \texttt{amaspero@sissa.it}},  Alberto Maspero$^*$ }









\maketitle

\begin{abstract}
We study the two-dimensional gravity–capillary water waves equations for a fluid of finite depth $\mathtt{h}>0$ under the combined effects of gravity and surface tension $\kappa \geq 0$. 
We analyze the linear stability and instability of 
 small-amplitude, 
$2\pi$-periodic Stokes wave solutions,  under the effect of longitudinal long-wave perturbations.
The corresponding linearized operator has  periodic coefficients and a defective zero eigenvalue of multiplicity four.
Using Bloch–Floquet theory, we investigate the associated family of periodic eigenvalue problems. For all surface tension values $\kappa \geq 0$ and depths $\mathtt{h} > 0$, we establish the complete splitting of the four eigenvalues near zero when both the wave amplitude and the Floquet parameter are small. Specifically, we rigorously prove that in the regions of unstable depth and capillarity identified formally by Djordjevic–Redekopp and Ablowitz–Segur in the 1970's, the spectrum of the linearized operator near the origin depicts a ``figure 8'' pattern.
\end{abstract}

\tableofcontents

\section{Introduction}\label{sec1}
 Stokes waves, first described in the groundbreaking work of Stokes in 1847 \cite{stokes}, represent one of the earliest known global-in-time solutions for dispersive PDEs. They are spatially periodic water waves that travel rigidly in one direction, remaining uniform in the transverse direction, and  persist indefinitely despite dispersive effects.
 The extensive literature on this topic is surveyed in \cite{ebb}. 

A long-standing challenge   in fluid dynamics is determining the stability of Stokes wave  under longitudinal long-wave  perturbations. This question was  answered in 1967 when Benjamin and Feir \cite{BF} experimentally demonstrated that small-amplitude Stokes waves are unstable under long-wave perturbations, and  proposed a heuristic justification for this phenomenon. 
Similar explanations were independently advanced also by Lighthill \cite{Li} and Zakharov \cite{Zak1}. 
This instability is nowadays  known as Benjamin-Feir or modulational instability, and it is supported by an enormous amount of physical observations and numerical simulations.

The first rigorous mathematical proof of the Benjamin-Feir instability 
was obtained thirty years ago    by 
Bridges-Mielke \cite{BrM} and more recently by
Nguyen-Strauss  \cite{NS}. These works prove, in case of  
 2d pure gravity water waves in finite respectively deep waters, 
that the linearized water waves operator at a Stokes wave of small amplitude  has unstable spectrum near the origin.
In the last few years,  thanks to the introduction of new methods and techniques,   Berti-Maspero-Ventura were able to obtain the full description of the $L^2(\R)$-spectrum of the linearized operator  near the origin  in deep waters \cite{BMV1}, finite depth \cite{BMV3}, and at the critical Whitham-Benjamin  depth $\tth_{\text{WB}} \approx 1.3627827\ldots$ \cite{BMV_ed}, proving the ``figure 8'' conjecture.

This paper examines the linear stability and instability of small-amplitude Stokes waves in the context of 2d gravity-capillary water waves.
Theoretical work conducted in the 1970s by Djordjevic and Redekopp \cite{DjRe}, as well as Ablowitz and Segur \cite{AbSe}, pinpointed zones of capillarity and depth --illustrated in gray in Figure \ref{fig:DR}-- for which small amplitude  Stokes waves are  linearly unstable. 
 Recently, Hur and Yang \cite{HY2} and Sun and Wahlen \cite{SW} have provided rigorous proofs, confirming that in the Djordjevic-Redekopp  unstable zones, the linearized operator has  unstable spectrum near the origin.

The goal of this paper is to provide the full description of the spectrum near zero of the linearized gravity-capillary water waves operator  at a small amplitude Stokes wave. 
In particular, we  prove that whenever the capillarity and depth fall within the unstable zones identified by Djordjevic and Redekopp, the spectrum of the linearized operator near the origin has the  ``figure 8" shape. 
Let us now present precisely our results.

\paragraph{Benjamin-Feir instability for gravity-capillary water waves.}
We consider the two-dimensional water wave equations governing a fluid of finite depth $\tth>0$ under the influence of gravity $g$ (that we set equal to 1) and capillary forces with surface tension $\kappa \geq 0$. 
We examine the linear stability/instability of a $2\pi$-periodic Stokes wave solution with a small amplitude $0 < \e \ll 1$ and   speed  $c_\epsilon = \chk + \mathcal{O}(\epsilon^2)$, where $\chk = \sqrt{(1+\kappa)\tanh(\tth)}$.
The linearized  water waves equations at the Stokes wave
are, in the inertial reference frame moving with speed $c_\e$, a linear time independent system of the form
$ h_t = \mathcal{L}_{\e}  h  $ where $ \mathcal{L}_{\e} := \mathcal{L}_{\e}({\kappa, \mathtt h}) $
is a linear operator  with $ 2 \pi $-periodic coefficients,  
see \eqref{mathcal L e} (such operator is actually  obtained 
conjugating the linearized water waves equations at the Stokes wave in the Zakharov  formulation 
via the  ``good unknown of Alinhac" \eqref{alinhac} and the 
Levi-Civita \eqref{LC} invertible transformations).
The operator 
$ \mathcal{L}_{\e} $ possesses the  eigenvalue $ 0 $, 
which is defective, with multiplicity four,  
 due to  symmetries of the water waves equations. 
The problem is to prove that  the linear system 
 $ h_t = \mathcal{L}_{\e}  h  $  
has  solutions  of the form $h(t,x) = \text{Re}\left(e^{\lambda t} e^{\im \mu x} v(x)\right)$
where $v(x)$ is a  $2\pi$-periodic function, $\mu$ in $ \R$ is the  Floquet exponent
and $\lambda$ has positive real part, thus $h(t,x)$ grows exponentially in time.
By Bloch-Floquet theory, such $\lambda$ is an  eigenvalue 
  of the operator  $ \mathcal{L}_{\mu,\e} 
:= e^{-\im \mu x } \,\mathcal{L}_{\e} \, e^{\im \mu x } $
acting on $2\pi$-periodic functions. 
\smallskip

 The main result of this paper proves, for any   value of the surface tension $\kappa \geq 0$ and depth $ \mathtt h>0 $, 
  the full splitting of the four 
 eigenvalues  close to 
zero of the operator $ \mathcal{L}_{\mu,\e}  := \mathcal{L}_{\mu,\e} (\kappa,\mathtt h ) $ 
when  $ \e $ and $ \mu $ are small enough, see Theorem \ref{Complete BF thm}.
We first present Theorem \ref{thm:main} focusing  on the figure $``8" $ formed by 
the Benjamin-Feir unstable eigenvalues. 
Such figure $``8" $ forms whenever $\kappa$ and $\tth$ belong to certain open sets, whose union we denote by  $\mathcal{U}$, c.f. \eqref{def:cU},   that were  identified before by  Djordjevic-Redekopp \cite{DjRe} and Ablowitz-Segur \cite{AbSe} via a formal modulational analysis. We picture them in gray color in  Figure \ref{fig:DR}.

 To describe such zones analytically, we need to introduce several functions of $\kappa$ and $\tth$.
The first one is 
\begin{equation}\label{def:eWB}
    \begin{aligned}
        \mathsf{e}_{\mathrm{WB}}:=\mathsf{e}_{\mathrm{WB}}(\kappa,\tth):=&\frac{(21-25\ch^4+6\ch^8)(1+\kappa)^2+(-12+6\ch^4+3\ch^8)(1+\kappa)+9\ch^4}{8\ch^3(1+\kappa)^{\frac{1}{2}}\,\left(\ch^4(1+\kappa) -3\kappa\right)}\\
        &-\frac{(1+\kappa)^{\frac{1}{2}}\,\left( (3\tth \ch^8 - 6\ch^6 - 6\tth\ch^4 + 6\ch^2 + 3\tth)
(1+\kappa)+4\ch^6\right)}{4\ch^3\,   \textup{D}_{\tth,\kappa}}
    \end{aligned}
\end{equation}
where
$\ch:= \sqrt{\tanh(\tth)}$ is the speed of the linear Stokes wave with zero capillarity, and
\begin{equation}\label{def:D}
 \textup{D}_{\tth,\kappa}:=\tth-\frac{1}{4}\mathsf{e}_{12}^2 \ , \quad 
    \mathsf{e}_{12}:=  \mathsf{e}_{12}(\kappa, \tth) :=
    2(1+\kappa)^{\frac12} \ch \left[ \frac{\ch^2 + (1 - \ch^4)\tth}{2 \ch^2} + \frac{\kappa}{1+\kappa}\right]. 
\end{equation}
{The function $\mathsf{e}_{12}$ is positive for all $(\kappa,\tth)\in \R_{\geq 0} \times \R_{>0}$, $\R_{\geq 0} :=[0,\infty)$, $\R_{>0}:=(0,\infty)$.
However, the function  $\textup{D}_{\tth,\kappa}$ vanishes} on an analytic curve, that we denote by
\begin{align} \label{def:fD}
\mathfrak{D} := \{\, (\kappa,\tth) \in \R_{\geq 0} \times \R_{>0} :\ \  \textup{D}_{\tth,\kappa} \   =  0 \,\}, \quad \textup{D}_{\tth,\kappa} \mbox{ in } \eqref{def:D} \ . 
\end{align}
Actually, in Appendix \ref{sec:D.sign} we show that $\textup{D}_{\tth,\kappa}<0$ under the Bond number condition $\kappa > \tth^2/3$. 
Also the denominator $\ch^4(1+\kappa) -3\kappa$ in the first addend of 
$ \mathsf{e}_{\mathrm{WB}}$ might vanish.  To take this into account, we introduce the set 
\begin{equation}\label{def:fR}
    \fR := \bigcup_{n \in \N, \   n \geq 2} \fR_n := \left\lbrace (\kappa, \tth) \in \R_{\geq 0} \times \R_{>0} \colon \ \ 
\kappa = \frac{\tanh({n \tth}) - n \tanh({\tth})}{n \left(\tanh(\tth) - n \tanh(n\tth) \right)}
    \right\rbrace \ . 
\end{equation}
A direct computation, using  
$\tanh(2\tth) =\frac{2 \ch^2}{1+\ch^4}$, gives 
$  \kappa - \frac{\tanh({2 \tth}) - 2 \tanh({\tth})}{2 \left(\tanh(\tth) - 2 \tanh(2\tth) \right)} = \frac{(1+\kappa) \ch^4 - 3\kappa}{\ch^4-3} $, therefore
 \begin{equation} \label{denneq0}
     (1+\kappa) \ch^4 - 3\kappa \neq 0 \quad \mbox{ provided } (\kappa, \tth) \not \in \fR_2 \ . 
 \end{equation}
We shall also require that $(\kappa, \tth) \not \in \fR_n$ whenever $n \geq 3$. 
These conditions are necessary to avoid  resonances in the construction of $2\pi$-periodic Stokes waves, see the comments after Theorem \ref{Thm: Stokes expansion}.
In Figure \ref{fig:DR}, the curve $\fR_2$ is plotted as  line $3$ (red color), and we also plot in dotted lines some of the curves $\fR_n$, $n \geq 3$, that accumulate to $\kappa =0$ at every fixed  $\tth >0$ as $n \to \infty$. 

Finally we  introduce the function
\begin{equation}\label{def:e22}
\mathsf{e}_{22}:=\mathsf{e}_{22}(\kappa,\tth) :={\frac{(-3\ch^8\tth^2+6\,\ch^6\tth+2\,\ch^4\tth^2-3\ch^4-6\ch^2\tth+\tth^2)(1+\kappa)^2+(4\ch^2\tth-4\ch^6\tth)(1+\kappa)+4\ch^4}{\ch^3{\left(1+\kappa \right)}^{\frac{3}{2}}}}.
\end{equation}
We define the  {\em unstable}  region
\begin{equation}\label{def:cU}
    \mathcal{U}:= \left\{ (\kappa, \tth) \in (\R_{\geq 0} \times \R_{>0}) \setminus (\fR  \cup \mathfrak{D}) \colon \quad \mathsf{e}_{22}(\kappa,\tth)\mathsf{e}_{\mathrm{WB}}(\kappa,\tth)> 0 \right\}  \ , \quad \mathfrak{D} \mbox{ in } \eqref{def:fD}, \quad \fR \mbox{ in } \eqref{def:fR},
\end{equation}
and the {\em stable} one
\begin{equation}
    \mathcal{S}:= \left\{ (\kappa, \tth) \in \R_{\geq 0} \times \R_{>0} \setminus (\fR  \cup \mathfrak{D}) \colon \quad \mathsf{e}_{22}(\kappa,\tth)\mathsf{e}_{\mathrm{WB}}(\kappa,\tth)< 0 \right\} \ . 
\end{equation}
Our main result says that whenever $(\kappa, \tth) \in  \mathcal{U}$, the spectrum of the linearized water waves operator $\cL_{\mu, \e}$ near the origin depicts a complete ``figure 8''.
On the contrary, if $(\kappa, \tth) \in  \mathcal{S}$, the spectrum near the origin is purely imaginary.
In Figure \ref{fig:DR} we plot the unstable region  $\mathcal{U}$ is gray color, and the stable region $\mathcal{S}$ in white color. 
Such figure is the same obtained previously by 
Djordjevic-Redekopp \cite{DjRe} and Ablowitz-Segur \cite{AbSe}. We shall comment more about the relation with these papers after Theorem \ref{thm:main}.
\begin{figure}[t]
  \centering
  \includegraphics[scale=0.9]{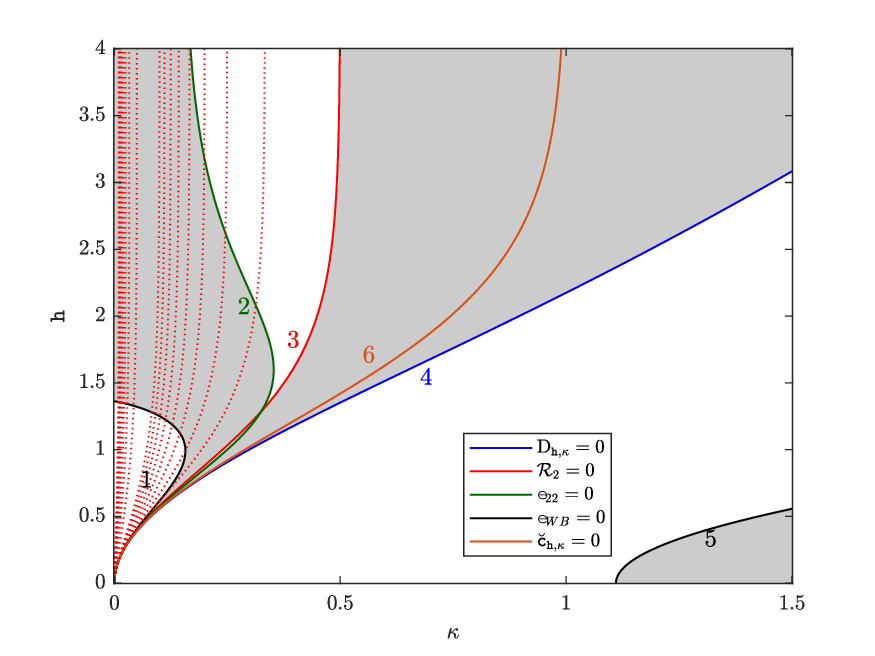} 
  \caption{Stability diagram for capillary–gravity Stokes wave. The capillarity $\kappa$ is on the horizontal axis and the depth $\tth$ on the vertical axis. 
  The light gray region is the  unstable region of parameters $\mathcal{U}$ in \eqref{def:cU}, while the white region is the stable region $\mathcal{S}$.
  Lines 1 and 5 (black color) are the zero set of the function $\mathsf{e}_{\mathrm{WB}}$ in \eqref{def:eWB}.
  Line 2 (green color) is the zero set of the function $\mathsf{e}_{22}$ in \eqref{def:e22}. Line 3 (red color) is the curve $(1+\kappa) \ch^4 - 3\kappa  = 0$, i.e. the curve defining the set $\fR_2$, c.f. \eqref{def:fR}. 
  Line 4 (blue line) is the zero set of the function $\textup{D}_{\tth,\kappa}$ in \eqref{def:D}. Line 6 (orange line) is the zero set of the function $\breve{\mathtt c}_{\tth,\kappa}$. Lines 3 and 4 are the singular lines of the function $\mathsf{e}_{\mathrm{WB}}$ on which the denominator vanishes. We also mark the curves $\mathfrak{R}_3,\mathfrak{R}_4,\ldots,\mathfrak{R}_{10}$ and $\mathfrak{R}_{10},\mathfrak{R}_{20},\ldots,\mathfrak{R}_{100}$, which are plotted as dotted lines and arranged in order from right to left.}
  \label{fig:DR}
\end{figure}

\smallskip
Along  the paper we denote by $r(\e^{m_1} \mu^{n_1}, \ldots, \e^{m_p} \mu^{n_p})$
a real analytic function fulfilling  for some $C >0$ and $\e, \mu$ sufficiently small, the estimate  $| r(\e^{m_1} \mu^{n_1}, \ldots, \e^{m_p} \mu^{n_p}) | \leq 
C \sum_{j=1}^p
 |\e|^{m_j} |\mu|^{n_j}
$, where  the constant $C:=C(\kappa,\tth)$ is uniform for $(\kappa,\tth)$ in any compact set of $\R_{\geq 0} \times \R_{>0}$.

\begin{theorem}[Benjamin--Feir unstable eigenvalues] \label{thm:main} 
Let  $(\kappa,\tth)\in \mathcal{U} $, c.f. \eqref{def:cU}.
There exist $\epsilon_1, \mu_0 > 0$ and an analytic function $\und{\mu} : [0, \epsilon_1) \to [0, \mu_0)$,
of the form
\begin{align} \label{und mu}
\und{\mu}(\epsilon) = \mathsf{e}_{\tth,\kappa} \epsilon (1 + r(\epsilon)), 
\qquad 
\mathsf{e}_{\tth,\kappa} := \sqrt{\frac{8 \mathsf{e}_{\mathrm{WB}}(\kappa,\tth)}{\mathsf{e}_{22}(\kappa,\tth)}},
\end{align}
such that, for any $\epsilon \in [0, \epsilon_1)$, the operator $\mathcal{L}_{\mu, \epsilon}$ has two eigenvalues $\lambda^{\pm}_1(\mu, \epsilon)$ of the form
\begin{equation}\label{eigelemu}
\begin{cases}
\im \frac{1}{2} \breve{\mathtt{c}}_{\tth,\kappa}\mu + \im r_{2}(\mu \epsilon^{2}, \mu^{2}\epsilon, \mu^{3}) 
\pm \tfrac{1}{8} \mu (1+r(\epsilon,\mu)) \, \sqrt{\Delta_{\mathrm{BF}}(\tth,\kappa;\mu,\epsilon)}, 
& \forall \mu \in [0, \mu(\epsilon)), \\[1em]
\im \frac{1}{2} \breve{\mathtt{c}}_{\tth,\kappa}\mu + \im r_{2}(\mu \epsilon^{2}, \mu^{2}\epsilon, \mu^{3}), 
& \mu = \mu(\epsilon), \\[1em]
\im \frac{1}{2} \breve{\mathtt{c}}_{\tth,\kappa}\mu + \im r_{2}(\mu \epsilon^{2}, \mu^{2}\epsilon, \mu^{3}) 
\pm \im \tfrac{1}{8} \mu (1+r(\epsilon,\mu))\, \sqrt{\left|\Delta_{\mathrm{BF}}(\tth,\kappa;\mu,\epsilon)\right|}, 
& \forall \mu \in (\mu(\epsilon), \mu_0)
\end{cases}
\end{equation}
where $\breve{\mathtt{c}}_{\tth,\kappa} := 2 \chk - \mathsf{e}_{12}(\kappa,\tth)$ and 
$\Delta_{\mathrm{BF}}(\kappa,\tth; \mu, \epsilon)$ is the \emph{Benjamin--Feir discriminant function}
\begin{align}\label{BFDF}
    \Delta_{\mathrm{BF}}(\kappa,\tth;\mu,\epsilon):=8\epsilon^2 \mathsf{e}_{22}(\kappa,\tth)\mathsf{e}_{\mathrm{WB}}(\kappa,\tth)+r_1(\e^3,\mu\e^2)-\mathsf{e}^2_{22}(\kappa,\tth)\mu^2(1+r_1''(\e,\mu))
\end{align}
where $\mathsf{e}_{\mathrm{WB}}(\kappa,\tth)$ in \eqref{def:eWB} and $\mathsf{e}_{22}(\kappa,\tth)$ in \eqref{def:e22}.
Note that, since  $(\kappa,\tth)\in \mathcal{U} $, for any $0 < \epsilon < \epsilon_1$ (depending on $(\kappa, \tth)$), 
the function $\Delta_{\mathrm{BF}}(\kappa,\tth; \mu, \epsilon)$ is positive, respectively $<0$, 
provided $0 < \mu < \mu(\epsilon)$, respectively $\mu > \mu(\epsilon)$.
\end{theorem}

Let us make some comments:

\begin{enumerate}
    \item {\sc Benjamin-Feir eigenvalues.}
Whenever $(\kappa, \tth) \in \mathcal{U}$, 
 for values of the Floquet parameter $ 0<\mu <  \underline \mu (\e) $, the eigenvalues 
$\lambda^\pm_1 (\mu, \epsilon) $ have opposite non-zero real part, according to \eqref{eigelemu}.
The ``figure 8'' forms in the upper or lower semiplane provided the function
$\breve{\mathtt{c}}_{\tth,\kappa}$ does not vanish. 
We plot in Figure \ref{fig:DR} the zero set of $\breve{\mathtt{c}}_{\tth,\kappa}$, see  line 6.
In particular, in the zone above line 6, $\breve{\mathtt{c}}_{\tth,\kappa}>0$. In this case, the imaginary part \eqref{eigelemu} is positive as $\mu>0$ and the two eigenvalues 
$\lambda^\pm_1 (\mu,\epsilon) $ stays in the upper semiplane $ \text{Im} (\lambda) > 0 $. 
As $ \mu $ tends to $  \underline  \mu (\e)$, the two eigenvalues $\lambda^\pm_1 (\mu,\epsilon) $ 
 collide on the imaginary axis {\it far} from $ 0 $,
along which they  keep moving  for $ \mu > \underline  \mu (\e) $.
In particular for $ \mu \in [0, \underline \mu(\e)]$  
we obtain the upper part of  
the figure  ``8'', see Figure \ref{figure-eigth},  which is 
well approximated by the curves 
\begin{equation}\label{appdr}
\mu \mapsto \Big( \pm\frac{\mu}{8}  \sqrt{8\epsilon^2 \mathsf{e}_{22}(\kappa,\tth)\mathsf{e}_{\mathrm{WB}}(\kappa,\tth)-\mathsf{e}^2_{22}(\kappa,\tth)\mu^2},  \ 
\tfrac12 \breve{\mathtt c}_{\tth,\kappa} \mu \Big) \,  \ , 
\end{equation}
confirming the numerical results by   Deconinck-Trichtchenko \cite{DT}.
For $ \mu < 0 $ the operator $ {\mathcal L}_{\mu,\e} $
possesses the symmetric eigenvalues 
$ \overline{\lambda_1^{\pm} (-\mu,\e)} $ in the semiplane $ \text{Im} (\lambda) < 0 $, which give rise to the lower portion of the ``figure 8''. 

Viceversa, in the region  below curve 6 where the function $\breve{\mathtt{c}}_{\tth,\kappa}<0$,   the imaginary part \eqref{appdr} is negative as $\mu>0$ and the ``figure 8'' forms in the lower semiplane as 
$ \mu \in [0, \underline \mu(\e)]$  (of course the upper ``figure 8'' is formed when $\mu <0$ by  symmetry).

\begin{figure}[h!] 
  \centering
   \includegraphics[scale=0.5]{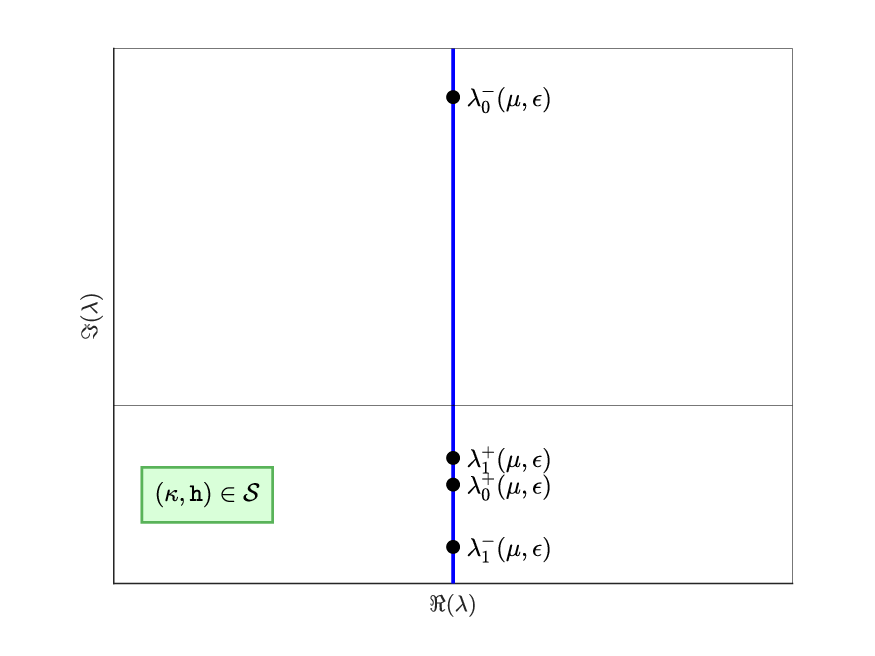}
  \includegraphics[scale=0.5]{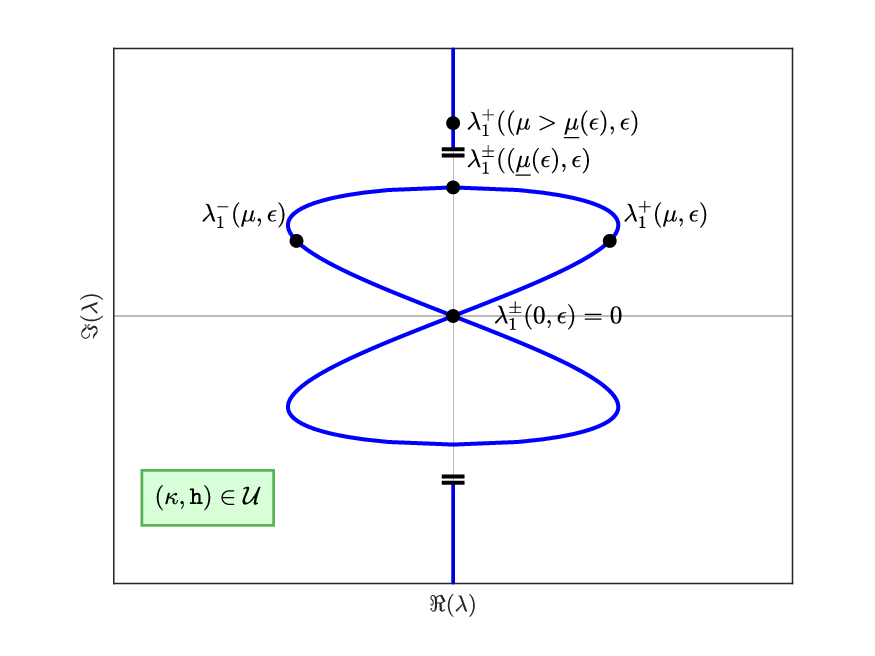}
  \caption{
    The left panel illustrates the eigenvalues $\lambda_{1}^{\pm}(\mu,\epsilon)$ and $\lambda_{0}^{\pm}(\mu,\epsilon)$ 
    in the regime $(\kappa,\tth)\in\mathcal{S}$, where they remain purely imaginary.
    In contrast, the right panel corresponds to the regime $(\kappa,\tth)\in\mathcal{U}$,
    showing the complex-plane trajectories of $\lambda_{1}^{\pm}(\mu,\epsilon)$ for fixed $|\epsilon| \ll 1$ as $\mu$ varies.
    The ``figure-eight'' shape depends on the depth $\tth$ and capillarity $\kappa$. As $\kappa = 0$, the spectrum agrees with~\cite{BMV3} and approaches the deep-water limit described in~\cite{BMV1} as $\tth \to +\infty$.}

  \label{figure-eigth}
\end{figure}

\item {\sc Stabilizing effect of the capillarity.}
Whatever the value of the depth $\tth>0$, it is possible to select the capillarity $\kappa$ so that  $ \Delta_{\mathrm{BF}}(\kappa, \tth) <0$ for $\e, \mu$ sufficiently small,  so that the linearized operator $\cL_{\mu,\e}$ has near the origin only purely imaginary eigenvalues. 
Indeed curve 4  in Figure \ref{fig:DR}, i.e. $\textup{D}_{\tth,\kappa} =0$,  is asymptotic to the line $\tth = \frac94 \kappa - \frac34$ as $\kappa \to \infty$,   see   Figure \ref{fig:asymptotic lines} orange line.
Curve 5, a subset of 
$\mathsf{e}_{\mathrm{WB}}(\kappa,\tth) = 0$, is asymptotic to the line $\tth = \frac94 \kappa - \frac{35}{4}$, see  Figure \ref{fig:asymptotic lines}  yellow line. Hence the stable region $\mathcal{S}$ contains asymptotically a  strip of fixed size  between curves 4 and 5, and values of $(\kappa, \tth)$ inside such strip give rise to Stokes wave linearly stable  near the origin.
\begin{figure}[h!] 
  \centering
  \includegraphics[scale=0.7]{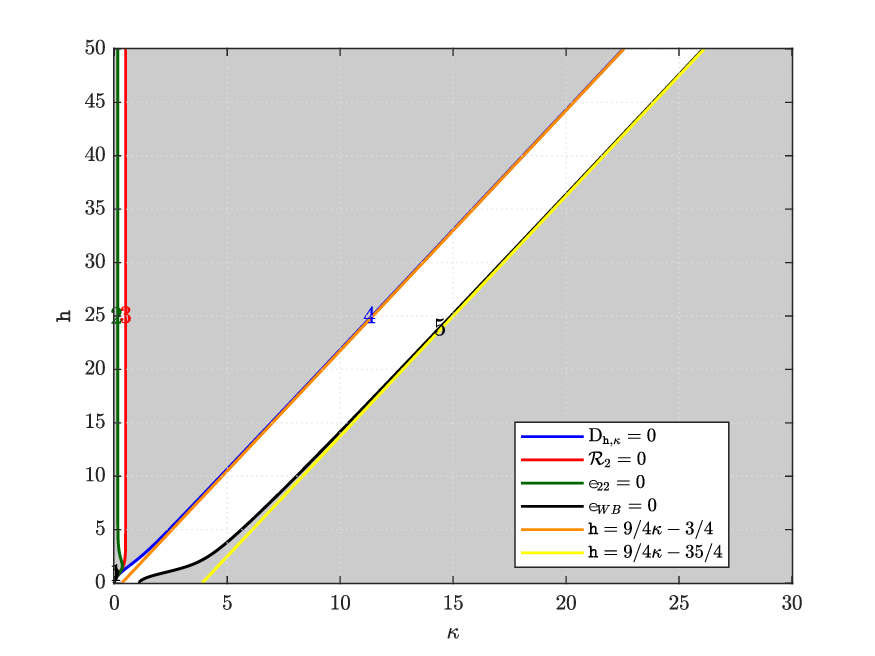} 
  \caption{As $\kappa \to \infty$, Curve~4 (given by $\mathrm{D}_{\tth,\kappa}=0$) approaches the straight line $\tth=\tfrac{9}{4}\kappa-\tfrac{3}{4}$ (orange).
Likewise, Curve~5—contained in the zero set $\mathsf{e}_{\mathrm{WB}}(\kappa,\tth)=0$—approaches the line $\tth=\tfrac{9}{4}\kappa-\tfrac{35}{4}$ (yellow).}

  \label{fig:asymptotic lines}
\end{figure}

    \item {\sc Relation with Djordjevic-Redekopp \cite{DjRe} and Ablowitz-Segur \cite{AbSe}.}
   Both Djordjevic-Redekopp and Ablowitz-Segur formally prove the linear instability of Stokes waves under long-wave perturbations via a modulational approximation with a Schr\"odinger equation. 
To make  a precise comparison with our result let us describe the  relation between our  functions  $\mathsf{e}_{\mathrm{WB}}(\kappa,\tth)$,  
$\mathsf{e}_{22}(\kappa,\tth)$ and the  coefficients obtained in \cite{DjRe,AbSe}.
In particular, denoting by $ \nu(\kappa, \tth)\equiv \nu$ and 
$\lambda (\kappa, \tth)\equiv \lambda$ the functions in   \cite[formula (2.17)]{DjRe}, we have 
\begin{equation}\label{id.DjRe}
\mathsf{e}_{\mathrm{WB}}(\kappa,\tth) =  \frac{1+\kappa}{2\ch^2}\nu(\kappa, \tth)   \ , \qquad \mathsf{e}_{22}(\kappa, \tth) = - 8 \lambda (\kappa, \tth) \ . 
\end{equation}
Note that  we put $g=1$ and the wavenumber $k =1$ in the formulas for $\nu(\kappa, \tth)$ and $\lambda (\kappa, \tth)$, in agreement with the choice in our paper. 
The identities \eqref{id.DjRe} are easily verified using Mathematica.
   In particular our  instability criterium in \eqref{def:cU}, namely 
   $\mathsf{e}_{\mathrm{WB}}(\kappa,\tth) \mathsf{e}_{22}(\kappa, \tth) > 0$, is equivalent to 
   $
   \nu(\kappa, \tth)\lambda (\kappa, \tth) < 0$ 
   which is the one described in \cite[Section 3]{DjRe}.

   In addition, our function $\mathsf{e}_{12}(\kappa, \tth)$ in \eqref{def:D} equals $2c_g$ in \cite[(2.7)]{DjRe}. 
The two singular curves described in \cite[Section 3]{DjRe}, namely $\kappa = \frac{\sigma^2}{3-\sigma^2}$, $\sigma = \ch^2$,  and $c_g^2 = \tth$, are respectively the singular  curves $\ch^4(1+\kappa) -3\kappa =0$ and 
   $
\textup{D}_{\tth,\kappa} = 0 $ where the denominator of $\mathsf{e}_{\mathrm{WB}}$  vanishes. 

    \item {\sc Relation with \cite{BMV3}.}
    Our result reduces to \cite{BMV3} at zero capillarity $\kappa = 0$. In particular the function 
 $\mathsf{e}_{\mathrm{WB}}(0,\tth)$ at zero-capillarity reduces to the function $\mathsf{e}_{\mathrm{WB}}$ of \cite[(1.1)]{BMV3}, hence it vanishes at the critical depth $\tth_{\mathrm{WB}} = 1.3627827\ldots$.

    \item {\sc Complete spectrum near $0$.}
    In Theorem \ref{thm:main} we have described just the two unstable eigenvalues of $\cL_{\mu,\e}$ close to zero for $ (\kappa, \tth) \in \mathcal{U}$. 
 There are also two larger 
purely imaginary eigenvalues of order $ \cO(\mu) $, see Theorem \ref{Complete BF thm}.  

\end{enumerate}

Before closing the introduction, we recall that  \cite{HY2} and \cite{SW}  prove that, 
provided the depth and capillarity are chosen in certain regions, 
a first isola of unstable spectrum is  present away from the origin.
Moreover, in \cite{SW} it is proved that the spectrum of the linearized operator is definitely purely imaginary outside a sufficiently large ball.
It would be interesting to combine such results with the analysis of isolae developed in \cite{BCMV} to give a complete description of the whole spectrum of the linearized gravity-capillary water waves operator.

Finally, we recall some related results: modulational instability under transversal perturbations has been studied in
\cite{CNS, CNS2,  JRSY,HTW};
high-frequencies instability of two-dimensional Stokes waves in the  pure gravity water waves was analyzed in \cite{HY, BMV4, BCMV}, together with numerical investigations in \cite{Mc1, Mc2, CD, BDS, CDT},  nonlinear modulational instability in \cite{ChenSu};
and various shallow-water or approximate water-wave models were also examined in \cite{HP, HJ, BHJ, BJ, GH, HK, BBHZ, BHR, MR}.

\section{The full Benjamin–Feir spectrum of gravity-capillary waves}
In this section we present the complete spectral Theorem \ref{Complete BF thm}.
First let us introduce the gravity-capillary water waves equations and the Stokes wave solutions.      \\
{\bf The water waves equations.} 
We consider the Euler equations of hydrodynamics for a 2-dimensional perfect and incompressible fluid  under the action of gravity and capillary forces at the free surface.
The fluid fills the region 
    \begin{align*}
       \mathcal{D}_\eta:=\{(x,y)\in \mathbb{R}\times \mathbb{R}: -\tth \leq y <\eta(t,x)\}\ ,
    \end{align*}
with finite depth. The irrotational velocity
field is the gradient of a harmonic scalar potential  $\Phi = \Phi(t,x,y)$ determined by  its trace $\psi(t,x):= \Phi(t, x,\eta(t,x))$ at the free surface $y=\eta(t,x)$. 
Actually $\Phi$ is
the unique solution of the elliptic equation $\Delta \Phi  = 0$ in $\mathcal{D}_\eta$  with Dirichlet datum
$\Phi(t, x,\eta(t, x))=\psi(t, x)$ and $\Phi_y(t,x,y) = 0$ at $y = - \tth$. 

 Imposing that the fluid particles
at the free surface remain on it along the evolution (kinematic boundary condition) and that the pressure of
the fluid plus the capillary forces at the free surface is equal to the constant atmospheric pressure (dynamic boundary condition), the time evolution of the fluid is determined by the non-local quasi-linear equations \cite{Zak68, CS}
\begin{equation} \label{Craig-Sulem formula}
\left\{\begin{aligned}
    \eta_t&=G(\eta)\psi\\
    \psi_t&=-g\eta-\frac{\psi_x^2}{2}+\frac{1}{2(1+\eta^2_x)}\left(G(\eta)\psi+\eta_x\psi_x\right)^2+\kappa\left(\frac{\eta_x}{(1+\eta_x^2)^{1/2}}\right)_x,
\end{aligned}\right.
\end{equation}
where $g>0$ is the gravity constant and $G(\eta):=G(\eta,\tth)$ denotes the Dirichlet-Neumann operator $[G(\eta)\psi](x):=\Phi_y(x,\eta(x))-\Phi_x(x,\eta(x))\eta_x(x)$. 
 In the sequel,
with no loss of generality, we set the gravity constant $g=1$.

System \eqref{Craig-Sulem formula} is Hamiltonian and can be written as
\begin{equation} \label{pa_t eta psi}
    \begin{aligned}
        \partial_t \begin{bmatrix}
\eta\\
\psi 
\end{bmatrix}=
\mathcal{J}\begin{bmatrix}
\nabla_\eta \mathcal{H}\\
\nabla_\psi \mathcal{H} 
\end{bmatrix}
    \end{aligned},~~\mathcal{J}:=
    \begin{bmatrix}
0& \mathrm{Id}\\
-\mathrm{Id} &0
\end{bmatrix},
\end{equation}
where $\nabla$ denote the $L^2$-gradient, and the Hamiltonian 
\begin{align} \label{mathcal H}
\mathcal{H}(\eta,\psi):=\int_{\mathbb{T}} \left(\frac{1}{2}\psi G(\eta)\psi+\frac{1}{2}\eta^2+\kappa(\sqrt{1+\eta^2_x}-1)\right) dx    
\end{align}
is the sum of the kinetic and potential energy of the fluid. 
In addition,  the water waves system \eqref{Craig-Sulem formula} is reversible with respect to the involution
\begin{equation} \label{rho involution}
    \begin{aligned}
        \rho \begin{bmatrix}
            \eta(x)\\
            \psi(x)
        \end{bmatrix}:=
        \begin{bmatrix}
            \eta(-x)\\
            -\psi(-x)
        \end{bmatrix},~~\mathrm{i.e.}~\mathcal{H}\circ \rho=\mathcal{H},
    \end{aligned}
\end{equation}
and it is space invariant.

 \paragraph{Stokes waves.} In a moving reference frame with constant speed $c$ (and with normalized gravity $g=1)$, the water waves system \eqref{Craig-Sulem formula} becomes
 \begin{equation} \label{in the reference frame}
\left\{\begin{aligned}
    \eta_t&=c\eta_x+G(\eta)\psi\\
    \psi_t&=c\psi_x-\eta-\frac{\psi_x^2}{2}+\frac{1}{2(1+\eta^2_x)}\left(G(\eta)\psi+\eta_x\psi_x\right)^2+\kappa\left(\frac{\eta_x}{(1+\eta_x^2)^{1/2}}\right)_x.
\end{aligned}\right.
\end{equation}
We consider small amplitude {\em Stokes waves solutions}, namely stationary solutions of \eqref{in the reference frame} which we further require to be  $2\pi$-periodic in space.
 The bifurcation of small-amplitude Stokes waves from the trivial solution was first studied for pure gravity water waves by Stokes \cite{stokes}, Levi-Civita \cite{LC}, Nekrasov \cite{Nek}, Struik \cite{Struik}. 
 In our setting, the existence and analyticity of  a bifurcating branch of traveling waves follows from the Crandall–Rabinowitz theorem. We denote $\mathbb{T}:=\mathbb{R}\setminus 2\pi\mathbb{Z}$ and $B(r):=\{x\in\mathbb{R}: |x|<r\}$ the open ball of radius $r$ centered at zero.
 \begin{theorem} [Stokes waves] \label{Thm: Stokes expansion} 
 Let $(\kappa, \tth) \in (\R_{\geq 0} \times \R_{>0}) \setminus \fR$, with  $\fR$  in \eqref{def:fR}. There 
 exist $\e_*:=\e_*(\kappa, \tth) >0$ and a unique family  of real analytic 
 solutions $(\eta_\e(x), \psi_\e(x), c_\e)$, parameterized by the amplitude $|\e| \leq \e_*$, of 
\begin{equation}\label{travelingWWstokes}
c \, \eta_x+G(\eta)\psi = 0 \, , \quad 
c \, \psi_x -  \eta - \dfrac{\psi_x^2}{2} + 
\dfrac{1}{2(1+\eta_x^2)} \big( G(\eta) \psi + \eta_x \psi_x \big)^2  +
\kappa\left(\frac{\eta_x}{(1+\eta_x^2)^{1/2}}\right)_x
= 0 \, , 
\end{equation}
  such that
 $ \eta_\e (x), \psi_\e (x) $ are $2\pi$-periodic;  $\eta_\e (x) $ is even
and $\psi_\e (x) $ is odd, of the form 
 \begin{equation}\label{exp:Sto}
 \begin{aligned}
  & \eta_\e (x) = \e \cos(x)+\e^2\left(\eta_2^{[0]}+\eta_2^{[2]}\cos(2x)\right) +\cO(\e^3)  , \\ 
  & \psi_\e (x)  =  \e \ch^{-1}(1+\kappa)^{\frac{1}{2}} \sin(x)+\e^2 \psi_2^{[2]}\sin(2x)+\cO(\e^3),  \\
  & c_\e = \chk  +\e^2 c_2+\cO(\e^3)\quad \text{where} \quad \chk := (1+\kappa)^{\frac{1}{2}}\ch= \sqrt{(1+\kappa)\tanh(\tth)} \, , 
   \end{aligned}
  \end{equation}
and 
  \begin{align}\label{expcoef}
 &\eta_{2}^{[0]} := \frac{1+\kappa}{4}\left(\ch^2-\ch^{-2}\right), \qquad 
\eta_{2}^{[2]} := -\frac{\left(\ch^4-3\right)\,\left(1+\kappa\right)}{4\,\ch^2\,\left( \ch^4(1+\kappa) -3\,\kappa\right) } , \qquad \\ \label{expcoef1}
&\psi_{2}^{[2]} :=  \frac{(1+\kappa)^{\frac{1}{2}}\,\left(9\,\kappa +3-6\,\ch^4\,\kappa +(1+\kappa)\ch^8\right)}{8\,\ch^3\,\left( \ch^4(1+\kappa) -3\,\kappa\right)}, \qquad 
\\  \label{expc2}
&c_2 := \frac{9\,\ch^4+
(6\,\ch^4+3\,\ch^8-12)(1+\kappa)+
(15-13\,\ch^4)\,{\left(1+\kappa\right)}^2 +
(-2\,\ch^{12} +10\,\ch^8-14\,\ch^4+ 6)\,{\left(1+\kappa\right)}^3}{16\,\ch^3\,(1+\kappa)^{\frac{1}{2}}\,\left(\ch^4(1+\kappa) -3\,\kappa\right)}.
\end{align}
More precisely for any  $ \sigma \geq  0 $ and $ {s > \frac72} $, there exists $ \e_*>0 $ such that
the map $\e \mapsto (\eta_\e, \psi_\e, c_\e)$ is analytic from $B(\e_*) \to H^{\sigma,s}_{\mathtt{ev}} (\T)\times H^{\sigma,s}_{\mathtt{odd}}(\T)\times \R$, where 
$ H^{\sigma,s}_{\mathtt{ev}}(\T) $, respectively $ H^{\sigma,s}_{\mathtt{odd}}(\T) $, denote the  space of even, respectively odd, 
 real valued $ 2 \pi $-periodic analytic functions
$ u(x) = \sum_{k \in \mathbb{Z}} u_k e^{\im k x} $
such that $ \| u \|_{\sigma,s}^2 := \sum_{k \in \mathbb{Z}} |u_k|^2 \langle k \rangle^{2s} 
e^{2 \sigma |k|} < + \infty$. 
\end{theorem}

The expansions in 
\eqref{exp:Sto}--\eqref{expc2} are proved in Appendix \ref{sec:App2}. 
The condition $(\kappa, \tth) \not \in \fR$ is used in Lemma \ref{lem:B0} to ensure that the kernel of the linearized operator at the flat surface is one-dimensional. 
Indeed $(\kappa, \tth) \not \in \fR$ 
is equivalent to ask that  the function
\begin{equation}\label{speedn}
n\mapsto \frac{\sqrt{(1+\kappa n^2) n \tanh(n\tth)}}{n} \ , 
\end{equation}
which is the quotient between the dispersion relation and the wave number, is injective on $\N$, which is enough for constructing Stokes waves with the fixed spatial period of $2\pi$.
The stronger celebrated Bond number condition
        \begin{equation}\label{Bond}
\frac{\kappa}{\tth^2} > \frac13
\end{equation}
guarantees that the function \eqref{speedn} is strictly monotone increasing on $(0, \infty)$. 
Remark that when $(\kappa, \tth) \in \fR$,  higher-order resonances happen. Nevertheless, traveling waves -- going under the name of Wilton ripples-- can be constructed  \cite{Wilton,reeder}.

\paragraph{Linearization.}
We linearize system
\eqref{in the reference frame} at the Stokes waves $(\eta_\epsilon(x),\psi_\epsilon(x))$ given in Theorem \ref{Thm: Stokes expansion} and evaluate $c$ at $c_\epsilon$. 
Using the shape derivative formula \cite{LD, LD book} 
$
\mathrm{d}_\eta G(\eta)[\hat \eta][\psi] = - G(\eta)(B\hat \eta) - \pa_x( V \hat \eta), 
$
where
the functions $(V(x),B(x))$ are the horizontal and vertical components of the velocity field $(\Phi_x,\Phi_y)$ at the free surface and are given by
\begin{align} \label{espV}
   V:= V(x)&:=-B(\eta_\epsilon)_x+(\psi_\epsilon)_x,\\ \label{espB}
  B:=  B(x)&:=\frac{G(\eta_\epsilon)\psi_\epsilon+(\psi_\epsilon)_x(\eta_\epsilon)_x}{1+(\eta_\epsilon)^2_x}=\frac{(\psi_\epsilon)_x-c_\epsilon}{1+(\eta_\epsilon)^2_x} (\eta_\epsilon)_x,
\end{align}
one obtains the real, autonomous linearized  system
\small
\begin{equation} \label{first linear eq}
    \begin{aligned}
        \begin{bmatrix}
            \hat{\eta}_t\\
            \hat{\psi}_t
        \end{bmatrix}=
        \left[\begin{array}{c|c} 
	 -G(\eta_\epsilon)B-\partial_x\circ(V-c_\epsilon) & G(\eta_\epsilon)  \\ 
	\hline 
	-1+B(V-c_\epsilon)\partial_x-B\partial_x\circ(V-c_\epsilon)-BG(\eta_\epsilon)\circ B+\kappa \partial_x \circ l \circ\partial_x & -(V-c_\epsilon)\partial_x+B G(\eta_\epsilon)
\end{array}\right] 
        \begin{bmatrix}
            \hat{\eta}\\
            \hat{\psi}
        \end{bmatrix},
    \end{aligned}
\end{equation}
\normalsize
where
\begin{align} \label{esp l}
   l:= l(x):=\frac{1}{(1+(\eta_\e)_x^2)^{3/2}}.
\end{align}
The map  $\e \to (V, B, l)$ is analytic as a map $B(\e_0) \to H^{\sigma, s-1}_{\mathtt{ev}}(\T) \times H^{\sigma, s-1}_{\mathtt{odd}}(\T) \times H^{\sigma, s-1}_{\mathtt{ev}}(\T) $. 
The real system \eqref{first linear eq} is Hamiltonian, i.e. of the form
$\cJ \cA$ with $\cA$ symmetric, $\cA = \cA^\top$, where the transpose is taken with respect to the real scalar product of $L^2(\T, \R) \times L^2(\T, \R)$.
Moreover the linear operator in \eqref{first linear eq} is reversible, namely it anti-commutes with the involution $\rho$ in \eqref{rho involution}. \\
Next, we conjugate \eqref{first linear eq} by using the time-independent ``good unknown of Alinhac'' linear transformation 
\begin{align}\label{alinhac}
    \begin{bmatrix}
        \hat{\eta}\\
        \hat{\psi}
    \end{bmatrix}:= Z
    \begin{bmatrix}
        u\\
        v
    \end{bmatrix},~~Z=\begin{bmatrix}
        1&0\\
        B&1
    \end{bmatrix},
    ~~Z^{-1}=\begin{bmatrix}
        1&0\\
        -B&1
    \end{bmatrix},
\end{align}
yielding the linear system 
\begin{align} \label{second linear eq}
    \begin{bmatrix}
        u_t\\
        v_t
    \end{bmatrix}=  \widetilde{\cL}_\e 
    \begin{bmatrix}
        u\\
        v
    \end{bmatrix} \ , \qquad 
    \widetilde{\cL}_\e := \begin{bmatrix}
        -\partial_x\circ (V-c_\epsilon)& G(\eta_\epsilon)\\
        -1-(V-c_\epsilon)B_x+\kappa \partial_x\circ l\circ \partial_x ~~& -(V-c_\epsilon)\partial_x
    \end{bmatrix} \ , 
\end{align}
which  is Hamiltonian and reversible since the transformation $Z$ is symplectic, $\mathrm{i.e.}$ $Z^T\mathcal{J} Z=\mathcal{J}$, and satisfies $Z\circ \rho=\rho\circ Z$.

Next, 
we perform a conformal change of variables to flatten 
the water surface. 
By \cite[Appendix A]{BBHM}, 
 there exists a diffeomorphism of $\mathbb{T}$,
 $ x\mapsto x+\mathfrak{p}(x)$, with a small $2\pi$-periodic  function $\mathfrak{p}(x)$, 
 and a small constant $\ttf $, such that, by defining the associated composition operator $ (\mathfrak{P}u)(x) := u(x+\mathfrak{p}(x))$, the Dirichlet-Neumann operator writes as \cite[Lemma A.5]{BBHM}
\begin{equation}\label{Gneta}
 G(\eta_\e) = \pa_x \circ \mathfrak{P}^{-1} \circ {\mathfrak H} \circ
 \tanh\big((\tth+\ttf)|D| \big)
 \circ \mathfrak{P} \, , 
\end{equation}
where $ {\mathfrak H} $ is the Hilbert transform, i.e. the  Fourier multiplier operator
$$
 \mathfrak{H}(e^{\im j x}):= - \im\, \textup{sign}(j) e^{\im j x} \, , 
 \quad  \forall j \in \Z \setminus \{0\} \, , 
 \quad \mathfrak{H}(1) := 0 \, . 
$$
The function $\mathfrak p(x)$ and the constant $\ttf $ are  determined as a fixed point  of 
(see  \cite[formula (A.15)]{BBHM})
\begin{equation} \label{def:ttf}
\mathfrak{p}  =  \frac{\mathfrak{H}}{\tanh \big((\tth + \ttf)|D| \big)}[\eta_\e ( x + \mathfrak{p}(x))] \, , 
 \qquad 
 \ttf:= \frac{1}{2\pi} \int_\T \eta_\e (x +  \mathfrak{p}(x)) \de x \, . 
  \end{equation}
  As proved in \cite{BMV3}, the map $\e \to (\mathfrak{p}, \ttf)$ is analytic as a map $B(\e_0) \to H^s_{\mathtt{odd}}(\T) \times \R$.
In addition, in  Appendix \ref{sec:App2} we prove the expansion
\begin{equation}
 \label{expfe}
 \begin{aligned}
&  \mathfrak p(x)  = \e \ch^{-2} \sin(x) +\e^2 \frac{\left(\ch^4+1\right)\,\left(\ch^4\,\kappa -3\,\kappa +\ch^4+3\right)}{8\,\ch^4\,\left(\ch^4\,(1+\kappa) -3\,\kappa\right)} \sin(2x)+\cO(\e^3) \, ,\\  
&   \ttf =
   \e^2 \frac{\ch^{4}(1+\kappa)-3-\kappa}{4\,\ch^2} +
   \cO(\e^3) \, .
   \end{aligned}
 \end{equation}
Under the symplectic and reversibility-preserving change of variables 
\begin{align}\label{LC}
    h=\mathcal{P}\begin{bmatrix}
        u\\
        v
\end{bmatrix},~~\mathcal{P}=\begin{bmatrix}
    (1+\mathfrak{p}_x)\mathfrak{P} & 0\\
    0 & \mathfrak{P}
\end{bmatrix} \ , 
\end{align}
one  transforms the system \eqref{second linear eq} into the linear system $h_t=\mathcal{L}_\epsilon h$ where $\mathcal{L}_\epsilon$ is the Hamiltonian and reversible real operator
\begin{equation} \label{mathcal L e}
\begin{aligned}
    \mathcal{L}_\epsilon:= \mathcal{L}_{\epsilon}(\kappa, \tth) :=\mathcal{P} \widetilde{\cL}_\e \mathcal{P}^{-1} = &\begin{bmatrix}
\partial_x\circ(\chk+p_\epsilon(x)) &  |D| \tanh((\tth + \ttf)|D|)\\
        -(1+a_\epsilon(x))+\kappa \Sigma_\e & ~~(\chk+p_\epsilon(x))\partial_x
    \end{bmatrix}\\
    =&\begin{bmatrix}
        0& \mathrm{Id}\\
        -\mathrm{Id}& 0
    \end{bmatrix}\begin{bmatrix}
 (1+a_\epsilon(x))-\kappa \Sigma_\e& ~~-(\chk+p_\epsilon(x))\partial_x\\
 \partial_x\circ(\chk+p_\epsilon(x))& ~~|D| \tanh((\tth + \ttf)|D|)
    \end{bmatrix},
\end{aligned}    
\end{equation}
where the functions $p_\e(x)$ and $a_\e(x)$ are given by
\begin{equation}\label{def:pa}
\chk+p_\e(x) :=  \displaystyle{\frac{ c_\e-V(x+\mathfrak{p}(x))}{ 1+\mathfrak{p}_x(x)}} \, , \quad 1+a_\e(x):=   \displaystyle{\frac{1+ (V(x + \mathfrak{p}(x)) - c_\e)
 B_x(x + \mathfrak{p}(x))  }{1+\mathfrak{p}_x(x)}} \,, 
\end{equation}
and the operator $\Sigma_\e$ is given by
\begin{equation} \label{def:Sigma g}
\begin{aligned}
  &   \Sigma_\e: =\frac{1}{1+\mathfrak{p}_x}\partial_x \circ g_\e(x)\circ \partial_x \circ \frac{1}{1+\mathfrak{p}_x},\\
  &  g_\e(x): = \frac{l(x+\mathfrak{p}(x))}{1+\mathfrak{p}_x}=\frac{1}{(1+((\eta_\epsilon)_x^2(x+\mathfrak{p}(x)))^{3/2}}\frac{1}{1+\mathfrak{p}_x}.
\end{aligned}    
\end{equation}
By the analyticity result of the map $\e \mapsto (V, B, l)$ given above,  the map
$\e \to (p_\e, a_\e, g_\e)$ is analytic as a map $B(\e_0) \to H^s_{\mathtt{ev}}(\T)\times H^s_{\mathtt{ev}}(\T) \times H^s_{\mathtt{ev}}(\T)$.
In Appendix \ref{sec:App2} we provide their Taylor expansions, that we collect in the following lemma:
\begin{lemma}\label{lem:pa.exp}
The analytic functions $p_\e (x) $ and $a_\e (x) $  in \eqref{def:pa} 
are even in $ x $, and
\begin{equation}\label{SN1}
p_\e (x)  
= \e p_1 (x) + \e^2 p_2 (x)  + \cO(\e^3) \, , \qquad   
a_\e (x)  
= \e a_1(x) +\e^2 a_2 (x) + \cO(\e^3) \, , 
\end{equation}
where
\begin{align}\label{pino1fd}
     p_1(x) & =   p_1^{[1]} \cos(x)\, , \qquad \quad p_1^{[1]} := -2(1+\kappa)^{\frac{1}{2}}\ch^{-1},\\
 \label{pino2fd}
    p_2(x) & =p_2^{[0]}+p_2^{[2]}\cos(2x),\\
   \label{pino2fd 1}
    & p_2^{[0]} := 
    \frac{9\,\ch^4+
(60-18\ch^4+3\ch^8)(1+\kappa) +
(-57+35\ch^4-8\ch^8){\left(1+\kappa \right)}^2
+(6-14\ch^4+10\ch^8-2\ch^{12}){\left(1+\kappa\right)}^3
}{   16(1+\kappa)^{\frac{1}{2}} \ch^3\left(\ch^4(1+\kappa) -3\kappa \right)},\\
\label{pino2fd 2}
   & p_2^{[2]} :=-\frac{(1+\kappa)^{\frac{1}{2}}\,\left(2\kappa (3 -\,\ch^4) +\ch^4+3\right)}{2\,\ch^3\,\left(\ch^4\,(1+\kappa) -3\,\kappa \right)},
\end{align}
and 
 \begin{align} 
a_1(x) \label{aino1fd} &= a_1^{[1]}\cos(x)\, , \qquad \qquad 
a_1^{[1]}:= -(\ch^{-2}+\ch^2\,\left(1+\kappa\right))\, , \\
a_2(x)& \label{aino2fd} = a_2^{[0]}+a_2^{[2]}\cos(2x)\, ,\quad  \, a_2^{[0]}:=\frac{3(1+\kappa)\ch^4+1}{2\ch^4}, \\
\label{a[2]2}
& a_2^{[2]} := \frac{(10\ch^8-30\ch^4)(1+\kappa)^2+(-\ch^8+22\ch^4-3)(1+\kappa)-6\ch^4}{4\,\ch^4\,\left(\ch^4\,(1+\kappa) -3\,\kappa \right)}.
\end{align}  
The function $g_\e(x)$ in \eqref{def:Sigma g} is even in $x$ and expands as
\begin{equation} \label{SN1 g}
    g_\e(x) = 1+ \e g_1(x)  + \e^2 g_2(x) + \cO(\e^3),
\end{equation}
where
\begin{align}\label{exp:g1}
    g_1(x) &= g_1^{[1]}\cos(x),\, \qquad \qquad g_1^{[1]}:=-\ch^{-2},\\
    \label{exp:g2}
    g_2(x) &=g_2^{[0]}+g_2^{[2]}\cos(2x),\, \qquad g_2^{[0]}:=\frac{2-3\ch^4}{4\ch^4}, \ \ g_2^{[2]}:=-\frac{3\,\kappa +5\,\ch^4\,\kappa -2\,\ch^8\,\kappa +2\,\ch^4-2\,\ch^8+3}{4\,\ch^4\,\left(\ch^4\,(1+\kappa) -3\,\kappa \right)}.
\end{align}
Finally the self-adjoint operator $\Sigma_\e$ in \eqref{def:Sigma g} expands as
\begin{equation} \label{Sigma epsilon}
    \Sigma_\e = \pa_{xx} + \e \Sigma_1 + \e^2 \Sigma_2 + \cO(\e^3), 
\end{equation}
where
\begin{align}\label{Sigma 1}
    \Sigma_j:=& d_j(x)\,\pa_{xx}+e_j(x)\pa_{x} + h_j(x), \quad j=1,2 
\end{align}
and
\begin{align}
\label{d1x}
    d_1(x)&=d_1^{[1]} \cos(x)\,, \quad d_1^{[1]} := -3\ch^{-2}\, ,  \\
    \label{d2x}
     e_1(x)& = e_1^{[1]} \sin(x) \,, \quad e_1^{[1]} := 3\ch^{-2}, \\ \label{h1 h11}
     h_1(x)& = h_1^{[1]}\cos(x) \,, \quad h_1^{[1]} := \ch^{-2},\\
    d_2(x)&= d_2^{[0]} + d_2^{[2]} \cos(2x)  \,, \quad 
    d_2^{[0]}:= \frac{3(4-\ch^4)}{4 \ch^4} \,, 
    \ \ 
    d_2^{[2]}:= -\frac{9\,\left(-\kappa \,\ch^4+3\,\kappa +1\right)}{4\,\ch^4 \left(\ch^4\,(1+\kappa) -3\,\kappa \right)}\,,
    \\
     e_2(x)& =  e_2^{[2]} \sin(2x)  \,,  \ \ 
     e_2^{[2]} := \frac{-9\,\kappa \,\ch^4+27\,\kappa +9}{2\,\ch^4\,\left(\ch^4\,(1+\kappa) -3\,\kappa \right)}\,,
     \\
     h_2(x)& = h_2^{[0]} + h_2^{[2]} \cos(2x) \,, \ \ 
     h_2^{[0]}:= -\frac{1}{2}\ch^{-4} \,, 
    \ \ \label{h2 h02}
    h_2^{[2]}:=\frac{15\,\kappa -11\,\ch^4\,\kappa +2\,\ch^8\,\kappa +\ch^4+2\,\ch^8+6}{2\,\ch^4\,\left(\ch^4\,(1+\kappa) -3\,\kappa \right)}.
\end{align}
\end{lemma}

\paragraph{Bloch-Floquet expansions.} 
Since the operator $\mathcal{L}_\e$ in \eqref{mathcal L e} has $2\pi$-periodic coefficients, Bloch-Floquet theory  \cite{RS78} implies that $\lambda\in\mathbb{C}$ belongs to the $L^2(\mathbb{R})-$spectrum of $\mathcal{L}_\e$ if and only if there exists a nontrivial Floquet–Bloch mode $\tilde{h}(x)=e^{\im \mu x} v(x)$, where $v$ is $2\pi-$periodic and $\mu\in[-\frac{1}{2},\frac{1}{2})$, such that $\lambda \tilde{h}=\mathcal{L}_\e \tilde{h}$, or equivalently, $\lambda v=e^{-\im \mu x}\mathcal{L}_\e e^{\im \mu x}v$. Therefore, we obtain
\begin{align*}
    \sigma_{L^2(\mathbb{R})}(\mathcal{L}_\e)=\bigcup_{\mu\in[-\frac{1}{2},\frac{1}{2})} \sigma_{L^2(\mathbb{T})}(\mathcal{L}_{\mu,\e})~~\mbox{where}~~~\mathcal{L}_{\mu,\e}:=e^{-\im\mu x}\mathcal{L}_\e e^{\im\mu x}.
\end{align*}
In particular, if $\lambda$ is an eigenvalue of $\mathcal{L}_{\mu,\e}$ on $L^2(\mathbb{T},\mathbb{C}^2)$ with eigenvector $v(x)$, then $h(t,x)=e^{\lambda t}e^{\im \mu x}v(x)$ solves $h_t=\mathcal{L}_\e h$. We remark that: \\
$(i)$ if $A=Op(a)$ is a pseudo-differential operator with symbol $a(x,\xi)$, which is $2\pi$-periodic in $x$, then $A_\mu:=e^{-\im\mu x}A e^{\im\mu x}=Op(a(x,\xi+\mu))$. \\
$(ii)$ If $A$ is a real operator then $\overline{A_\mu}=A_{-\mu}$. As a consequence the spectrum $\sigma(A_{-\mu})=\overline{\sigma(A_\mu)}$ and we can study $\sigma(A_\mu)$ just for $\mu>0$.\\
$(iii)$  $\sigma(A_\mu)$ is a $1$-periodic set with respect to $\mu$, so one can restrict to $\mu\in[0,\frac{1}{2})$.

By the previous remarks the Floquet operator associated with the real operator $\mathcal{L}_\e$ in \eqref{mathcal L e} is the complex Hamiltonian and reversible operator
\begin{equation} \label{mathcal L mu e}
\begin{aligned}
    \mathcal{L}_{\mu,\e}:=&\begin{bmatrix}
(\partial_x+\im\mu)\circ(\chk+p_\epsilon(x)) &  |D+\mu| \tanh((\tth + \ttf)|D+\mu|)\\
        -(1+a_\epsilon(x))+\kappa \Sigma_{\mu,\e} & ~~(\chk+p_\epsilon(x))(\partial_x+\im\mu)
    \end{bmatrix}\\
    =&\underbrace{\begin{bmatrix}
        0& \mathrm{Id}\\
        -\mathrm{Id}& 0
    \end{bmatrix}}_{=:\mathcal{J}}\underbrace{\begin{bmatrix}
 (1+a_\epsilon(x))-\kappa \Sigma_{\mu,\e}& ~~-(\chk+p_\epsilon(x))(\partial_x+\im\mu)\\
 (\partial_x+\im\mu)\circ(\chk+p_\epsilon(x))& ~~|D+\mu| \tanh((\tth + \ttf)|D+\mu|)
    \end{bmatrix}}_{=:\mathfrak{B}_{\mu,\e}},
\end{aligned}    
\end{equation}
where $\Sigma_{\mu,\e}:=e^{-\im\mu x}\Sigma_\e e^{\im\mu x}$ is given by the selfadjoint operator
\begin{align}\label{Sigma}
    \Sigma_{\mu,\e}=\frac{1}{1+\mathfrak{p}_x}(\partial_x + \im \mu) \circ g_\e(x)\circ (\partial_x+\im\mu) \circ \frac{1}{1+\mathfrak{p}_x} \ . 
\end{align}

We regard $\mathcal{L}_{\mu,\e}$ as an operator with domain $Y:=H^2(\mathbb{T},\mathbb{C})\times H^1(\mathbb{T},\mathbb{C})$ and range $X:=L^2(\mathbb{T},\mathbb{C})\times L^2(\mathbb{T},\mathbb{C})$, equipped with the complex scalar product
\begin{align} \label{complex product}
    (f,g):=\frac{1}{2\pi}\int_0^{2\pi} \left(f_1\overline{g_1}+f_2\overline{g_2} \right)\,\mathrm{d}x,~~\forall~f=\vet{f_1}{f_2},~~g=\vet{g_1}{g_2}\in L^2(\mathbb{T},\mathbb{C}^2).
\end{align}
We also denote $\|f\|^2=(f,f)$.

The complex operator $\mathcal{L}_{\mu,\e}$ in \eqref{mathcal L mu e} is complex Hamiltonian and reversible. Recall that if $\mathcal{L} : Y \to X$ is a complex linear operator, we say that it is 
\begin{itemize}
    \item \textbf{Complex Hamiltonian:} if there exists a self-adjoint operator, namely $\mathcal{B}=\mathcal{B}^*$, where $\mathcal{B}^*$ (with domain $Y$) is the adjoint with respect to the complex scalar product \eqref{complex product} such that $\mathcal{L} = \mathcal{J} \mathcal{B}$.

    \item \textbf{Reversible:} if 
    \begin{align} \label{reversible}
    \mathcal{L} \circ \overline{\rho} = -\overline{\rho} \circ \mathcal{L}, \quad \text{where} \quad 
    \overline{\rho} \begin{bmatrix} \eta(x) \\ \psi(x) \end{bmatrix} := 
    \begin{bmatrix} \overline{\eta}(-x) \\ -\overline{\psi}(-x) \end{bmatrix} \ . 
    \end{align}
\end{itemize}

The property \eqref{reversible} for $\mathcal{L}_{\mu,\e}$ follows because $\mathcal{L}_\e$ is a real operator which is reversible with respect to the involution $\rho$ in \eqref{rho involution}. Equivalently, since $\mathcal{J}\circ \overline{\rho}=-\overline{\rho}\circ \mathcal{J}$, the self-adjoint operator $\mathfrak{B}_{\mu,\e}$ is reversibility-preserving, $\mathrm{i.e.}$
\begin{align}\label{B rho=rho B}
    \mathfrak{B}_{\mu,\e}\circ\overline{\rho}=\overline{\rho}\circ \mathfrak{B}_{\mu,\e}.
\end{align}
In addition $(\mu,\e)\rightarrow \mathcal{L}_{\mu,\e}\in\mathcal{L}( Y,X)$ is analytic, since the functions $\e \mapsto a_\e, ~p_\e$ and $g_\e$ defined in \eqref{SN1}, \eqref{SN1 g} are analytic and $\mathcal{L}_{\mu,\e}$ is analytic with respect to $\mu$, since, for any $\mu\in[-\frac{1}{2},\frac{1}{2})$,
\begin{align} \label{DtanhD}
    |D+\mu|\tanh\Big((\tth+\ttf)|D+\mu|\Big)=(D+\mu)\tanh\Big((\tth+\ttf)(D+\mu)\Big).
\end{align}
Recall also the identity \cite[Section 5.1]{NS}
\begin{align} \label{D+mu}
    |D+\mu|=|D|+\mu(\mathrm{sgn}(D)+\Pi_0),~~{\forall \mu\in[0,\frac{1}{2})},
\end{align}
where $\mathrm{sgn}(D)$ is the Fourier multiplier operator, acting on $2\pi$-periodic functions, with symbol
\begin{equation} \label{sgn D}
\mathrm{sgn}(k):=1 \ \ \forall k >0 \ , \ \
\mathrm{sgn}(0):=0  , \ \ \ 
\mathrm{sgn}(k):= -1 \ \ \forall k < 0 \  , 
\end{equation}
and $\Pi_0$ is the projector operator on the zero mode, $\Pi_0 f(x):=\frac{1}{2\pi}\int_\mathbb{T} f(x) dx$.

\smallskip 

Our goal is to prove the existence of eigenvalues of $\mathcal{L}_{\mu,\e}$ in \eqref{mathcal L mu e} with non zero real part. We remark that the Hamiltonian structure of $\mathcal{L}_{\mu,\e}$ implies that eigenvalues with non zero real part may arise only from multiple eigenvalues of $\mathcal{L}_{\mu,0}$ (``Krein criterion''), because if $\lambda$ is an eigenvalue of $\mathcal{L}_{\mu,\e}$ then also $-\overline{\lambda}$ is, and the total algebraic multiplicity of the eigenvalues is conserved under small perturbation. We now describe the spectrum of $\mathcal{L}_{\mu,0}$.

\paragraph{The spectrum of $\mathcal{L}_{\mu,0}$.} The spectrum of the Fourier multiplier matrix operator 
\begin{equation} \label{mathcal L mu 0}
\begin{aligned}
\mathcal{L}_{\mu,0}:=&\begin{bmatrix}
\chk(\partial_x+\im\mu)&  |D+\mu| \tanh(\tth|D+\mu|)\\
        -1+\kappa (\pa_x + \im \mu)^2 & ~~\chk(\partial_x+\im\mu)
    \end{bmatrix}
\end{aligned}    
\end{equation}
consists of the purely imaginary eigenvalues $\{\lambda^{\pm}_k(\mu),~k\in\mathbb{Z}\}$, where
\begin{align} \label{eigenvalues of Lmu0}
    \lambda^{\pm}_k(\mu):=\im\left(\chk(\pm k+\mu)\mp \sqrt{\left(1+\kappa(k\pm\mu)^2\right)\, |k\pm \mu|\tanh(\tth| k\pm \mu|)}\right) \ . 
\end{align}
For $(\kappa, \tth) \not\in \mathfrak{R}$ (cf. \eqref{def:fR}), at $\mu=0$ the real operator $\mathcal{L}_{0,0}$ possesses the eigenvalue $0$ with algebraic multiplicity $4$, 
\begin{align*}
    \lambda^+_0(0)=\lambda^-_0(0)=\lambda^+_1(0)=\lambda^-_1(0)=0,
\end{align*}
and geometric multiplicity $3$. A real basis of the Kernel of $\mathcal{L}_{0,0}$ is 
\begin{align} \label{eigenfunc of mathcall L00}
    f^+_1:=\vet{(\ch/\sqrt{1+\kappa})^{1/2}\cos(x)}{(\ch/\sqrt{1+\kappa})^{-1/2}\sin(x)}, ~~f^-_1:=\vet{-(\ch/\sqrt{1+\kappa})^{1/2}\sin(x)}{(\ch/\sqrt{1+\kappa})^{-1/2}\cos(x)}, ~~f^-_0:=\vet{0}{1},
\end{align}
together with the generalized eigenvector
\begin{align} \label{eigenfunc f+0}
    f^+_0:=\vet{1}{0}, ~~\mathcal{L}_{0,0} f^+_0=-f^-_0.
\end{align}
Furthermore $0$ is an isolated eigenvalue for $\mathcal{L}_{0,0}$, namely the spectrum $\sigma(\mathcal{L}_{0,0})$ decomposes in two separated parts,
\begin{align} \label{sigma decomposition}
    \sigma(\mathcal{L}_{0,0})=\sigma'(\mathcal{L}_{0,0})\cup\sigma''(\mathcal{L}_{0,0}), ~~\mbox{where}~~\sigma'(\mathcal{L}_{0,0}):=\{0\},
\end{align}
and $\sigma''(\mathcal{L}_{0,0}):=\{\lambda^{\sigma}_k(0),~k\neq 0,1, ~\sigma=\pm\}$.

In addition, following the proof  in \cite[Lemma 4.1]{NS}, the operator $\mathcal{L}_{0,\e}$ possesses, for any sufficiently small $\e\neq 0$, the eigenvalue $0$ with a four dimensional generalized Kernel, spanned by $\e$-dependent vectors $U_1~,\Tilde{U}_2,~U_3,~U_4$ satisfying, for some real constant $\alpha_\e$, $\beta_\e$, 
\begin{equation} \label{237}
    \begin{aligned}
        &\mathcal{L}_{0,\e} U_1=0, ~\mathcal{L}_{0,\e}\tilde{U}_2=0,~\mathcal{L}_{0,\e} U_3=\alpha_\e \tilde{U}_2,\\
        &\mathcal{L}_{0,\e} U_4=-U_1-\beta_\e \tilde{U}_2,~~U_1=\vet{0}{1}.
    \end{aligned}
\end{equation}
By Kato's perturbation theory for any $\mu,\e\neq 0$ sufficiently small, the perturbed spectrum $\sigma(\mathcal{L}_{\mu,\e})$ admits a disjoint decomposition as 
\begin{align} \label{disjoint decomposition of spectrum}
    \sigma(\mathcal{L}_{\mu,\e})=\sigma'(\mathcal{L}_{\mu,\e})\cup \sigma''(\mathcal{L}_{\mu,\e}),
\end{align}
where $\sigma'(\mathcal{L}_{\mu,\e})$ consists of $4$ eigenvalues close to $0$. We denote by $\mathcal{V}_{\mu,\e}$ the spectral subspace associated with $\sigma'(\mathcal{L}_{\mu,\e})$, which has dimension $4$ and it is invariant by $\mathcal{L}_{\mu,\e}$. Our goal is to prove that, for $\e$ small, for values of the Floquet exponent $\mu$ in an interval of order $\e$, the $4\times 4$ matrix which represents the operator $\mathcal{L}_{\mu,\e}:\mathcal{V}_{\mu,\e}\rightarrow \mathcal{V}_{\mu,\e}$ possesses a pair of eigenvalues close to zero with opposite non zero real parts. 

Before stating our main result, let us introduce a notation that we shall use throughout the paper.

\begin{itemize}
\item
\textbf{Notation:} we denote by $\mathcal{O}(\mu^{m_1} \epsilon^{n_1}, \dots, \mu^{m_p} \epsilon^{n_p})$, $m_j, n_j \in \mathbb{N}$ (for us $\mathbb{N} := \{1,2, \dots \}$), analytic functions of $(\mu, \epsilon)$ with values in a Banach space $X$ which satisfy, for some $C > 0$ uniform for $\tth$ in any compact set of $(0,+\infty)$, the bound
\[
\|\mathcal{O}(\mu^{m_j} \epsilon^{n_j})\|_X \leq C \sum_{j=1}^{p} |\mu|^{m_j} |\epsilon|^{n_j}
\]
for small values of $(\mu, \epsilon)$. Similarly we denote $r_k(\mu^{m_1} \epsilon^{n_1}, \dots, \mu^{m_p} \epsilon^{n_p})$ scalar functions $\mathcal{O}(\mu^{m_1} \epsilon^{n_1}, \dots, \mu^{m_p} \epsilon^{n_p})$ which are also \textit{real} analytic.
\end{itemize}

Our main spectral result is the following one:
\begin{theorem}[Complete Benjamin-Feir spectrum] \label{Complete BF thm}
Let $(\kappa, \tth) \in (\R_{\geq 0} \times \R_{>0}) \setminus (\fR \cup \mathfrak{D})$, where $\fR$ and $\mathfrak{D}$ are defined in \eqref{def:fR} and \eqref{def:fD} respectively. There exist $\epsilon_{0}, \mu_{0} > 0$, such that, for any $0 < \mu < \mu_{0}$ and $0 \leq \epsilon < \epsilon_{0}$, the operator $\mathcal{L}_{\mu,\epsilon} : \mathcal{V}_{\mu,\epsilon} \to \mathcal{V}_{\mu,\epsilon}$ can be represented by a $4 \times 4$ matrix of the form
\begin{align} \label{U S diag}
\begin{pmatrix}
U & 0 \\
0 & S
\end{pmatrix},
\end{align}
where $U$ and $S$ are $2 \times 2$ matrices, with identical diagonal entries each, of the form
\begin{equation} \label{U S}
\begin{aligned}
U &= 
\begin{pmatrix}
\im \left( (\chk - \tfrac{1}{2} \mathsf{e}_{12}) \mu + r_{2}(\mu \epsilon^{2}, \mu^{2}\epsilon, \mu^{3}) \right) 
& -\mathsf{e}_{22}\tfrac{\mu}{8}(1 + r_{5}(\epsilon,\mu)) \\
- \mu \epsilon^2 \mathsf{e}_{\mathrm{WB}} + r_{1}'(\mu \epsilon^{3}, \mu^{2}\epsilon^2) + \mathsf{e}_{22}\tfrac{\mu^{3}}{8}(1 + r_{1}''(\epsilon,\mu)) 
& \im \left( (\chk - \tfrac{1}{2} \mathsf{e}_{12}) \mu + r_{2}(\mu \epsilon^{2}, \mu^{2}\epsilon, \mu^{3}) \right)
\end{pmatrix},\\
S &= 
\begin{pmatrix}
\im \chk \mu + \im r_{9}(\mu \epsilon^{2}, \mu^{2}\epsilon) & \tanh(\tth\mu) + r_{10}(\mu \epsilon) \\[1ex]
- \mu(1+\kappa\mu^2) + r_{8}(\mu \epsilon^{2}, \mu^{3} ) & \im \chk \mu + \im r_{9}(\mu \epsilon^{2}, \mu^{2}\epsilon)
\end{pmatrix},
\end{aligned}
\end{equation}
where $\mathsf{e}_{\mathrm{WB}}, \mathsf{e}_{12}, \mathsf{e}_{22}$ are defined in \eqref{def:eWB},
\eqref{def:D}, \eqref{def:e22}.
The eigenvalues of $U$ have the form
\begin{align} \label{lambda 1}
\lambda_{1}^{\pm}(\mu,\epsilon) 
= \im \frac{1}{2} \breve{\mathtt{c}}_{\tth,\kappa}\mu + \im r_{2}(\mu \epsilon^{2}, \mu^{2}\epsilon, \mu^{3}) 
\pm \tfrac{1}{8} \mu \sqrt{1+r_{5}(\epsilon,\mu)} \, \sqrt{\Delta_{\mathrm{BF}}(\kappa, \tth;\mu,\epsilon)},
\end{align}
where $\breve{\mathtt{c}}_{\tth,\kappa} := 2 \chk - \mathsf{e}_{12}(\kappa,\tth)$ and $\Delta_{\mathrm{BF}}(\kappa,\tth;\mu,\epsilon)$ is the Benjamin--Feir discriminant function in \eqref{BFDF}. 
The eigenvalues  in \eqref{lambda 1} have non-zero real part if and only if $\Delta_{\mathrm{BF}}(\kappa,\tth;\mu,\epsilon)>0$.

The eigenvalues of the matrix $S$ are a pair of purely imaginary eigenvalues of the form
\begin{align} \label{lambda zero}
\lambda_{0}^{\pm}(\mu,\epsilon) = \im \chk \mu (1 + r_{9}(\epsilon^{2}, \mu \epsilon)) \mp \im \sqrt{\mu \tanh(\tth \mu)} (1 +\kappa\mu^2+ r(\epsilon)).
\end{align}
For $\epsilon = 0$ the eigenvalues $\lambda_{1}^{\pm}(\mu,0), \lambda_{0}^{\pm}(\mu,0)$ coincide with those in \eqref{eigenvalues of Lmu0}.
\end{theorem}

\begin{remark}
At $\epsilon = 0$, the eigenvalues in \eqref{lambda 1} have the Taylor expansion
\[
\lambda_{1}^{\pm}(\mu,0) 
= \im \left( \chk - \tfrac{1}{2} \mathsf{e}_{12}(\kappa, \tth) \right) \mu 
\pm \im \tfrac{1}{8}|\mathsf{e}_{22}(\kappa, \tth)| \mu^{2} + \mathcal{O}(\mu^{3}),
\]
which coincides with the one of $\lambda_{1}^{\pm}(\mu)$ in \eqref{eigenvalues of Lmu0}, in view of the coefficients $\mathsf{e}_{12}(\kappa,\tth)$ and $\mathsf{e}_{22}(\kappa,\tth)$ defined in \eqref{def:D} and \eqref{def:e22} respectively.
\end{remark}

{We conclude this section by describing our approach in detail.}

\paragraph{Ideas and scheme of proof.} 
The proof of the result follows the procedure of  \cite{BMV1,BMV3}, and relies on  Kato's theory of similarity transformations and block decoupling. 
Via  Kato’s similarity transformations we   analytically continue the unperturbed symplectic basis 
of the generalized kernel of the linearized operator at the flat surface into a symplectic basis 
\(\{f_k^{\sigma}(\mu,\epsilon)\}\) of the perturbed spectral subspace $\mathcal{V}_{\mu,\e}$, depending analytically on \((\mu,\epsilon)\). 
The expansion of the basis \(\{f_k^{\sigma}(\mu,\epsilon)\}\) in \((\mu,\epsilon)\) is provided in 
Lemma~\ref{expansion of the basis F}. 
This allows us to represent the action of the operator \(\mathscr{L}_{\mu,\epsilon}\) 
on the basis \(\{f_k^{\sigma}(\mu,\epsilon)\}\) as 
a  \(4\times4\) Hamiltonian and reversible 
matrix $\mathsf{L}_{\mu,\epsilon}$ (Lemma~\ref{matrix representation mathsf L}) having the form  (Lemma \ref{B decomposition}) 
\begin{equation}\label{eq:L mu epsilon structure}
\mathsf{L}_{\mu,\epsilon}
= \mathsf{J}_4
\begin{pmatrix}
E & F \\[2pt]
F^{*} & G
\end{pmatrix}
=
\begin{pmatrix}
\mathsf{J}_2 E & \mathsf{J}_2 F \\[2pt]
\mathsf{J}_2 F^{*} & \mathsf{J}_2 G
\end{pmatrix},
\qquad
\mathsf{J}_4 =
\begin{pmatrix}
\mathsf{J}_2 & 0 \\[2pt]
0 & \mathsf{J}_2
\end{pmatrix},
\quad
\mathsf{J}_2 =
\begin{pmatrix}
0 & 1 \\[2pt]
-1 & 0
\end{pmatrix}
\end{equation}
where the $2\times2$ matrices $E = E^*,\, F,\, G = G^*$ admit the expansions
\eqref{E}--\eqref{G}.
In the gravity-capillary setting, this computation is substantially more delicate than in the pure gravity case of \cite{BMV3}, 
since the coefficients of the matrix depend not only on the depth parameter $\tth$ but also on the capillarity $\kappa$, and 
one must therefore track with precision the combined dependence of all matrix entries on $(\kappa,\tth)$.
In particular, the matrix $E$ reads
\begin{equation} \label{E first}
\begin{aligned} 
E &= 
\begin{pmatrix}
\mathsf{e}_{11} \epsilon^2 (1 + \hat{r}_1(\epsilon,\mu\epsilon)) - \mathsf{e}_{22} \frac{\mu^2}{8} (1 + \check{r}_1(\epsilon,\mu)) &
\im\left(\frac{1}{2} \mathsf{e}_{12} \mu + r_2(\mu \epsilon^2, \mu^2 \epsilon, \mu^3)\right) \\
- \im\left(\frac{1}{2} \mathsf{e}_{12} \mu + r_2(\mu \epsilon^2, \mu^2 \epsilon, \mu^3)\right) &
- \mathsf{e}_{22} \frac{\mu^2}{8}(1 + r_5(\epsilon, \mu))
\end{pmatrix},
\end{aligned}
\end{equation}
where the coefficients $\mathsf{e}_{11}(\kappa,\tth)$ and $\mathsf{e}_{22}(\kappa,\tth)$ are defined in \eqref{e11 f11} and \eqref{def:e22}.
In Lemma~\ref{e11 le 0} we verify that $\mathsf{e}_{11}$ never vanishes for $(\kappa,\tth)\notin \fR$, 
while the locus where $\mathsf{e}_{22}=0$ can be identified in Figure~\ref{fig:DR}. In fact, by analyzing the eigenvalues of the submatrix $\mathsf{J}_2 E$, 
we find that $\mathsf{J}_2 E$ possesses a pair of real nonzero eigenvalues 
if and only if $\mathsf{e}_{22}(\kappa,\tth)>0$, provided 
$0<\mu<\overline{\mu}(\epsilon)\sim\epsilon$. 
This is not the correct instability criteria in  \eqref{def:cU}, that instead requires the sign of the function  $\mathsf{e}_{\mathrm{WB}}$ in \eqref{def:eWB}.
This indicates that, in the gravity-capillary case, the correct eigenvalues of 
$\mathsf{L}_{\mu,\epsilon}$ are not a small perturbation of those of $
\begin{pmatrix}
\mathsf{J}_2 E & 0 \\[2pt]
0 & \mathsf{J}_2 G
\end{pmatrix}.
$
Indeed, such correct eigenvalues  will emerge only after one non-perturbative step of block diagonalization, that we perform in Section \ref{sec:BD}. 
Such step produces a correction to the (1,1) entry of the matrix $E$, modifying the coefficient
$\mathsf{e}_{11}$ into the coefficient $\mathsf{e}_{\mathrm{WB}}$, as proved in 
Lemma \ref{Step of block-decoupling}, see in particular formula \eqref{eWB}.
After this non-perturbative step, we exploit the implicit function theory to construct a transformation that conjugates $\mathsf{L}_{\mu,\e}$ into the block-diagonal matrix 
\eqref{U S diag} (Lemma \ref{lemma S2}).

\section{Perturbative Approach to the Separated Eigenvalues}

In this section, we analyze the splitting of the eigenvalues of $\mathcal{L}_{\mu, \epsilon}$ close to 0 for small values of $\mu$ and $\epsilon$, using Kato’s similarity transformation theory \cite[I-§4-6, II-§4]{Kato1966} and \cite{BMV1, BMV3, BMV_ed}. To this end, it is convenient to rewrite the operator $\mathcal{L}_{\mu, \epsilon}$ in \eqref{mathcal L mu e} as
\begin{equation} \label{mathcal L= ichmu+mathscr L}
\mathcal{L}_{\mu, \epsilon} = \im \chk \mu + \mathscr{L}_{\mu, \epsilon}, \quad \mu > 0,
\end{equation}
where, using also \eqref{D+mu}, $\mathscr{L}_{\mu, \epsilon}$ is the Hamiltonian operator
\begin{equation} \label{mathscr L mu e}
\mathscr{L}_{\mu, \epsilon} = \mathcal{J} \mathcal{B}_{\mu, \epsilon},
\end{equation}
with $\mathcal{B}_{\mu, \epsilon}$ the self-adjoint operator
\begin{equation} \label{mathcal B mu e}
\mathcal{B}_{\mu, \epsilon} := \begin{bmatrix} 1 + a_{\epsilon}(x)-\kappa\Sigma_{\mu,\e} & -(\chk + p_{\epsilon}(x)) \partial_x - \im \mu p_{\epsilon}(x) \\ \partial_x \circ(\chk + p_{\epsilon}(x)) + \im \mu p_{\epsilon}(x) & |D + \mu| \tanh ((\tth + \tf_\epsilon)|D + \mu|) \end{bmatrix} \ ,
\quad \Sigma_{\mu,\e} \mbox{ in } \eqref{Sigma} \ . 
\end{equation}
In addition  $\mathscr{L}_{\mu, \epsilon}$   is also complex-reversible, namely it satisfies, by \eqref{reversible},
\begin{equation} \label{mathscr rho=-rho mathscr}
\mathscr{L}_{\mu, \epsilon} \circ \bar{\rho} = - \bar{\rho} \circ \mathscr{L}_{\mu, \epsilon},
\end{equation}
whereas $\mathcal{B}_{\mu, \epsilon}$ is reversibility-preserving, i.e. fulfills \eqref{B rho=rho B}. Note also that $\mathcal{B}_{0, \epsilon}$ is a real operator.

The scalar operator $\im \chk \mu \equiv \im \chk \mu \,\text{Id}$ just translates the spectrum of $\mathcal{L}_{\mu, \epsilon}$ along the imaginary axis of the quantity $\im \chk \mu$, that is, in view of \eqref{mathcal L= ichmu+mathscr L},
\begin{equation}
\sigma(\mathcal{L}_{\mu, \epsilon}) = \im \chk \mu + \sigma(\mathscr{L}_{\mu, \epsilon}).
\end{equation}
Thus in the sequel we focus on studying the spectrum of $\mathscr{L}_{\mu, \epsilon}$.

Note also that $\mathscr{L}_{0, \epsilon} = \mathcal{L}_{0, \epsilon}$ for any $\epsilon \geq 0$. In particular $\mathscr{L}_{0,0}$ has zero as an isolated eigenvalue with algebraic multiplicity 4, geometric multiplicity 3 and generalized kernel spanned by the vectors $\{f_1^{+}, f_1^{-}, f_0^{+}, f_0^{-}\}$ in \eqref{eigenfunc of mathcall L00}, \eqref{eigenfunc f+0}; furthermore, its spectrum is separated as in \eqref{sigma decomposition}. For any $\epsilon \neq 0$ small, $\mathscr{L}_{0, \epsilon}$ has zero as an isolated eigenvalue with geometric multiplicity 2, and two generalized eigenvectors satisfying \eqref{237}.

We remark that, in view of \eqref{D+mu}, the operator $\mathscr{L}_{\mu, \epsilon}$ is analytic with respect to $\mu$. The operator $\mathscr{L}_{\mu, \epsilon}: Y \subset X \to X$ has domain $Y := H^2(\mathbb{T},\mathbb{C})\times H^1(\mathbb{T},\mathbb{C})$ and range $X := L^2(\mathbb{T},\mathbb{C})\times L^2(\mathbb{T},\mathbb{C})$.

\begin{lemma} \label{kato thm}
Let $\Gamma$ be a closed, counterclockwise-oriented curve around $0$ in the complex plane separating $\sigma' (\mathscr{L}_{0,0}) = \{0\}$ and the other part of the spectrum $\sigma'' (\mathscr{L}_{0,0})$ in \eqref{sigma decomposition}. There exist $\mu_0,\,\epsilon_0,  > 0$ such that for any $(\mu, \epsilon) \in B(\mu_0) \times B(\epsilon_0)$ the following statements hold:

\begin{enumerate}
    \item \textit{The curve $\Gamma$ belongs to the resolvent set of the operator $\mathscr{L}_{\mu,\epsilon} : Y \subset X \to X$ defined in \eqref{mathscr L mu e}.}
    \item \textit{The operators}
    \begin{equation} \label{Projection P mu e}
        P_{\mu,\epsilon} := - \frac{1}{2\pi \im} \oint_\Gamma (\mathscr{L}_{\mu,\epsilon} - \lambda)^{-1} \ d\lambda : X \to Y
    \end{equation}
    \textit{are well-defined projectors commuting with $\mathscr{L}_{\mu,\epsilon}$, i.e., $P_{\mu,\epsilon}^2 = P_{\mu,\epsilon}$ and $P_{\mu,\epsilon} \mathscr{L}_{\mu,\epsilon} = \mathscr{L}_{\mu,\epsilon} P_{\mu,\epsilon}$. The map $(\mu, \epsilon) \mapsto P_{\mu,\epsilon}$ is analytic from $B(\mu_0) \times B(\epsilon_0)$ to $\mathcal{L}(X,Y)$.}
    \item \textit{The domain $Y$ of the operator $\mathscr{L}_{\mu,\epsilon}$ decomposes as the direct sum}
    \begin{equation} \label{Y=V+ker P}
        Y = \mathcal{V}_{\mu,\epsilon} \oplus \ker(P_{\mu,\epsilon}), \quad \mathcal{V}_{\mu,\epsilon} := \operatorname{Rg}(P_{\mu,\epsilon}) = \ker(\operatorname{Id} - P_{\mu,\epsilon}),
    \end{equation}
    \textit{of closed invariant subspaces, namely $\mathscr{L}_{\mu,\epsilon} : \mathcal{V}_{\mu,\epsilon} \to \mathcal{V}_{\mu,\epsilon}$, $\mathscr{L}_{\mu,\epsilon} : \ker(P_{\mu,\epsilon}) \to \ker(P_{\mu,\epsilon})$. Moreover}
    \begin{equation} \label{spectrum separated by Gamma}
    \begin{aligned}
        \sigma(\mathscr{L}_{\mu,\epsilon}) \cap \{z \in \mathbb{C} \text{ inside } \Gamma\} &= \sigma(\mathscr{L}_{\mu,\epsilon} |_{\mathcal{V}_{\mu,\epsilon}}) = \sigma'(\mathscr{L}_{\mu,\epsilon}), \\
        \sigma(\mathscr{L}_{\mu,\epsilon}) \cap \{z \in \mathbb{C} \text{ outside } \Gamma\} &= \sigma(\mathscr{L}_{\mu,\epsilon} |_{\ker(P_{\mu,\epsilon})}) = \sigma''(\mathscr{L}_{\mu,\epsilon}).
    \end{aligned}
    \end{equation}
    \item \textit{The projectors $P_{\mu,\epsilon}$ are similar to each other; the transformation operators}
    \begin{equation} \label{U transformation operators}
        U_{\mu,\epsilon} := (\operatorname{Id} - (P_{\mu,\epsilon} - P_{0,0})^2)^{-1/2} \big[P_{\mu,\epsilon} P_{0,0} + (\operatorname{Id} - P_{\mu,\epsilon})(\operatorname{Id} - P_{0,0}) \big]
    \end{equation}
    \textit{are bounded and invertible in $Y$ and in $X$, with inverse}
    \begin{equation} \label{U inverse}
        U_{\mu,\epsilon}^{-1} = \big[P_{0,0} P_{\mu,\epsilon} + (\operatorname{Id} - P_{0,0})(\operatorname{Id} - P_{\mu,\epsilon})\big](\operatorname{Id} - (P_{\mu,\epsilon} - P_{0,0})^2)^{-1/2},
    \end{equation}
    \textit{and $U_{\mu,\epsilon} P_{0,0} U_{\mu,\epsilon}^{-1} = P_{\mu,\epsilon}$ as well as $U_{\mu,\epsilon}^{-1} P_{\mu,\epsilon} U_{\mu,\epsilon} = P_{0,0}$. The map $(\mu, \epsilon) \mapsto U_{\mu,\epsilon}$ is analytic from $B(\mu_0) \times B(\epsilon_0)$ to $\mathcal{L}(Y)$.}
    \item \textit{The subspaces $\mathcal{V}_{\mu,\epsilon} = \operatorname{Rg}(P_{\mu,\epsilon})$ are isomorphic to each other: $\mathcal{V}_{\mu,\epsilon} = U_{\mu,\epsilon} \mathcal{V}_{0,0}$. In particular $\dim \mathcal{V}_{\mu,\epsilon} = \dim \mathcal{V}_{0,0} = 4$, for any $(\mu, \epsilon) \in B(\mu_0) \times B(\epsilon_0)$.}
\end{enumerate}
\end{lemma}

The proof of Lemma \ref{kato thm} is similar to the one of \cite[Lemma 3.1]{BMV1} and we skip it.  Recalling \eqref{mathscr L mu e}-\eqref{mathscr rho=-rho mathscr}, the Hamiltonian and reversible nature of the operator $\mathscr{L}_{\mu,\e}$ imply additional algebraic properties for spectral projectors $P_{\mu,\e}$ and the transformation operators $U_{\mu,\e}$ as follows.  

\begin{lemma} \label{properties of U and P}
For any $(\mu, \epsilon) \in B(\mu_0) \times B(\epsilon_0)$, the following holds true:

\begin{itemize}
    \item[(i)] The projectors $P_{\mu,\epsilon}$ defined in \eqref{Projection P mu e} are skew-Hamiltonian, namely $\mathcal{J} P_{\mu,\epsilon} = P_{\mu,\epsilon}^* \mathcal{J}$, and reversibility preserving, i.e. $\bar{\rho} P_{\mu,\epsilon} = P_{\mu,\epsilon} \bar{\rho}$.
    
    \item[(ii)] The transformation operators $U_{\mu,\epsilon}$ in \eqref{U transformation operators} are symplectic, namely $U_{\mu,\epsilon}^* \mathcal{J} U_{\mu,\epsilon} = \mathcal{J}$, and reversibility preserving.
    
    \item[(iii)] $P_{0,\epsilon}$ and $U_{0,\epsilon}$ are real operators, i.e. $\bar{P}_{0,\epsilon} = P_{0,\epsilon}$ and $\bar{U}_{0,\epsilon} = U_{0,\epsilon}$.
\end{itemize}
    
\end{lemma} 

See \cite[Lemma 3.2]{BMV1} for details. By the previous lemma, the linear involution $\bar{\rho}$ commutes with the spectral projectors $P_{\mu,\epsilon}$ and then $\bar{\rho}$ leaves invariant the subspace $\mathcal{V}_{\mu,\epsilon} = \mathrm{Rg}(P_{\mu,\epsilon})$.

\paragraph{Symplectic and reversible basis of $\mathcal{V}_{\mu,\epsilon}$.} It is convenient to represent the Hamiltonian and reversible operator $\mathscr{L}_{\mu,\epsilon}: \mathcal{V}_{\mu,\epsilon} \to \mathcal{V}_{\mu,\epsilon}$ in a basis which is symplectic and reversible, according to the following definition:

\begin{definition}[Symplectic and reversible basis]  \label{Symplectic and reversible basis}
    A basis $\mathsf{F} := \{ \mathsf{f}_1^+, \mathsf{f}_1^-, \mathsf{f}_0^+, \mathsf{f}_0^- \}$ of $\mathcal{V}_{\mu,\epsilon}$ is \textit{symplectic} if, for any $k, k' = 0,1$,
\begin{equation} \label{basis is symplectic}
    \begin{aligned}
    (\mathcal{J} \mathsf{f}_k^\mp, \mathsf{f}_k^\pm) = \pm 1, ~~(\mathcal{J} \mathsf{f}_k^\sigma, \mathsf{f}_k^\sigma) = 0, \quad \forall \,\sigma = \pm; \\
    \text{if } k \neq k', \text{ then } (\mathcal{J} \mathsf{f}_k^\sigma, \mathsf{f}_{k'}^{\sigma'}) = 0, \quad \forall\, \sigma, \sigma' = \pm.
\end{aligned}
\end{equation}   

This is \textit{reversible} if
\begin{equation} \label{basis is reversible}
    \begin{aligned}
    \bar{\rho} \mathsf{f}_1^+ &= \mathsf{f}_1^+,  ~\bar{\rho} \mathsf{f}_1^- = -\mathsf{f}_1^-,~\bar{\rho} \mathsf{f}_0^+ = \mathsf{f}_0^+,  ~\bar{\rho} \mathsf{f}_0^- = -\mathsf{f}_0^-, \\
    &\text{i.e. } \bar{\rho} \mathsf{f}_k^\sigma = \sigma \mathsf{f}_k^\sigma, \quad \forall\, \sigma = \pm, k = 0,1.
\end{aligned}
\end{equation}
\end{definition} 

We use the following notation along the paper: we denote by $\mathrm{even}(x)$ a real $2\pi$-periodic function which is even in $x$, and by $\mathrm{odd}(x)$ a real $2\pi$-periodic function which is odd in $x$.

\begin{remark}[Parity structure of the reversible basis \eqref{basis is reversible}]
The elements of a reversible basis \( \mathsf{F}=\{\mathsf{f}^+_1,\mathsf{f}^-_1,\mathsf{f}^+_0,\mathsf{f}^-_0\} \) enjoys specific parity properties. Specifically,
\begin{align} \label{Parity structure}
\mathsf{f}_k^+(x) = 
\begin{bmatrix}
\mathit{even}(x) + \im\,\mathit{odd}(x) \\
\mathit{odd}(x) + \im\,\mathit{even}(x)
\end{bmatrix}, \quad
\mathsf{f}_k^-(x) = 
\begin{bmatrix}
\mathit{odd}(x) + \im\,\mathit{even}(x) \\
\mathit{even}(x) + \im\,\mathit{odd}(x)
\end{bmatrix}.
\end{align}
This structure follows from the reversibility of the problem, specifically from the involution \( \overline{\rho} \) defined in equation \eqref{reversible}, which implies that the real and imaginary parts of each component satisfy definite parity conditions.
\end{remark}

\begin{remark}[Symplectic expansion using the basis in \eqref{basis is symplectic}]
We can express any vector \( \mathsf{f} \in \mathcal{V}_{\mu,\epsilon} \) as a linear combination of the symplectic basis:
\begin{align} \label{remark f expansion}
    \mathsf{f} = \alpha_1^+ \mathsf{f}_1^+ + \alpha_1^- \mathsf{f}_1^- + \alpha_0^+ \mathsf{f}_0^+ + \alpha_0^- \mathsf{f}_0^-,
\end{align}
for suitable coefficients \( \alpha_k^\sigma \in \mathbb{C} \). These coefficients are computed by applying the symplectic form \( \mathcal{J} \), taking \( L^2 \)-scalar products with the basis elements, and using the symplecticity \eqref{basis is symplectic}. Therefore, we may rewrite \eqref{remark f expansion} as
\begin{align} \label{expasion of f by using symplectic basis}
    \mathsf{f}=-(\mathcal{J}\mathsf{f},\mathsf{f}^-_1)\mathsf{f}^+_1+(\mathcal{J}\mathsf{f},\mathsf{f}^+_1)\mathsf{f}^-_1-(\mathcal{J}\mathsf{f},\mathsf{f}^-_0)\mathsf{f}^+_0+(\mathcal{J}\mathsf{f},\mathsf{f}^+_0)\mathsf{f}^-_0.
\end{align}
\end{remark}

We now compute the matrix representation of $\mathscr{L}_{\mu,\e}$ in a symplectic and reversible basis of $\mathcal{V}_{\mu,\e}$.

\begin{lemma} \label{matrix representation mathsf L}
The $4 \times 4$ matrix that represents the Hamiltonian and reversible operator $\mathscr{L}_{\mu,\epsilon} = \mathcal{J} \mathcal{B}_{\mu,\epsilon} : \mathcal{V}_{\mu,\epsilon} \to \mathcal{V}_{\mu,\epsilon}$ with respect to a symplectic and reversible basis $\mathsf{F} = \{ \mathsf{f}_1^{+}, \mathsf{f}_1^{-}, \mathsf{f}_0^{+}, \mathsf{f}_0^{-} \}$ of $\mathcal{V}_{\mu,\epsilon}$ is
\begin{align} \label{J4Bmuepsilon}
  \mathsf{L}_{\mu,\epsilon}:=  \mathsf{J}_4 \mathsf{B}_{\mu,\epsilon}, \quad \mathsf{J}_4 := 
\left(
\begin{array}{c|c}
\mathsf{J}_2 & 0 \\
\hline
0 & \mathsf{J}_2
\end{array}
\right), \quad
    \mathsf{J}_2 := 
    \begin{pmatrix}
        0 & 1 \\
        -1 & 0
    \end{pmatrix}, \quad \text{where} \quad \mathsf{B}_{\mu,\epsilon} = \mathsf{B}_{\mu,\epsilon}^{*}.
\end{align}
The self-adjoint matrix
\begin{align} \label{Bmuepsilon matrix representation}
    \mathsf{B}_{\mu,\epsilon} = \left( 
    \begin{array}{cccc}
        (\mathcal{B}_{\mu,\epsilon} \mathsf{f}_1^{+}, \mathsf{f}_1^{+}) & (\mathcal{B}_{\mu,\epsilon} \mathsf{f}_1^{-}, \mathsf{f}_1^{+}) & (\mathcal{B}_{\mu,\epsilon} \mathsf{f}_0^{+}, \mathsf{f}_1^{+}) & (\mathcal{B}_{\mu,\epsilon} \mathsf{f}_0^{-}, \mathsf{f}_1^{+}) \\
        (\mathcal{B}_{\mu,\epsilon} \mathsf{f}_1^{+}, \mathsf{f}_1^{-}) & (\mathcal{B}_{\mu,\epsilon} \mathsf{f}_1^{-}, \mathsf{f}_1^{-}) & (\mathcal{B}_{\mu,\epsilon} \mathsf{f}_0^{+}, \mathsf{f}_1^{-}) & (\mathcal{B}_{\mu,\epsilon} \mathsf{f}_0^{-}, \mathsf{f}_1^{-}) \\
        (\mathcal{B}_{\mu,\epsilon} \mathsf{f}_1^{+}, \mathsf{f}_0^{+}) & (\mathcal{B}_{\mu,\epsilon} \mathsf{f}_1^{-}, \mathsf{f}_0^{+}) & (\mathcal{B}_{\mu,\epsilon} \mathsf{f}_0^{+}, \mathsf{f}_0^{+}) & (\mathcal{B}_{\mu,\epsilon} \mathsf{f}_0^{-}, \mathsf{f}_0^{+}) \\
        (\mathcal{B}_{\mu,\epsilon} \mathsf{f}_1^{+}, \mathsf{f}_0^{-}) & (\mathcal{B}_{\mu,\epsilon} \mathsf{f}_1^{-}, \mathsf{f}_0^{-}) & (\mathcal{B}_{\mu,\epsilon} \mathsf{f}_0^{+}, \mathsf{f}_0^{-}) & (\mathcal{B}_{\mu,\epsilon} \mathsf{f}_0^{-}, \mathsf{f}_0^{-})
    \end{array}
    \right).
\end{align}
The entries of the matrix $\mathsf{B}_{\mu,\epsilon}$ are alternatively real or purely imaginary: for any $\sigma = \pm, k = 0, 1$,
\begin{align} \label{B are alternatively real or imaginary}
    (\mathcal{B}_{\mu,\epsilon} \mathsf{f}_k^{\sigma}, \mathsf{f}_{k'}^{\sigma}) \text{ is real}, \quad
    (\mathcal{B}_{\mu,\epsilon} \mathsf{f}_k^{\sigma}, \mathsf{f}_{k'}^{-\sigma}) \text{ is purely imaginary}.
\end{align}  
\end{lemma} 
\begin{proof}
    Recalling \eqref{mathscr L mu e} and \eqref{expasion of f by using symplectic basis}, for $\sigma=\pm$, $k=0,1$, we obtain
    \begin{equation*} 
        \begin{aligned}
            \mathscr{L}_{\mu,\epsilon}\mathsf{f}^{\sigma}_k&=-(\mathcal{J}\mathscr{L}_{\mu,\epsilon}\mathsf{f}^{\sigma}_k,\mathsf{f}^-_1)\mathsf{f}^+_1+(\mathcal{J}\mathscr{L}_{\mu,\epsilon}\mathsf{f}^{\sigma}_k,\mathsf{f}^+_1)\mathsf{f}^-_1-(\mathcal{J}\mathscr{L}_{\mu,\epsilon}\mathsf{f}^{\sigma}_k,\mathsf{f}^-_0)\mathsf{f}^+_0+(\mathcal{J}\mathscr{L}_{\mu,\epsilon}\mathsf{f}^{\sigma}_k,\mathsf{f}^+_0)\mathsf{f}^-_0\\
            &=(\mathcal{B}_{\mu,\epsilon}\mathsf{f}^{\sigma}_k,\mathsf{f}^-_1)\mathsf{f}^+_1-(\mathcal{B}_{\mu,\epsilon}\mathsf{f}^{\sigma}_k,\mathsf{f}^+_1)\mathsf{f}^-_1+(\mathcal{B}_{\mu,\epsilon}\mathsf{f}^{\sigma}_k,\mathsf{f}^-_0)\mathsf{f}^+_0-(\mathcal{B}_{\mu,\epsilon}\mathsf{f}^{\sigma}_k,\mathsf{f}^+_0)\mathsf{f}^-_0.
        \end{aligned}
    \end{equation*}
This verifies that the matrix representation of $\mathscr{L}_{\mu,\epsilon}$ with respect to $\mathsf{F}$ is $\mathsf{J}_4\mathsf{B}_{\mu,\epsilon}$. Also, the martix $\mathsf{B}_{\mu,\e}$ is self-adjoing because $\mathcal{B}_{\mu,\e}$ is self-adjoint. Next, recall from \eqref{complex product} and \eqref{reversible} that the inner product satisfies
\begin{align} \label{(f,g)=(rho f, rho g)}
    (\mathsf{f},\mathsf{g})=\overline{(\overline{\rho}\mathsf{f},\overline{\rho}\mathsf{g})}.
\end{align}
Then, since $\mathcal{B}_{\mu,\e}$ is both self-adjoint and reversibility-preserving \eqref{B rho=rho B} and \eqref{basis is reversible}, we compute:
\begin{align*}
    (\mathcal{B}_{\mu,\epsilon} \mathsf{f}_k^{\sigma}, \mathsf{f}_{k'}^{\sigma'})=\overline{(\overline{\rho}\mathcal{B}_{\mu,\e} \mathsf{f}^{\sigma}_k,\overline{\rho} \mathsf{f}^{\sigma'}_{k'})}=\overline{(\mathcal{B}_{\mu,\e}\overline{\rho} \mathsf{f}^{\sigma}_k,\overline{\rho} \mathsf{f}^{\sigma'}_{k'})}=\sigma\sigma'\overline{(\mathcal{B}_{\mu,\e}\mathsf{f}^{\sigma}_k,\mathsf{f}^{\sigma'}_{k'})},
\end{align*}
which proves \eqref{B are alternatively real or imaginary}.
\end{proof}

We conclude this section recalling some notation. A $2n \times 2n$, $n = 1, 2$, matrix of the form $\mathsf{L} = \mathsf{J}_{2n} \mathsf{B}$ is \textit{Hamiltonian} if $\mathsf{B}$ is a self-adjoint matrix, i.e. $\mathsf{B} = \mathsf{B}^*$. It is \textit{reversible} if $\mathsf{B}$ is reversibility-preserving, i.e.
\[
\rho_{2n} \circ \mathsf{B} = \mathsf{B} \circ \rho_{2n},
\]
where
\[
\rho_4 := \begin{pmatrix} \rho_2 & 0 \\ 0 & \rho_2 \end{pmatrix}, \quad
\rho_2 := \begin{pmatrix} \mathfrak{c} & 0 \\ 0 & -\mathfrak{c} \end{pmatrix}, 
\]
and $\mathfrak{c} : z \mapsto \bar{z}$ is the conjugation of the complex plane.
Equivalently, $\rho_{2n} \circ \mathsf{L} = - \mathsf{L} \circ \rho_{2n}.$

The transformations preserving the Hamiltonian structure are called \textit{symplectic}, and satisfy
\begin{align} \label{YstarJ4Y=J4}
Y^* \mathsf{J}_4 Y = \mathsf{J}_4.
\end{align}
If $Y$ is symplectic then $Y^*$ and $Y^{-1}$ are symplectic as well. A Hamiltonian matrix $\mathsf{L} = \mathsf{J}_4 \mathsf{B}$, with $\mathsf{B} = \mathsf{B}^*$, is conjugated through a symplectic matrix $Y$ in a new Hamiltonian matrix
\begin{align} \label{L1=J4B}
    \mathsf{L}_1=Y^{-1}\mathsf{L}Y=Y^{-1}\mathsf{J}_4(Y^-)^*Y^*\mathsf{B}Y=\mathsf{J}_4 \mathsf{B}_1\,\,\textit{where}\,\,\mathsf{B}_1:=Y^*\mathsf{B}Y=\mathsf{B}^*_1.
\end{align}
A $4\times 4$ matrix $\mathsf{B} = (\mathsf{B}_{ij})_{i,j=1,\ldots,4}$ is reversibility-preserving if and only if its entries are alternatively real and purely imaginary, namely $\mathsf{B}_{ij}$ is real when $i + j$ is even and purely imaginary otherwise, as in \eqref{B are alternatively real or imaginary}. A $4\times 4$ complex matrix $\mathsf{L} = (\mathsf{L}_{ij})_{i,j=1,\ldots,4}$ is reversible if and only if $\mathsf{L}_{ij}$ is purely imaginary when $i + j$ is even and real otherwise.

We finally mention that the flow of a Hamiltonian reversibility-preserving matrix is symplectic and reversibility-preserving (see \cite[Lemma 3.8]{BMV1}).

\section{Matrix Representation of $\mathscr{L}_{\mu,\e}$ on $\mathcal{V}_{\mu,\e}$} Using the transformation operator $U_{\mu,\e}$ in \eqref{U transformation operators}, we construct the basis of $\mathcal{V}_{\mu,\e}$
\begin{equation} \label{F basis set and f}
    \begin{aligned}
        \mathcal{F}&:=\{f^+_1(\mu,\e),f^-_1(\mu,\e),f^+_0(\mu,\e),f^-_0(\mu,\e)\}, \qquad 
        f^\sigma_k(\mu,\e):=U_{\mu,\e} f^{\sigma}_k,~~\sigma=\pm,~k=0,1,
    \end{aligned}
\end{equation}
where 
\begin{align} \label{eigenfunc of mathcall L00 2}
    f^+_1=\vet{(\ch/\sqrt{1+\kappa})^{1/2}\cos(x)}{(\ch/\sqrt{1+\kappa})^{-1/2}\sin(x)}, ~~f^-_1=\vet{-(\ch/\sqrt{1+\kappa})^{1/2}\sin(x)}{(\ch/\sqrt{1+\kappa})^{-1/2}\cos(x)}, ~~f^+_0=\vet{1}{0},~~f^-_0=\vet{0}{1},
\end{align}
form a basis of $\mathcal{V}_{0,0}=\mathrm{Rg}(P_{0,0})$, cf. \eqref{eigenfunc of mathcall L00}-\eqref{eigenfunc f+0}. Note that the real valued vectors $\{f^\pm_1,f^\pm_0\}$ form a symplectic and reversible basis for $\mathcal{V}_{0,0}$, according to Definition \ref{Symplectic and reversible basis}. Then, by Lemma \ref{kato thm} and Lemma \ref{properties of U and P} we deduce that:
\begin{lemma} \label{F is symplectic and reversible}
    The basis $\mathcal{F}$ of $\mathcal{V}_{\mu,\e}$ defined in \eqref{F basis set and f}, is symplectic and reversible, i.e. satisfies \eqref{basis is symplectic} and \eqref{basis is reversible}. Each map $(\mu,\e) \mapsto f^\sigma_k(\mu,\e)$ is analytic as a map $B(\mu_0)\times B(\mu_0)\rightarrow H^2(\mathbb{T})\times H^1(\mathbb{T})$.
\end{lemma}
\begin{proof}
By Lemma \ref{properties of U and P}-(ii), the operators $U_{\mu,\epsilon}$ are symplectic and reversibility preserving, i.e. 
$U_{\mu,\epsilon}^* \mathcal{J} U_{\mu,\epsilon} = \mathcal{J}$ and $U_{\mu,\epsilon}\circ\overline{\rho}=\overline{\rho}\circ U_{\mu,\epsilon}$. Recall \eqref{F basis set and f}, i.e. $f_k^\sigma(\mu,\epsilon):=U_{\mu,\epsilon} f_k^\sigma, \sigma=\pm, k=0,1.$
Then for any $k,k',\sigma,\sigma'$,
\begin{align*}
(\mathcal{J} f_k^\sigma(\mu,\epsilon), f_{k'}^{\sigma'}(\mu,\epsilon))
= (\mathcal{J} f_k^\sigma, f_{k'}^{\sigma'}),
\end{align*}
so the symplectic relations \eqref{basis is symplectic} are preserved. Moreover,
\begin{align*}
\overline{\rho} f_k^\sigma(\mu,\epsilon)
= U_{\mu,\epsilon}\circ \overline{\rho} f_k^\sigma
= \sigma f_k^\sigma(\mu,\epsilon),
\end{align*}
which shows that the reversibility conditions \eqref{basis is reversible} hold as well. 
Finally, the analyticity of $f_k^\sigma(\mu,\epsilon)$ follows from the analyticity of 
$U_{\mu,\epsilon}$ (Lemma \ref{kato thm}). 
\end{proof}

We then expand the vectors $f^\sigma_k(\mu,\e)$ in $(\mu,\e)$. We denote by $\mathrm{even}_0(x)$ a real, even, $2\pi$-periodic function with zero space average. In the sequel $\cO(\mu^m\e^n)\vet{\mathrm{even}(x)}{\mathrm{odd}(x)}$ denotes an analytic map in $(\mu,\e)$ with values in $Y=H^2(\mathbb{T},\mathbb{C})\times H^1(\mathbb{T},\mathbb{C})$, whose first component is $\mathrm{even}(x)$ and the second one $\mathrm{odd}(x)$; we have a similar meaning for $\cO(\mu^m\e^n)\vet{\mathrm{odd}(x)}{\mathrm{even}(x)}$, etc $\ldots$.

\begin{lemma} [Expansion of the basis $\mathcal{F}$] \label{expansion of the basis F} For small values of $(\mu,\e)$ the basis $\mathcal{F}$ in \eqref{F basis set and f} has the expansion

\begin{equation} \label{43 f+1}
\begin{aligned} 
f_1^+(\mu,\epsilon) &= 
\vet{(\ch/\sqrt{1+\kappa})^{1/2}\cos(x)}{(\ch/\sqrt{1+\kappa})^{-1/2}\sin(x)}
+ \im \frac{\mu}{4} \gamma_{\tth,\kappa}
\vet{(\ch/\sqrt{1+\kappa})^{1/2}\sin(x)}{(\ch/\sqrt{1+\kappa})^{-1/2}\cos(x)}
+ \epsilon
\begin{bmatrix}
\alpha_{\tth,\kappa} \cos(2x) \\
\beta_{\tth,\kappa} \sin(2x)
\end{bmatrix} \\
& \quad + \mathcal{O}(\mu^2)
\begin{bmatrix}
\mathit{even}_0(x) + \im \mathit{odd}(x) \\
\mathit{odd}(x) + \im \mathit{even}_0(x)
\end{bmatrix}
+ \mathcal{O}(\epsilon^2)
\begin{bmatrix}
\mathit{even}_0(x) \\
\mathit{odd}(x)
\end{bmatrix} + \im \mu \epsilon
\begin{bmatrix}
\mathit{odd}(x) \\
\mathit{even}(x)
\end{bmatrix}
+ \mathcal{O}(\mu^2 \epsilon, \mu \epsilon^2),
\end{aligned}
\end{equation}

\begin{equation} \label{44 f-1}
\begin{aligned}
f_1^-(\mu,\epsilon) &= 
\vet{-(\ch/\sqrt{1+\kappa})^{1/2}\sin(x)}{(\ch/\sqrt{1+\kappa})^{-1/2}\cos(x)}
+ \im \frac{\mu}{4}\gamma_{\tth,\kappa}
\vet{(\ch/\sqrt{1+\kappa})^{1/2}\cos(x)}{-(\ch/\sqrt{1+\kappa})^{-1/2}\sin(x)}
+ \epsilon
\begin{bmatrix}
- \alpha_{\tth,\kappa} \sin(2x) \\
\beta_{\tth,\kappa} \cos(2x)
\end{bmatrix} \\
& \quad + \mathcal{O}(\mu^2)
\begin{bmatrix}
\mathit{odd}(x) + \im \mathit{even}_0(x) \\
\mathit{even}_0(x) + \im \mathit{odd}(x)
\end{bmatrix}
+ \mathcal{O}(\epsilon^2)
\begin{bmatrix}
\mathit{odd}(x) \\
\mathit{even}(x)
\end{bmatrix}   + \im \mu \epsilon
\begin{bmatrix}
\mathit{even}(x) \\
\mathit{odd}(x)
\end{bmatrix}
+ \mathcal{O}(\mu^2 \epsilon, \mu \epsilon^2),
\end{aligned}
\end{equation}

\begin{equation} \label{45 f+0}
\begin{aligned}
f_0^+(\mu,\epsilon) &=
\begin{bmatrix}
1 \\
0
\end{bmatrix}
+ \epsilon \delta_{\tth,\kappa}
\vet{(\ch/\sqrt{1+\kappa})^{1/2}\cos(x)}{-(\ch/\sqrt{1+\kappa})^{-1/2}\sin(x)}
+ \mathcal{O}(\epsilon^2)
\begin{bmatrix}
\mathit{even}_0(x) \\
\mathit{odd}(x)
\end{bmatrix}  \\
& \quad + \im \mu \epsilon
\begin{bmatrix}
\mathit{odd}(x) \\
\mathit{even}(x)
\end{bmatrix}
+ \mathcal{O}(\mu^2 \epsilon, \mu \epsilon^2),
\end{aligned}
\end{equation}

\begin{equation} \label{46 f-0}
\begin{aligned}
f_0^-(\mu,\epsilon) &=
\begin{bmatrix}
0 \\
1
\end{bmatrix}
+ \im \mu \epsilon
\begin{bmatrix}
\mathit{even}_0(x) \\
\mathit{odd}(x)
\end{bmatrix}
+ \mathcal{O}(\mu^2 \epsilon, \mu \epsilon^2),
\end{aligned}
\end{equation}
where the remainders $\mathcal{O}()$ are vectors in $Y=H^2(\mathbb{T},\mathbb{C})\times H^1(\mathbb{T},\mathbb{C})$ and
\begin{equation} \label{47 coefficients}
\begin{aligned}
\alpha_{\tth,\kappa} &= \frac{\left(\ch^4 (\kappa+1)-3 \kappa+3\right)}{2 \ch^{\frac{3}{2}} \sqrt[4]{\kappa+1} \left(\ch^4 (\kappa+1)-3 \kappa\right)}, \quad
\beta_{\tth,\kappa} = \frac{\left(3-\ch^4\right) \left(\ch^4+1\right) (\kappa+1)^{\frac{5}{4}}}{4 \ch^{5/2} \left(\ch^4 (\kappa+1)-3 \kappa\right)}, \\
\gamma_{\tth,\kappa} &= 1+\tth(\ch^{-2}-\ch^2)-2\kappa(1+\kappa)^{-1}, \quad
\delta_{\tth,\kappa} = \frac{1}{4}(1+\kappa)^{\frac{1}{4}}\ch^{-\frac{1}{2}}(\ch^2+3\ch^{-2}).
\end{aligned}
\end{equation}
For $\mu = 0$, the basis $\{ f_j^\pm(0,\epsilon),\, \epsilon = 0 \}$ is real and
\begin{equation} \label{48 f-0=01}
\begin{aligned}
f_1^+(0,\epsilon) &= 
\begin{bmatrix}
\mathit{even}_0(x) \\
\mathit{odd}(x)
\end{bmatrix}, \quad
f_1^-(0,\epsilon) = 
\begin{bmatrix}
\mathit{odd}(x) \\
\mathit{even}(x)
\end{bmatrix},  \\
f_0^+(0,\epsilon) &=\vet{1}{0}+
\begin{bmatrix}
\mathit{even}_0(x) \\
\mathit{odd}(x)
\end{bmatrix}, \quad
f_0^-(0,\epsilon) = 
\begin{bmatrix}
0 \\
1
\end{bmatrix}.
\end{aligned}
\end{equation}
\end{lemma}
\begin{proof}
    The long computations are given in Appendix \ref{secA1}.
\end{proof}

We now state the main result of this section.

\begin{lemma} \label{B decomposition}
    The matrix that represents the Hamiltonian and reversible operator $\mathscr{L}_{\mu,\epsilon} : \mathcal{V}_{\mu,\epsilon} \to \mathcal{V}_{\mu,\epsilon}$ in the symplectic and reversible basis $\mathcal{F}$ of $\mathcal{V}_{\mu,\epsilon}$ defined in \eqref{F basis set and f}, is a Hamiltonian matrix $\mathsf{L}_{\mu,\epsilon} = \mathsf{J}_4 \mathsf{B}_{\mu,\epsilon}$, where $\mathsf{B}_{\mu,\epsilon}$ is a self-adjoint and reversibility preserving (i.e., satisfying \eqref{B are alternatively real or imaginary}) $4 \times 4$ matrix of the form
\begin{align} \label{B mu epsilon}
    \mathsf{B}_{\mu,\epsilon} = 
\begin{pmatrix}
E & F \\
F^* & G
\end{pmatrix}, \quad E = E^*, \quad G = G^*,
\end{align}
where $E, F, G$ are the $2 \times 2$ matrices
\begin{equation} \label{E}
\begin{aligned} 
E &:= 
\begin{pmatrix}
\mathsf{e}_{11} \epsilon^2 (1 + \hat{r}_1(\epsilon,\mu\epsilon)) - \mathsf{e}_{22} \frac{\mu^2}{8} (1 + \check{r}_1(\epsilon,\mu)) &
\im\left(\frac{1}{2} \mathsf{e}_{12} \mu + r_2(\mu \epsilon^2, \mu^2 \epsilon, \mu^3)\right) \\
- \im\left(\frac{1}{2} \mathsf{e}_{12} \mu + r_2(\mu \epsilon^2, \mu^2 \epsilon, \mu^3)\right) &
- \mathsf{e}_{22} \frac{\mu^2}{8}(1 + r_5(\epsilon, \mu))
\end{pmatrix}, 
\end{aligned}
\end{equation}

\begin{equation} \label{F}
\begin{aligned} 
F &:= 
\begin{pmatrix}
\mathsf{f}_{11} \epsilon + r_3(\epsilon^3, \mu \epsilon^2, \mu^2 \epsilon) &
\im \,\mu \epsilon \ch^{-\frac{1}{2}}(1+\kappa)^{\frac{1}{4}} + \im\, r_4(\mu \epsilon^2, \mu^2 \epsilon) \\
\im\, r_6(\mu \epsilon) &
r_7(\mu^2 \epsilon)
\end{pmatrix}, 
\end{aligned}
\end{equation}

\begin{equation} \label{G}
\begin{aligned} 
G &:= 
\begin{pmatrix}
1 +\kappa\mu^2+ r_8(\epsilon^2, \mu^2 \epsilon) &
- \im r_9(\mu \epsilon^2, \mu^2 \epsilon) \\
\im r_9(\mu \epsilon^2, \mu^2 \epsilon) &
\mu \tanh(h \mu) + r_{10}(\mu^2 \epsilon)
\end{pmatrix}, 
\end{aligned}
\end{equation}
where $\mathsf{e}_{22}$ is defined in \eqref{def:e22}, $\mathsf{e}_{\mathrm{WB}}$ is defined in \eqref{def:D}, and
\begin{equation} \label{e11 f11}
    \begin{aligned} 
&\mathsf{e}_{11}:=\mathsf{e}_{11}(\kappa,\tth) := \frac{(21-25\ch^4+6\ch^8)(1+\kappa)^2+(-12+6\ch^4+3\ch^8)(1+\kappa)+9\ch^4}{8\ch^3(1+\kappa)^{\frac{1}{2}}\,\left(\ch^4(1+\kappa) -3\kappa\right)},\\
&\mathsf{f}_{11}:=\mathsf{f}_{11}(\kappa,\tth) := -\frac{1}{2}\left(\ch^{\frac{5}{2}}-\ch^{-\frac{3}{2}}\right)\left(1+\kappa\right)^{\frac{3}{4}}. 
\end{aligned}
\end{equation}
\end{lemma}

Recalling  \ref{denneq0}, we notice that $(\kappa,\tth)\notin \fR_2$ implies that the denominator in $\mathsf{e}_{11}$ is not zero. We then verify that $\mathsf{e}_{11}$ is nonzero for all $(\kappa,\tth)\notin \fR$ before proceeding to the proof of Lemma \ref{B decomposition}.

\begin{lemma} \label{e11 le 0}
    For any $\kappa>0$ and $\tth>0$ the numerator in $\mathsf{e}_{11}$ in \eqref{e11 f11} is positive and, hence, $\mathsf{e}_{11}$ is not zero for $(\kappa,\tth)\notin \fR$. In particular, under the Bond number condition (cf.  \eqref{Bond}) $\mathsf{e}_{11}$ is negative. 
\end{lemma}
\begin{proof}
    A straightforward calculation reveals that
    \begin{equation*}
    \begin{aligned}
        &21(1+\kappa)^2-12(1+\kappa) -25\ch^4{\left(1+\kappa\right)}^2+6\ch^8\,{\left(1+\kappa\right)}^2+9\ch^4+6\ch^4\left(1+\kappa\right)+3\ch^8\left(1+\kappa\right)\\
       =&\left(6\left(\ch^4-\frac{25}{12}\right)^2-\frac{121}{24}\right)\kappa^2+\left(15\left(\ch^4-\frac{22}{15}\right)^2-\frac{34}{15}\right)\kappa+9(1-\ch^4)^2+8\ch^4>0,
    \end{aligned}
    \end{equation*}
    since $0<\ch^4<1$ and, hence,
    \begin{align*}
        2< 6\left(\ch^4-\frac{25}{12}\right)^2-\frac{121}{24}< 21, \quad 1< 15\left(\ch^4-\frac{22}{15}\right)^2-\frac{34}{15}< 30.
    \end{align*}
    Under the Bond number condition, we observe that the denominator of $\mathsf{e}_{11}$ is negative, since
    $$
    3\kappa - \ch^4(1+\kappa) = \ch^4 \left[\tth^2 \left(\frac{3}{\ch^4} -1 \right) \frac{\kappa}{\tth^2} -1 \right] > \ch^4 \left[\tth^2 \left(\frac{1}{\ch^4} - \frac13 \right)  -1 \right]
    $$
    and one easily checks that the function $\tth^2 \left(\frac{1}{\ch^4} - \frac13 \right) > 1$ for any $\tth > 0$.
\end{proof}
We decompose $\mathcal{B}_{\mu,\epsilon}$ in \eqref{mathcal B mu e} as
\[
\mathcal{B}_{\mu,\epsilon} = \mathcal{B}_\epsilon +\mathcal{B}^{\flat} + \mathcal{B}^\sharp+\mathcal{B}^s,
\]
where $\mathcal{B}_\epsilon$, $\mathcal{B}^{\flat}$, $\mathcal{B}^\sharp$, $\mathcal{B}^s$ are the self-adjoint and reversibility preserving operators:
\begin{align} \label{B epsilon}
\mathcal{B}_\epsilon := \mathcal{B}_{0,\epsilon} := 
\begin{bmatrix}
1 + a_\epsilon(x)-\kappa \Sigma_{0,\epsilon} & -(\chk + p_\epsilon(x))\partial_x \\
\partial_x \circ (\chk + p_\epsilon(x)) & |D| \tanh((\tth + \ttf)|D|)
\end{bmatrix}, 
\end{align}

\begin{align} \label{B b}
\mathcal{B}^{\flat} := 
\begin{bmatrix}
0 & 0 \\
0 & |D + \mu| \tanh((\tth + \ttf)|D + \mu|) - |D| \tanh((\tth + \ttf)|D|)
\end{bmatrix}, 
\end{align}

\begin{align} \label{B sharp}
\mathcal{B}^\sharp :=  
\begin{bmatrix}
0 & -\im\, \mu p_\epsilon \\
\im\, \mu p_\epsilon & 0
\end{bmatrix}, 
\qquad 
\mathcal{B}^s :=  
\begin{bmatrix}
-\kappa \left(\mu\dot{\Sigma}_{0,\epsilon}+\frac{1}{2}\mu^2\ddot{\Sigma}_{0,\epsilon}\right) & 0 \\
0 & 0
\end{bmatrix}, 
\end{align}
where $\ttf$ in \eqref{def:ttf}, $a_\e(x)$ and $p_\e(x)$ in \eqref{SN1} and 
\begin{align*}
    \dot{\Sigma}_{0,\epsilon}:=\frac{\im}{1+\mathfrak{p}_x}\left(g_\e(x)\circ\pa_x+\pa_x\circ g_\e(x)\right)\frac{1}{1+\mathfrak{p}_x}, \quad \ddot{\Sigma}_{0,\epsilon}:=-\frac{2g_{\e}(x)}{(1+\mathfrak{p}_x)^2}.
\end{align*}
In view of \eqref{DtanhD}, the operator $\mathcal{B}^{\flat}$ is analytic in $\mu$.

\begin{lemma} [Expansion of $\mathsf{B}_\epsilon$] \label{Expansion of B epsilon}
The self-adjoint and reversibility preserving matrix 
$\mathsf{B}_\epsilon := \mathsf{B}_\epsilon(\mu)$ associated, as in \eqref{Bmuepsilon matrix representation}, 
with the self-adjoint and reversibility preserving operator 
$\mathcal{B}_\epsilon$ defined in \eqref{B epsilon}, with respect to the basis 
$\mathcal{F}$ of $\mathcal{V}_{\mu,\epsilon}$ in \eqref{F basis set and f}, expands as
\begin{equation} \label{Be90}
\begin{aligned}
\mathsf{B}_\epsilon =
\left(
\begin{NiceArray}{cc|ccc}[code-for-first-col=\scriptstyle]
\mathsf{e}_{11} \epsilon^2 + \zeta_{\tth,\kappa} \mu^2 + r_1(\epsilon^3, \mu \epsilon^3)
  & \im\, r_2(\mu \epsilon^2)
  & \mathsf{f}_{11} \epsilon + r_3(\epsilon^3, \mu \epsilon^2)
  & \im\, r_4(\mu \epsilon^3)
  &  \\
- \im\, r_2(\mu \epsilon^2)
  & \zeta_{\tth,\kappa} \mu^2
  & \im\, r_6(\mu \epsilon)
  & 0
  &  \\
\hline
\mathsf{f}_{11} \epsilon + r_3(\epsilon^3, \mu \epsilon^2)
  & - \im\, r_6(\mu \epsilon)
  & 1 + r_8(\epsilon^2, \mu \epsilon^2)
  & \im\, r_9(\mu \epsilon^2)
  & \\
- \im\, r_4(\mu \epsilon^3)
  & 0
  & - \im\, r_9(\mu \epsilon^2)
  & 0
  &  \\
\end{NiceArray}
\right)
\end{aligned}+\cO(\mu^2\e,\mu^3)
\end{equation}
where $\mathsf{e}_{11}$, $\mathsf{f}_{11}$ are defined respectively in \eqref{e11 f11}, and
\begin{equation} \label{zeta h}
\zeta_{\tth,\kappa} := \frac{1}{8}\ch(1+\kappa)^{\frac{1}{2}}\gamma^2_{\tth,\kappa}=\frac{1}{8}\ch(1+\kappa)^{\frac{1}{2}}\left(1+\tth(\ch^{-2}-\ch^2)-2\kappa(1+\kappa)^{-1}\right)^2.
\end{equation}
\end{lemma}

\begin{proof}
    We expand the matrix $\mathsf{B}_\epsilon(\mu)$ as

\begin{equation} \label{expansion of B_e in mu}
\mathsf{B}_\epsilon(\mu) = \mathsf{B}_\epsilon(0) + \mu(\partial_\mu \mathsf{B}_\epsilon)(0) + \frac{\mu^2}{2} (\partial_\mu^2 \mathsf{B}_0)(0) + \mathcal{O}(\mu^2 \epsilon, \mu^3).
\end{equation}

\textbf{The matrix $\mathsf{B}_\epsilon(0)$.}
The main result of this long paragraph is to prove that the matrix $\mathsf{B}_\epsilon(0)$ has the expansion \eqref{Be(0)}. We start recalling that the self-adjoint $\mathsf{B}_\e(0)$ has the form 
\begin{align} \label{B epsilon 0}
    \mathsf{B}_{\epsilon}(0) = 
\left(
\begin{array}{cccc}
\left( \mathcal{B}_{\epsilon} f_1^{+}(\epsilon),\ f_{1}^{+}(\epsilon) \right) & 
\left( \mathcal{B}_{\epsilon} f_1^{-}(\epsilon),\ f_{1}^{+}(\epsilon) \right) & 
\left( \mathcal{B}_{\epsilon} f_0^{+}(\epsilon),\ f_{1}^{+}(\epsilon) \right) & 
\left( \mathcal{B}_{\epsilon} f_0^{-}(\epsilon),\ f_{1}^{+}(\epsilon) \right) \\
\left( \mathcal{B}_{\epsilon} f_1^{+}(\epsilon),\ f_{1}^{-}(\epsilon) \right) & 
\left( \mathcal{B}_{\epsilon} f_1^{-}(\epsilon),\ f_{1}^{-}(\epsilon) \right) & 
\left( \mathcal{B}_{\epsilon} f_0^{+}(\epsilon),\ f_{1}^{-}(\epsilon) \right) & 
\left( \mathcal{B}_{\epsilon} f_0^{-}(\epsilon),\ f_{1}^{-}(\epsilon) \right) \\
\left( \mathcal{B}_{\epsilon} f_1^{+}(\epsilon),\ f_{0}^{+}(\epsilon) \right) & 
\left( \mathcal{B}_{\epsilon} f_1^{-}(\epsilon),\ f_{0}^{+}(\epsilon) \right) & 
\left( \mathcal{B}_{\epsilon} f_0^{+}(\epsilon),\ f_{0}^{+}(\epsilon) \right) & 
\left( \mathcal{B}_{\epsilon} f_0^{-}(\epsilon),\ f_{0}^{+}(\epsilon) \right) \\
\left( \mathcal{B}_{\epsilon} f_1^{+}(\epsilon),\ f_{0}^{-}(\epsilon) \right) & 
\left( \mathcal{B}_{\epsilon} f_1^{-}(\epsilon),\ f_{0}^{-}(\epsilon) \right) & 
\left( \mathcal{B}_{\epsilon} f_0^{+}(\epsilon),\ f_{0}^{-}(\epsilon) \right) & 
\left( \mathcal{B}_{\epsilon} f_0^{-}(\epsilon),\ f_{0}^{-}(\epsilon) \right)
\end{array}
\right),
\end{align}
where $f^{\sigma}_k(\e):=f^{\sigma}_k(0,\e)$, for $\sigma=\pm$, $k=0,1$. The matrix $\mathsf{B}_\epsilon(0)$ is real, because the operator $\mathcal{B}_\epsilon$ is real and the basis $\{f_k^{\pm}(0, \epsilon)\}_{k=0,1}$ is real.
On the other hand, by \eqref{B are alternatively real or imaginary}, its matrix elements $(\mathsf{B}_\epsilon(0))_{i,j}$ vanish for $i + j$ odd. In addition $f_0^{-}(0,\epsilon) = 
\begin{bmatrix}
0 \\ 1
\end{bmatrix}
$ by \eqref{48 f-0=01}, and, by \eqref{B epsilon}, we get $\mathcal{B}_\epsilon f_0^{-}(0,\epsilon) = 0$, for any $\epsilon$.
We deduce that the self-adjoint matrix $\mathsf{B}_\epsilon(0)$ in \eqref{B epsilon 0} has the form
\begin{equation} \label{B_e 0}
\mathsf{B}_\epsilon(0) =\left(
\begin{NiceArray}{cc|ccc}[code-for-first-col=\scriptstyle]
E_{11}(0,\epsilon)
  & 0
  & F_{11}(0,\epsilon)
  & 0
  &  \\
0
  & E_{22}(0,\epsilon)
  & 0
  & 0
  &  \\
\hline
F_{11}(0,\epsilon)
  & 0
  & G_{11}(0,\epsilon)
  & 0
  & \\
0
  & 0
  & 0
  & 0
  &  \\
\end{NiceArray}
\right),
\end{equation}
where \( E_{11}(0,\epsilon), E_{22}(0,\epsilon), G_{11}(0,\epsilon), F_{11}(0,\epsilon) \) are real.
We claim that \( E_{22}(0,\epsilon) \equiv 0 \) for any \( \epsilon \). As a first step, following \cite{BMV1}, we prove that

\begin{align} \label{E22=0 or E11=0=F11}
    \text{either } E_{22}(0,\epsilon) \equiv 0, \quad \text{or} \quad E_{11}(0,\epsilon) \equiv 0 \equiv F_{11}(0,\epsilon).
\end{align}
Indeed, by \eqref{237}, the operator \( \mathscr{L}_{0,\epsilon} \equiv \mathcal{L}_{0,\epsilon} \) possesses, for any sufficiently small \( \epsilon \neq 0 \), the eigenvalue 0 with a four dimensional generalized Kernel 
\[
\mathcal{W}_{\epsilon} := \text{span}\{ U_1, \tilde{U}_2, U_3, U_4 \},
\]
spanned by \( \epsilon \)-dependent vectors. By Lemma \ref{kato thm} it results that \( \mathcal{W}_\epsilon = \mathcal{V}_{0,\epsilon} = \text{Rg}(P_{0,\epsilon}) \) and by \eqref{237} we have \( \mathscr{L}_{0,\epsilon}^2 = 0 \) on \( \mathcal{V}_{0,\epsilon} \). Thus the matrix

\begin{equation} \label{Le0}
\mathsf{L}_\epsilon(0) := \mathsf{J}_4 \mathsf{B}_\epsilon(0) = 
\left(
\begin{NiceArray}{cc|ccc}[code-for-first-col=\scriptstyle]
0
  & E_{22}(0,\epsilon)
  & 0
  & 0
  &  \\
-E_{11}(0,\epsilon)
  & 0
  & -F_{11}(0,\epsilon)
  & 0
  &  \\
\hline
0
  & 0
  & 0
  & 0
  & \\
-F_{11}(0,\epsilon)
  & 0
  & -G_{11}(0,\epsilon)
  & 0
  &  \\
\end{NiceArray}
\right),
\end{equation}
which represents \( \mathscr{L}_{0,\epsilon} : \mathcal{V}_{0,\epsilon} \to \mathcal{V}_{0,\epsilon} \), satisfies \( \mathsf{L}_\epsilon^2(0) = 0 \), namely

\begin{equation}
\mathsf{L}_\epsilon^2(0) = - 
\left(
\begin{NiceArray}{cc|ccc}[code-for-first-col=\scriptstyle]
(E_{11}E_{22})(0,\epsilon)
  & 0
  & (E_{22}F_{11})(0,\e)
  & 0
  &  \\
0
  & (E_{11}E_{22})(0,\epsilon)
  & 0
  & 0
  &  \\
\hline
0
  & 0
  & 0
  & 0
  & \\
0
  & (E_{22}F_{11})(0,\epsilon)
  & 0
  & 0
  &  \\
\end{NiceArray}
\right)=0,
\end{equation}
which implies \eqref{E22=0 or E11=0=F11}. We now prove that the matrix \( \mathsf{B}_\epsilon(0) \) defined in \eqref{B_e 0} expands as

\begin{equation} \label{Be(0)}
\mathsf{B}_\epsilon(0) = 
\left(
\begin{NiceArray}{cc|ccc}[code-for-first-col=\scriptstyle]
\mathsf{e}_{11}\e^2+r(\e^3)
  & 0
  & \mathsf{f}_{11}\e+r(\e^3)
  & 0
  &  \\
0
  & 0
  & 0
  & 0
  &  \\
\hline
\mathsf{f}_{11}\e+r(\e^3)
  & 0
  & 1+r(\e^2)
  & 0
  & \\
0
  & 0
  & 0
  & 0
  &  \\
\end{NiceArray}
\right)
\end{equation}
where $\mathsf{e}_{11}$ and $\mathsf{f}_{11}$ are in \eqref{E110epsilon} and \eqref{F110e}. We expand the operator $\mathcal{B}_{\e}$ in \eqref{B epsilon} as 
\begin{equation} \label{Be B0 B1 B2}
\begin{aligned}
    &\mathcal{B}_{\e}=\mathcal{B}_{0}+\e \mathcal{B}_{1}+\e^2 \mathcal{B}_{2}+\cO(\e^3), \quad \mathcal{B}_{0}:=\begin{bmatrix}
1-\kappa\pa_{xx}&  -\chk\pa_x\\
        \chk\pa_x & ~~|D|\tanh(\tth|D|)
    \end{bmatrix},\\
    &\mathcal{B}_{1}:=\begin{bmatrix}
a_1(x)-\kappa\Sigma_1&  -p_1(x)\pa_x\\
       \pa_x\circ p_1(x) & ~~0
    \end{bmatrix}, \quad \mathcal{B}_{2}:=\begin{bmatrix}
a_2(x)-\kappa\Sigma_2&  -p_2(x)\pa_x\\
        \pa_x \circ p_2(x) & ~~ -\tf^{[2]}_2 \pa^2_x(1-\tanh^2(\tth|D|))
    \end{bmatrix},
\end{aligned}
\end{equation}
where the remainder term $\cO(\e^3)\in\mathcal{L}(Y,X)$, the functions $a_1$, $p_1$, $a_2$, $p_2$,  are given in \eqref{SN1}-\eqref{a[2]2}, the operators $\Sigma_1$, $\Sigma_2$ are given in \eqref{Sigma epsilon}-\eqref{h2 h02} and, in view of \eqref{expfe}, $\tf^{[2]}_2=\frac{1}{4}\ch^{-2}\left(\ch^4(1+\kappa)-(3+\kappa)\right)$.

Our goal here is to determine the expansion of $E_{11}(0,\e)=\mathsf{e}_{11}\e^2+r(\e^3)$. By \eqref{43 f+1} we rewrite the real function $f^+_1(0,\e)$ as
\begin{equation} \label{expansion fp1 0 epsilon}
    \begin{aligned}
        & f^+_1(0,\e)=f^+_1+\e f^+_{1_1}+\e^2 f^+_{1_2}+\cO(\e^3),\\
        & f^+_1=\vet{\left(\frac{\ch}{\sqrt{1+\kappa}}\right)^{\frac{1}{2}}\cos(x)}{\left(\frac{\ch}{\sqrt{1+\kappa}}\right)^{-\frac{1}{2}}\sin(x)}, \quad f^+_{1_1}:=\vet{\alpha_{\tth,\kappa} \cos(2x)}{\beta_{\tth,\kappa}\sin(2x)}, \quad f^+_{1_2}:=\vet{\mathrm{even}_0(x)}{\mathrm{odd}(x)},
    \end{aligned}
\end{equation}
where both $f^+_{1_2}$ and $\cO(\e^3)$ are vectors in $H^2(\mathbb{T},\mathbb{C})\times H^1(\mathbb{T},\mathbb{C})$. Since $\mathcal{B}_0 f^+_1=-\mathcal{J}\mathscr{L}_{0,0} f^+_1=0$, and both $\mathcal{B}_0$, $\mathcal{B}_1$ are self-adjoint real operators, it results
\begin{equation}
\begin{aligned}
E_{11}(0, \epsilon) &= \left( \mathcal{B}_\epsilon f_1^+(0,\epsilon), f_1^+(0,\epsilon) \right) \\
&= \epsilon \left( \mathcal{B}_1 f_1^+, f_1^+ \right) + \epsilon^2 \left[ \left( \mathcal{B}_2 f_1^+, f_1^+ \right) + 2 \left( \mathcal{B}_1 f_1^+, f^+_{1_1} \right) + \left( \mathcal{B}_0 f^+_{1_1}, f^+_{1_1} \right) \right] + \mathcal{O}(\epsilon^3).
\end{aligned}
\end{equation}
By \eqref{Be B0 B1 B2} one has
\begin{equation} \label{B1fp1 B2fp1 B0fp11}
\mathcal{B}_1 f_1^+ = 
\begin{bmatrix}
A^{[0]}_1 + A^{[2]}_1\cos(2x)) \\
B^{[2]}_1 \sin(2x)
\end{bmatrix},
\quad 
\mathcal{B}_2 f_1^+ = 
\begin{bmatrix}
A^{[1]}_2 \cos(x) + A^{[3]}_2 \cos(3x) \\
B^{[1]}_2 \sin(x) + B^{[3]}_2 \sin(3x)
\end{bmatrix}, \quad
\mathcal{B}_0 f_{1_1}^+ = 
\begin{bmatrix}
A^{[2]}_3 \cos(2x) \\
B^{[2]}_3 \sin(2x)
\end{bmatrix}, 
\end{equation}
with 
\begin{equation} \label{ABABAB}
\begin{aligned}
A^{[0]}_1 &:= -\frac{1}{2}\left(\ch^{\frac{5}{2}}-\ch^{-\frac{3}{2}}\right)\left(1+\kappa\right)^{\frac{3}{4}}, \quad A^{[2]}_1:= -\frac{(1+\kappa)\ch^4+5\kappa-1}{2\ch^{\frac{3}{2}}(1+\kappa)^{\frac{1}{4}}}, \quad B^{[2]}_1 := 2 \ch^{-\frac{1}{2}} (1+\kappa)^{\frac{1}{4}}, \\
A^{[1]}_2 &:=\ch^\frac{1}{2}(1+\kappa)^{-\frac{1}{4}}\left(a^{[0]}_2+\frac{1}{2}a^{[2]}_2\right)-\ch^{-\frac{1}{2}}(1+\kappa)^{\frac{1}{4}}\left(p^{[0]}_2+\frac{1}{2}p^{[2]}_2\right)\\
&+\kappa\ch^{\frac{1}{2}}(1+\kappa)^{-\frac{1}{4}}\left(d^{[0]}_2+\frac{1}{2}d^{[2]}_2+\frac{1}{2}e^{[2]}_2-h^{[0]}_2-\frac{1}{2}h^{[2]}_2\right), \\
A^{[3]}_2 &:=\frac{1}{2}\ch^\frac{1}{2}(1+\kappa)^{-\frac{1}{4}} a^{[2]}_2-\frac{1}{2}\ch^{-\frac{1}{2}}(1+\kappa)^{\frac{1}{4}} p^{[2]}_2+\frac{1}{2}\kappa\ch^{\frac{1}{2}}(1+\kappa)^{-\frac{1}{4}}\left(d^{[2]}_2-e^{[2]}_2-h^{[2]}_2\right),\\
B^{[1]}_2 &:=-\ch^{\frac{1}{2}}(1+\kappa)^{-\frac{1}{4}} p^{[0]}_2-\frac{1}{2}\ch^{\frac{1}{2}}(1+\kappa)^{-\frac{1}{4}} p^{[2]}_2+\mathtt{f}^{[2]}_2\ch^{-\frac{1}{2}}(1+\kappa)^{\frac{1}{4}}(1-\ch^4), \quad B^{[3]}_2:=-\frac{3}{2}\ch^{\frac{1}{2}}(1+\kappa)^{-\frac{1}{4}}p^{[2]}_2,\\
A^{[2]}_3 &:=\alpha_{\tth,\kappa}(1+4\kappa)-2\beta_{\tth,\kappa}\chk, \quad B^{[2]}_3:=-2\alpha_{\tth,\kappa}\chk+\frac{4\ch^2}{1+\ch^4}\beta_{\tth,\kappa}. 
\end{aligned}
\end{equation}
By \eqref{expansion fp1 0 epsilon}-\eqref{B1fp1 B2fp1 B0fp11}, we deduce
\begin{equation} \label{E110epsilon}
\begin{aligned}
E_{11}(0, \epsilon) &= \mathsf{e}_{11} \epsilon^2 + r(\epsilon^3), \\
\mathsf{e}_{11} &:= \frac{1}{2} \left( A^{[1]}_2 \ch^{\frac{1}{2}}(1+\kappa)^{-\frac{1}{4}} + B^{[1]}_2 \ch^{-\frac{1}{2}}(1+\kappa)^{\frac{1}{4}} + 2\alpha_{\tth,\kappa} A^{[2]}_1 + 2 \beta_{\tth,\kappa} B^{[2]}_1  + \alpha_{\tth,\kappa} A^{[2]}_3 + \beta_{\tth,\kappa} B^{[2]}_3 \right). 
\end{aligned}
\end{equation}
By \eqref{ABABAB}, \eqref{E110epsilon}, \eqref{47 coefficients}, \eqref{SN1}-\eqref{a[2]2}, and \eqref{Sigma epsilon}-\eqref{h2 h02}, we obtain  that $\mathsf{e}_{11}$ writes as in \eqref{e11 f11}. Since $\mathsf{e}_{11}$ is not zero as proved in Lemma \ref{e11 le 0},  the second alternative in \eqref{E22=0 or E11=0=F11} is ruled out, implying $E_{22}(0, \epsilon) \equiv 0$.

Next let us determine the expansion of $G_{11}(0,\e)$. By \eqref{45 f+0} we split the real-valued function $f_0^+(0, \epsilon)$ as
\begin{equation} \label{fp0 expansion}
\begin{aligned}
f_0^+(0, \epsilon) = f_0^+ + \epsilon f_{0_1}^+ + \epsilon^2 f_{0_2}^+ + \mathcal{O}(\epsilon^3), \quad 
f_0^+ = \begin{bmatrix} 1 \\ 0 \end{bmatrix},\\
f_{0_1}^+ := \delta_{\tth,\kappa}
\begin{bmatrix}
\left(\frac{\ch}{\sqrt{1+\kappa}}\right)^{\frac{1}{2}} \cos(x) \\
- \left(\frac{\ch}{\sqrt{1+\kappa}}\right)^{-\frac{1}{2}} \sin(x)
\end{bmatrix}, \quad
f_{0_2}^+ := 
\begin{bmatrix}
\mathrm{even}_0(x) \\
\mathrm{odd}(x)
\end{bmatrix}. 
\end{aligned}
\end{equation}
Since, by \eqref{eigenfunc of mathcall L00} and \eqref{Be B0 B1 B2}, $\mathcal{B}_0 f_0^+ = f_0^+$, using that $\mathcal{B}_0$, $\mathcal{B}_1$ are self-adjoint real operators, and $(f_0^+,f_0^+) = 1$, $(f_0^+, f_{0_1}^+)=0$, we have
\[
G_{11}(0, \epsilon) = (\mathcal{B}_\epsilon f_0^+(0, \epsilon), f_0^+(0, \epsilon)) = 1 + \epsilon(\mathcal{B}_1 f_0^+, f_{0}^+) + \epsilon(\mathcal{B}_0 f_{0_1}^+, f_{0}^+)+
r(\epsilon^2).
\]
By \eqref{Be B0 B1 B2} and \eqref{SN1}, \eqref{pino1fd}, \eqref{aino1fd}, and \eqref{Sigma epsilon}-\eqref{h1 h11} one has
\begin{equation} \label{B1fp0fp0 B0fp01fp0}
    \begin{aligned}
        \left(\mathcal{B}_1 f_0^+,f^+_0\right) =
\left(\begin{bmatrix}
\left(a_1^{[1]}-\kappa h^{[1]}_1\right) \cos(x) \\
- p_1^{[1]} \sin(x)
\end{bmatrix},\vet{1}{0}\right)=\vet{0}{0}, \quad (\mathcal{B}_0 f_{0_1}^+, f_{0}^+)=( f_{0_1}^+,\mathcal{B}_0f_{0}^+)=\vet{0}{0}.
    \end{aligned}
\end{equation}
By \eqref{fp0 expansion} and \eqref{B1fp0fp0 B0fp01fp0}, we deduce $G_{11}(0, \epsilon) = 1 + r(\epsilon^2)$.

Next we provide the expansion of  $F_{11}(0, \epsilon)$.  
By \eqref{Be B0 B1 B2}, \eqref{expansion fp1 0 epsilon}, \eqref{fp0 expansion}, using that $\mathcal{B}_0$, $\mathcal{B}_1$ are self-adjoint and real, and $\mathcal{B}_0 f_1^+ = 0$, $\mathcal{B}_0 f_0^+ = f_0^+$, we obtain
\begin{align*}
F_{11}(0, \epsilon) 
&= \epsilon \left[ \left( \mathcal{B}_1 f_1^+, f_0^+ \right) + \left( f_{1_1}^+, f_0^+ \right) \right] \\
&\quad + \epsilon^2 \left[ \left( \mathcal{B}_2 f_1^+, f_0^+ \right) + \left( \mathcal{B}_1 f_1^+, f_{0_1}^+ \right) + \left( \mathcal{B}_1 f_0^+, f_{1_1}^+ \right) \right. \\
&\qquad \left. + \left( f_{1_2}^+, f_0^+ \right) + \left( \mathcal{B}_0 f_{1_1}^+, f_{0_1}^+ \right) \right] + r(\epsilon^3).
\end{align*}
By \eqref{expansion fp1 0 epsilon}, \eqref{B1fp1 B2fp1 B0fp11}, \eqref{ABABAB}, \eqref{fp0 expansion}, \eqref{B1fp0fp0 B0fp01fp0}, all the only non-zero scalar product is the  first one, yielding
\begin{align} \label{F110e}
    F_{11}(0, \epsilon) = \mathsf{f}_{11} \epsilon + r(\epsilon^3), \quad \mathsf{f}_{11} := A^{[0]}_1 = -\frac{1}{2}\left(\ch^{\frac{5}{2}}-\ch^{-\frac{3}{2}}\right)\left(1+\kappa\right)^{\frac{3}{4}}.
\end{align}
The expansion \eqref{Be(0)} is proved.

The next step is to compute the terms linear in $\mu$ of $\mathsf{B}_\epsilon(\mu)$. We have 
\begin{align} \label{X+X}
    \partial_\mu \mathsf{B}_\epsilon(0) = X + X^* \quad \text{where} \quad 
X := \left( \left( \mathcal{B}_\epsilon f_k^\sigma(0,\epsilon), \, (\partial_\mu f_{k'}^{\sigma'})(0,\epsilon) \right) \right)_{k,k'=0,1;\, \sigma,\sigma'=\pm}. 
\end{align}
and now prove that 
\begin{equation} \label{X}
X = 
\left(
\begin{NiceArray}{cc|ccc}[code-for-first-col=\scriptstyle]
\cO(\e^3)
  & 0
  & \cO(\e^2)
  & 0
  &  \\
\cO(\e^2)
  & 0
  & \cO(\e)
  & 0
  &  \\
\hline
\cO(\e^3)
  & 0
  & \cO(\e^2)
  & 0
  & \\
\cO(\e^3)
  & 0
  & \cO(\e^2)
  & 0
  &  \\
\end{NiceArray}
\right)
\end{equation}
Indeed consider the matrix $\mathsf{L}_\epsilon(0)$ in \eqref{Le0}, where $E_{22}(0, \epsilon) = 0$, and recall that it represents the action of the operator $\mathscr{L}_{0,\epsilon} : \mathcal{V}_{0,\epsilon} \to \mathcal{V}_{0,\epsilon}$ in the basis $\{f_k^\sigma(0,\epsilon)\}$;  we deduce that
\[
\mathscr{L}_{0,\epsilon} f_1^-(0,\epsilon) = 0,\quad \mathscr{L}_{0,\epsilon} f_0^-(0,\epsilon) = 0.
\]
Thus also $\mathcal{B}_\epsilon f_1^-(0,\epsilon) =\mathcal{B}_\epsilon f_0^-(0,\epsilon) = 0$, and the second and fourth columns of the matrix $X$ in \eqref{X} are zero. Now we compute $\pa_\mu f_k^\sigma(0,\e)$.  
In view of \eqref{43 f+1}-\eqref{47 coefficients} and by denoting with a dot the derivative w.r.t.\ $\mu$, one has
\begin{equation} \label{dotfp1 dotfp0 dotfn1 dotfn0}
    \begin{aligned}
\dot{f}_1^+(0,\epsilon) &= \frac{\im}{4} \gamma_{\tth,\kappa} 
\begin{bmatrix}
\ch^{\frac{1}{2}}(1+\kappa)^{-\frac{1}{4}}\sin(x) \\
\ch^{-\frac{1}{2}}(1+\kappa)^{\frac{1}{4}} \cos(x)
\end{bmatrix}
+ \im\epsilon 
\begin{bmatrix}
\mathrm{odd}(x) \\
\mathrm{even}(x)
\end{bmatrix}
+ \mathcal{O}(\epsilon^2), \\
\dot{f}_0^+(0,\epsilon) &= \im\epsilon 
\begin{bmatrix}
\mathrm{odd}(x) \\
\mathrm{even}_0(x)
\end{bmatrix}
+ \mathcal{O}(\epsilon^2), \\
\dot{f}_1^-(0,\epsilon) &= \frac{\im}{4} \gamma_{\tth,\kappa} 
\begin{bmatrix}
\ch^{\frac{1}{2}}(1+\kappa)^{-\frac{1}{4}}\cos(x) \\
-\ch^{-\frac{1}{2}}(1+\kappa)^{\frac{1}{4}}\sin(x)
\end{bmatrix}
+ \im\epsilon 
\begin{bmatrix}
\mathrm{even}(x) \\
\mathrm{odd}(x)
\end{bmatrix}
+ \mathcal{O}(\epsilon^2), \\
\dot{f}_0^-(0,\epsilon) &= \im\epsilon 
\begin{bmatrix}
\mathrm{even}_0(x) \\
\mathrm{odd}(x)
\end{bmatrix}
+ \mathcal{O}(\epsilon^2). 
\end{aligned}
\end{equation}
In view of \eqref{pa_t eta psi}, \eqref{43 f+1}-\eqref{48 f-0=01}, \eqref{Le0}, \eqref{E110epsilon}, \eqref{F110e}, and since $\mathcal{B}_\epsilon f_k^\sigma(0, \epsilon) = -\mathcal{J} \mathscr{L}_\epsilon f_k^\sigma(0, \epsilon)$, we have
\begin{equation} \label{Befp1 Befp0}
   \begin{aligned}
    \mathcal{B}_\epsilon f_1^+(0, \epsilon) &= E_{11}(0,\epsilon) \, \mathcal{J} f_1^-(0,\epsilon) + F_{11}(0,\epsilon) \, \mathcal{J} f_0^- 
    = \epsilon 
    \begin{bmatrix}
        \mathsf{f}_{11} \\
        0
    \end{bmatrix}
    + \epsilon^2 \mathsf{e}_{11}
    \begin{bmatrix}
        \ch^{-\frac{1}{2}}(1+\kappa)^{\frac{1}{4}} \cos(x) \\
        \ch^{\frac{1}{2}}(1+\kappa)^{-\frac{1}{4}} \sin(x)
    \end{bmatrix}
    + \mathcal{O}(\epsilon^3),
    \\
    \mathcal{B}_\epsilon f_0^+(0, \epsilon) &= F_{11}(0,\epsilon) \, \mathcal{J} f_1^-(0,\epsilon) + G_{11}(0,\epsilon) \, \mathcal{J} f_0^- 
    = 
    \begin{bmatrix}
        1 \\
        0
    \end{bmatrix}
    + \epsilon \mathsf{f}_{11}
    \begin{bmatrix} 
        \ch^{-\frac{1}{2}}(1+\kappa)^{\frac{1}{4}} \cos(x) \\ 
        \ch^{\frac{1}{2}}(1+\kappa)^{-\frac{1}{4}} \sin(x)
    \end{bmatrix}
    + \mathcal{O}(\epsilon^2).
\end{aligned} 
\end{equation}
We deduce \eqref{X} by \eqref{dotfp1 dotfp0 dotfn1 dotfn0} and \eqref{Befp1 Befp0}.

Next we compute the terms quadratic in $\mu$. By denoting with a double dot the double derivative with respect to $\mu$, we have
\begin{align} \label{ddotB00}
\partial_\mu^2 \mathsf{B}_0(0) = \left( \mathcal{B}_0 f_k^\sigma, \ddot{f}_{k'}^{\sigma'}(0,0) \right)
+ \left( \ddot{f}_k^\sigma(0,0), \mathcal{B}_0 f_{k'}^{\sigma'} \right)
+ 2 \left( \mathcal{B}_0 \dot{f}_k^\sigma(0,0), \dot{f}_{k'}^{\sigma'}(0,0) \right)
= : Y + Y^* + 2Z. 
\end{align}
Arguing as in \cite[pag. 29]{BMV3}, one shows that $Y=0$ and 
\begin{align} \label{Z}
Z = \left( \mathcal{B}_0 \dot{f}_k^\sigma(0,0), \dot{f}_{k'}^{\sigma'}(0,0) \right)_{k,k'=0,1;\, \sigma,\sigma'=\pm}
= 
\left(
\begin{NiceArray}{cc|ccc}[code-for-first-col=\scriptstyle]
\zeta_{\tth,\kappa}
  & 0
  & 0
  & 0
  &  \\
0
  & \zeta_{\tth,\kappa}
  & 0
  & 0
  &  \\
\hline
0
  & 0
  & 0
  & 0
  & \\
0
  & 0
  & 0
  & 0
  &  \\
\end{NiceArray}
\right)
\end{align}
with 
\[
\zeta_{\tth,\kappa} := \left( \mathcal{B}_0 \dot{f}_1^+(0,0), \dot{f}_1^+(0,0) \right)
= \left( \mathcal{B}_0 \dot{f}_1^-(0,0), \dot{f}_1^-(0,0) \right)
= \frac{1}{8} \ch(1+\kappa)^{\frac{1}{2}} \gamma_{\tth,\kappa}^2.
\]
In conclusion, \eqref{expansion of B_e in mu}, \eqref{X+X}, \eqref{X}, \eqref{ddotB00}, the fact that $Y = 0$ and \eqref{Z} imply \eqref{Be90}, using also the self-adjointness of $\mathcal{B}_{\e}$ and \eqref{B are alternatively real or imaginary}. 
\end{proof}

\begin{lemma}
    [Expansion of $\mathsf{B}^{\flat}$] \label{Expansion of B flat} \textit{The self-adjoint and reversibility-preserving matrix $\mathsf{B}^{\flat}$ associated, as in \eqref{Bmuepsilon matrix representation}, to the self-adjoint and reversibility-preserving operator $\mathcal{B}^{\flat}$, defined in \eqref{B b}, with respect to the basis $\mathcal{F}$ of $\mathcal{V}_{\mu, \epsilon}$ in \eqref{F basis set and f}, admits the expansion}
\begin{equation} \label{mathsfBb}
\mathsf{B}^{\flat} = 
\left(
\begin{NiceArray}{cc|ccc}[code-for-first-col=\scriptstyle]
-\frac{\mu^2}{4}\mathbf{b}_{\tth,\kappa}
  & \im\,(\frac{\mu}{2}\hat{\mathbf{b}}_{\tth,\kappa}+r_2(\mu\e^2))
  & 0
  & 0
  &  \\
-\im\,(\frac{\mu}{2}\hat{\mathbf{b}}_{\tth,\kappa}+r_2(\mu\e^2))
  & -\frac{\mu^2}{4}\mathbf{b}_{\tth,\kappa}
  & \im\, r_6(\mu\e)
  & 0
  &  \\
\hline
0
  & -\im\, r_6(\mu\e)
  & 0
  & 0
  & \\
0
  & 0
  & 0
  & \mu \tanh(\tth \mu)
  &  \\
\end{NiceArray}
\right)
+ \mathcal{O}(\mu^2 \epsilon, \mu^3),
\end{equation}
where 
\begin{align} \label{bh440}
\mathbf{b}_{\tth,\kappa}& := \gamma_{\tth,\kappa}\ch(1+\kappa)^{\frac{1}{2}}+\tth\ch^{-1}(1+\kappa)^{\frac{1}{2}}(1-\ch^4)\left(\gamma_{\tth,\kappa}-2(1-\tth\ch^2)\right)\\
\label{bh hat}
\hat{\mathbf{b}}_{\tth,\kappa}&:=\ch^{-1}(1+\kappa)^{\frac{1}{2}} (\ch^2 + (1 - \ch^4)\tth). 
\end{align}
\end{lemma}
\begin{proof}
    The proof is analogous to the one of \cite[Lemma 4.5]{BMV3}, using the expansions of Lemma \ref{expansion of the basis F}.
\end{proof}

Next, we consider $\mathcal{B}^{\sharp}$.
\begin{lemma} [Expansion of $\mathsf{B}^{\sharp}$] \label{Expansion of B sharp}
    \textit{The self-adjoint and reversibility-preserving matrix $\mathsf{B}^{\sharp}$ associated, as in \eqref{Bmuepsilon matrix representation}, to the self-adjoint and reversibility-preserving operators $\mathcal{B}^{\sharp}$, defined in \eqref{B sharp}, with respect to the basis $\mathcal{F}$ of $\mathcal{V}_{\mu,\epsilon}$ in \eqref{F basis set and f}, admits the expansion}
\begin{align} \label{mathsf B sharp}
\mathsf{B}^{\sharp} =
\left(
\begin{NiceArray}{cc|ccc}[code-for-first-col=\scriptstyle]
0
  & \im\, r_2(\mu\e^2)
  & 0
  & \im\,\mu\e\ch^{-\frac{1}{2}}(1+\kappa)^{\frac{1}{4}}+\im\, r_4(\mu\e^2)
  &  \\
-\im\, r_2(\mu\e^2)
  & 0
  & -\im\, r_6(\mu\e)
  & 0
  &  \\
\hline
0
  & \im\, r_6(\mu\e)
  & 0
  & -\im\, r_9(\mu\e^2)
  & \\
-\im\, \mu \e \ch^{-\frac{1}{2}}(1+\kappa)^{\frac{1}{4}}-i r_4(\mu\e^2)
  & 0
  & \im\, r_9(\mu\e^2)
  & 0
  &  \\
\end{NiceArray}
\right)
+ \mathcal{O}(\mu^2 \epsilon),
\end{align}
\end{lemma}
\begin{proof}
    The proof is analogous to the one of \cite[Lemma 4.6]{BMV3}, using the expansions of Lemma \ref{expansion of the basis F}.
\end{proof}

Finally, we consider the new term  $\mathcal{B}^{s}$ that takes into the account the capillary term.
\begin{lemma} [Expansion of $\mathsf{B}^{s}$] \label{Expansion of B s}
    \textit{The self-adjoint and reversibility-preserving matrix $\mathsf{B}^{s}$ associated, as in \eqref{Bmuepsilon matrix representation}, to the self-adjoint and reversibility-preserving operators $\mathcal{B}^{s}$, defined in \eqref{B sharp}, with respect to the basis $\mathcal{F}$ of $\mathcal{V}_{\mu,\epsilon}$ in \eqref{F basis set and f}, admits the expansion}
\begin{equation}\label{mathsf B s}
\begin{aligned} 
\mathsf{B}^{s} &=
\left(
\begin{NiceArray}{cc|ccc}[code-for-first-col=\scriptstyle]
\frac{\mu^2}{2}\kappa(1+\kappa)^{-\frac{1}{2}}\ch(\gamma_{\tth,\kappa}+1)
  & \im\,\kappa\mu\ch(1+\kappa)^{-\frac{1}{2}}+\im \,r_2(\mu\e^2)
  & 0
  & \im \,r_4(\mu\e^2)
  &  \\
-\im\,\kappa\mu\ch(1+\kappa)^{-\frac{1}{2}}-\im \,r_2(\mu\e^2)
  & \frac{\mu^2}{2}\kappa(1+\kappa)^{-\frac{1}{2}}\ch(\gamma_{\tth,\kappa}+1)
  & \im \,r_6(\mu\e)
  & 0
  &  \\
\hline
0
  & -\im \,r_6(\mu\e)
  & \kappa\mu^2
  & \im \,r_9(\mu\e^2)
  & \\
-\im \,r_4(\mu\e^2)
  & 0
  & -\im \,r_9(\mu\e^2)
  & 0
  &  \\
\end{NiceArray}
\right)
\\&+ \mathcal{O}(\mu^2 \epsilon,\mu^3),
\end{aligned}
\end{equation}
\end{lemma}
\begin{proof}
    Recalling \eqref{expfe} and \eqref{SN1 g}, we write the operator $\mathcal{B}^{s}$ in \eqref{B sharp} as
    \begin{align}
        \mathcal{B}^s=-\kappa\mu\mathcal{B}^s_1- \kappa \mu^2\mathcal{B}^s_2,
    \end{align}
    where
    \begin{equation}\label{B s 1}
            \begin{aligned} 
        \mathcal{B}^s_1:=&\frac{\im}{1+\mathfrak{p}_x}\left(g_\e(x)\circ\pa_x+\pa_x\circ g_\e(x)\right)\frac{1}{1+\mathfrak{p}_x} \Lambda_{\tilde{\Pi}}\\
        =&\underbrace{2\im\begin{bmatrix}
 \pa_x & 0 \\
0 & 0
\end{bmatrix}}_{=:\mathcal{B}^s_{11}}+\underbrace{3\im\e\ch^{-2}\begin{bmatrix}
-\pa_x\circ \cos(x)-\cos(x)\circ\pa_x & 0 \\
0 & 0
\end{bmatrix}}_{=:\mathcal{B}^s_{12}}+\cO(\e^2),
    \end{aligned}
    \end{equation}
and    
    \begin{align} \label{B s 2}
        \mathcal{B}^s_2:=-\frac{g_{\e}(x)}{(1+\mathfrak{p}_x)^2} \Lambda_{\tilde{\Pi}}=-\Lambda_{\tilde{\Pi}}+\cO(\e), \quad \Lambda_{\tilde{\Pi}}:=\begin{bmatrix}
\mathrm{Id} & 0 \\
0 & 0
\end{bmatrix},
    \end{align}
with $\cO(\e),\, \cO(\e^2)\in\mathcal{L}(Y, X)$.

Using  \eqref{43 f+1}-\eqref{46 f-0}, we derive the following expansion:
\begin{equation} \label{Lambda f pm 10}
\begin{aligned}
  \mathcal{B}^s_{11} f^+_1(\mu,\e)&= 2\im \vet{-\ch^{\frac{1}{2}}(1+\kappa)^{-\frac{1}{4}}\sin(x)}{0}
- \frac{\mu}{2} \gamma_{\tth,\kappa}
\vet{\ch^{\frac{1}{2}}(1+\kappa)^{-\frac{1}{4}}\cos(x)}{0}
-4\im  \alpha_{\tth,\kappa}\epsilon
\begin{bmatrix}
 \sin(2x) \\
0
\end{bmatrix} \\
& \quad
+\cO(\e^2)\vet{\mathit{odd}(x)}{0}+ \mathcal{O}(\mu^2, \mu \epsilon),\\
 \mathcal{B}^s_{11}f^-_1(\mu,\e)&=2\im \vet{-\ch^{\frac{1}{2}}(1+\kappa)^{-\frac{1}{4}}\cos(x)}{0}
+ \frac{\mu}{2} \gamma_{\tth,\kappa}
\vet{\ch^{\frac{1}{2}}(1+\kappa)^{-\frac{1}{4}}\sin(x)}{0}
-4\im  \alpha_{\tth,\kappa}\epsilon
\begin{bmatrix}
 \cos(2x) \\
0
\end{bmatrix} \\
& \quad
+\cO(\e^2)\vet{\mathit{even}(x)}{0}+ \mathcal{O}(\mu^2, \mu \epsilon),\\
 \mathcal{B}^s_{11} f^+_0(\mu,\e)&=-2\im \e\delta_{\tth,\kappa}\vet{\ch^{\frac{1}{2}}(1+\kappa)^{-\frac{1}{4}}\sin(x)}{0}+\cO(\e^2)\vet{\mathit{odd}(x)}{0}+\cO(\mu\e), \quad 
  \mathcal{B}^s_{11} f^-_0(\mu,\e)=\cO(\mu\e).
\end{aligned}
\end{equation}
Using \eqref{Lambda f pm 10} and  the self-adjoint and reversibility-preserving properties of $\mathcal{B}^s_{11}$, the matrix
$-\kappa\mu\mathsf{B}^{s}_{11}$ representing the action of 
the operator $-\kappa\mu\mathcal{B}^s_{11}$ on the basis 
$\mathcal{F}$ of $\mathcal{V}_{\mu,\epsilon}$ in \eqref{F basis set and f}
admits the expansion
\begin{equation}\label{mathsf B s11} 
\begin{aligned} 
-\kappa\mu\mathsf{B}^{s}_{11} &=
\left(
\begin{NiceArray}{cc|ccc}[code-for-first-col=\scriptstyle]
\frac{\mu^2}{2}\kappa(1+\kappa)^{-\frac{1}{2}}\gamma_{\tth,\kappa} \ch
  & \im\,\kappa\mu\ch(1+\kappa)^{-\frac{1}{2}}+\im \,r(\mu\e^2)
  & 0
  & 0
  &  \\
-\im\,\kappa\mu\ch(1+\kappa)^{-\frac{1}{2}}+\im \,r(\mu\e^2)
  & \frac{\mu^2}{2}\kappa(1+\kappa)^{-\frac{1}{2}}\gamma_{\tth,\kappa} \ch
  & \im\, r(\mu\e)
  & 0
  &  \\
\hline
0
  & \im\, r(\mu\e)
  & 0
  & 0
  & \\
0
  & 0
  & 0
  & 0
  &  \\
\end{NiceArray}
\right) + \mathcal{O}(\mu^2 \epsilon,\mu^3).
\end{aligned}
\end{equation}
Consider now  $-\kappa\mu\mathcal{B}^{s}_{12}$, we have 
\begin{equation}
\left( -\kappa\mu\mathcal{B}^{s}_{12} f_k^\sigma (\mu, \epsilon),\, f_{k'}^{\sigma'}(\mu, \epsilon) \right)
= \left( -\kappa\mu\mathcal{B}^{s}_{12} f_k^\sigma (0, \epsilon),\, f_{k'}^{\sigma'}(0, \epsilon) \right)
+ \mathcal{O}(\mu^2 \epsilon).
\end{equation}

The matrix entries $\left( -\kappa\mu\mathcal{B}^{s}_{12} f_k^\sigma (0, \epsilon),\, f_{k'}^{\sigma}(0, \epsilon) \right),\, k, k' = 0, 1,\, \sigma = \pm$ are zero, because they are simultaneously real by \eqref{B are alternatively real or imaginary}, and purely imaginary, being the operator $-\kappa\mu\mathcal{B}^{s}_{12}$ purely imaginary and the basis $\{ f_k^{\pm}(0, \epsilon) \}_{k=0,1}$ real. Hence the matrix $-\kappa\mu\mathsf{B}^{s}_{12}$
 representing the action of 
the operator $-\kappa\mu\mathcal{B}^s_{11}$ on the basis 
$\mathcal{F}$ of $\mathcal{V}_{\mu,\epsilon}$
has the form
\begin{equation}
-\kappa\mu\mathsf{B}^{s}_{12} =
\left(
\begin{NiceArray}{cc|ccc}[code-for-first-col=\scriptstyle]
0 & \mathrm{i} a^{s}_{12} & 0 & \mathrm{i} b^{s}_{12} \\
-\mathrm{i} a^{s}_{12} & 0 & -\mathrm{i} c^{s}_{12} & 0 \\
\hline
0 & \mathrm{i} c^{s}_{12} & 0 & \mathrm{i} d^{s}_{12} \\
-\mathrm{i} b^{s}_{12} & 0 & -\mathrm{i} d^{s}_{12} & 0
\end{NiceArray}
\right)
+ \mathcal{O}(\mu^2 \epsilon),
\end{equation}
where
\begin{equation}
\begin{aligned}
& \left( -\kappa\mu\mathcal{B}^{s}_{12} f_1^{-}(0, \epsilon),\, f_1^{+}(0, \epsilon) \right) := \mathrm{i} a^{s}_{12}, \quad \left( -\kappa\mu\mathcal{B}^{s}_{12} f_1^{-}(0, \epsilon),\, f_0^{+}(0, \epsilon) \right) := \mathrm{i} c^{s}_{12}, \\
& \left( -\kappa\mu\mathcal{B}^{s}_{12} f_0^{-}(0, \epsilon),\, f_1^{+}(0, \epsilon) \right) := \mathrm{i} b^{s}_{12}, \quad \left( -\kappa\mu\mathcal{B}^{s}_{12} f_0^{-}(0, \epsilon),\, f_0^{+}(0, \epsilon) \right) := \mathrm{i} d^{s}_{12},
\end{aligned}
\end{equation}
and $a^{s}_{12}, b^{s}_{12}, c^{s}_{12}, d^{s}_{12}$ are real numbers. As $-\kappa\mu\mathcal{B}^{s}_{12} = \mathcal{O}(\mu \epsilon)$ in $\mathcal{L}(Y, X)$, we deduce that $c^{s}_{12} = r(\mu \epsilon)$. Let us compute the expansion of $a^{s}_{12}$, $b^{s}_{12}$ and $d^{s}_{12}$. In view of \eqref{43 f+1}-\eqref{46 f-0}, $f_1^{\pm}(0, \epsilon) = f_1^{\pm} + \mathcal{O}(\epsilon)$, $f_0^{+}(0, \epsilon) = f_0^{+} + \mathcal{O}(\epsilon)$, $f_0^{-}(0, \epsilon) = \begin{bmatrix} 0 \\ 1 \end{bmatrix}$, where $f_k^\sigma$ are in \eqref{eigenfunc of mathcall L00 2}. By \eqref{B s 1} we have

\[ 
-\kappa\mu\mathcal{B}^{s}_{12} f_1^{-} =
\begin{bmatrix}
-\frac{3}{2}\im\kappa\mu\e\ch^{-\frac{3}{2}}(1+\kappa)^{-\frac{1}{4}}(1+3\cos(2x)) \\
0
\end{bmatrix},
\quad
-\kappa\mu\mathcal{B}^{s}_{12} f_0^{-} =
\begin{bmatrix}
0 \\
0
\end{bmatrix},
\]
and then
\begin{align*} 
a^{s}_{12} &=  \left( -\kappa\mu\mathcal{B}_{12}^{s} f_1^{-},\, f_1^{+} \right) + r(\mu \epsilon^2) = r(\mu \epsilon^2), \quad 
b^{s}_{12} =  \left( -\kappa\mu\mathcal{B}_{12}^{s} f_0^{-},\, f_1^{+} \right) + r(\mu \epsilon^2) =  r(\mu \epsilon^2), \\
d^{s}_{12} &=  \left( -\kappa\mu\mathcal{B}_{12}^{s} f_0^{-},\, f_0^{+} \right) + r(\mu \epsilon^2) = r(\mu \epsilon^2).
\end{align*}
\color{black}
Next we consider the matrix
$-\kappa\mu^2\mathsf{B}^{s}_2$ representing the action of the operator
$ -\kappa\mu^2\mathcal{B}^{s}_2$ on the basis 
$\mathcal{F}$.
Recalling \eqref{B s 2}, we have $-\kappa\mu^2\mathcal{B}^{s}_2 = \kappa \mu^2\Lambda_{\tilde{\Pi}} + \cO(\mu^2,\e) $, hence one has 
\begin{equation}
\left( -\kappa\mu^2\mathcal{B}^{s}_2 f_k^\sigma (\mu, \epsilon),\, f_{k'}^{-\sigma'}(\mu, \epsilon) \right)
= \left( \kappa\mu^2\Lambda_{\tilde{\Pi}} f_k^\sigma (0, \epsilon),\, f_{k'}^{\sigma'}(0, \epsilon) \right)
+ \mathcal{O}(\mu^2\e,\mu^3).
\end{equation}
The matrix entries $\left( -\kappa\mu^2\mathcal{B}^{s}_2 f_k^\sigma (0, \epsilon),\, f_{k'}^{-\sigma}(0, \epsilon) \right),\, k, k' = 0, 1,\, \sigma = \pm$ are zero, because they are simultaneously purely imaginary by \eqref{B are alternatively real or imaginary}, and real, being the operator $\mathcal{B}^{s}_2$ real and the basis $\{ f_k^{\pm}(0, \epsilon) \}_{k=0,1}$ real. Hence $-\kappa\mu^2\mathcal{B}^{s}_2$, with respect to the basis $\mathcal{F}$ of $\mathcal{V}_{\mu,\epsilon}$ in \eqref{F basis set and f}, admits the expansion 
\begin{equation}
-\kappa\mu^2\mathsf{B}^{s}_2 =
\left(
\begin{NiceArray}{cc|ccc}[code-for-first-col=\scriptstyle]
a^s_2 & 0 & b^s_2 & 0 \\
0 & c^s_2 & 0 & d^s_2 \\
\hline
b^s_2 & 0 & e^s_2 & 0 \\
0 & d^s_2 & 0 & f^s_2
\end{NiceArray}
\right)
+ \mathcal{O}(\mu^2 \epsilon,\mu^3),
\end{equation}
where
\begin{align*}
a^{s}_2 &= \left( \kappa\mu^2\Lambda_{\tilde{\Pi}} f_1^+ (0, \epsilon),\, f_{1}^{+}(0, \epsilon) \right) =\left( \kappa\mu^2\Lambda_{\tilde{\Pi}} f_1^+ ,\, f_{1}^{+} \right)
+r(\mu^2\e)=\frac{\mu^2}{2}\kappa(1+\kappa)^{-\frac{1}{2}}\ch+r(\mu^2\e), \\
b^{s}_2 &:=\left( \kappa\mu^2\Lambda_{\tilde{\Pi}} f_1^+ (0, \epsilon),\, f_{0}^{+}(0, \epsilon) \right) = \left( \kappa\mu^2\Lambda_{\tilde{\Pi}} f_1^+ ,\, f_{0}^{+} \right)+r(\mu^2\e)=r(\mu^2\e),\\
c^{s}_2 &:= \left( \kappa\mu^2\Lambda_{\tilde{\Pi}} f_1^- (0, \epsilon),\, f_{1}^{-}(0, \epsilon) \right) = \left( \kappa\mu^2\Lambda_{\tilde{\Pi}} f_1^- ,\, f_{1}^{-} \right)+r(\mu^2\e)=\frac{\mu^2}{2}\kappa(1+\kappa)^{-\frac{1}{2}}\ch+r(\mu^2\e), \\
d^{s}_2 &
:=  \left( \kappa\mu^2\Lambda_{\tilde{\Pi}} f_1^- (0, \epsilon),\, f_{0}^{-}(0, \epsilon) \right)
= \left( \kappa\mu^2\Lambda_{\tilde{\Pi}} f_1^- ,\, f_{0}^{-} \right)+r(\mu^2\e)=r(\mu^2\e),\\
e^{s}_2 &:=\left( \kappa\mu^2\Lambda_{\tilde{\Pi}} f_0^+ (0, \epsilon),\, f_{0}^{+}(0, \epsilon) \right) =\left( \kappa\mu^2\Lambda_{\tilde{\Pi}} f_0^+,\, f_{0}^{+} \right)+r(\mu^2\e)=\kappa\mu^2+r(\mu^2\e),\\
f^{s}_2 &:=\left( \kappa\mu^2\Lambda_{\tilde{\Pi}} f_0^- (0, \epsilon),\, f_{0}^{-}(0, \epsilon) \right) =\left( \kappa\mu^2\Lambda_{\tilde{\Pi}} f_0^- ,\, f_{0}^{-} \right)+r(\mu^2\e)=r(\mu^2\e).
\end{align*}
This proves \eqref{mathsf B s}. 
\end{proof}
Lemma \ref{Expansion of B epsilon}-Lemma \ref{Expansion of B s} imply \eqref{B mu epsilon} where the matrix $E$ has the form \eqref{E} and
\[
\mathsf{e}_{22} := 2(\mathbf{b}_{\tth,\kappa} - 4\zeta_{\tth,\kappa}-2\kappa(1+\kappa)^{-\frac{1}{2}}\ch(1+\gamma_{\tth,\kappa})) 
\]
with $\mathbf{b}_{\tth,\kappa}$ in \eqref{bh440} and $\zeta_{\tth,\kappa}$ in \eqref{zeta h}, $ \gamma_{\tth,\kappa}$ in \eqref{47 coefficients}. The term $\mathsf{e}_{22}$ has the expansion in \eqref{e11 f11}. 

Moreover,
\begin{equation} \label{F 2nd}
\begin{aligned} 
F &:= F(\mu,\e)=
\begin{pmatrix}
\mathsf{f}_{11} \epsilon + r_3(\epsilon^3, \mu \epsilon^2, \mu^2 \epsilon,\mu^3) &
\im \,\mu \epsilon \ch^{-\frac{1}{2}}(1+\kappa)^{\frac{1}{4}} + \im\, r_4(\mu \epsilon^2, \mu^2 \epsilon,\mu^3) \\
\im\, r_6(\mu \epsilon,\mu^3) &
r_7(\mu^2 \epsilon,\mu^3)
\end{pmatrix}, 
\end{aligned}
\end{equation}

\begin{equation} \label{G 2nd}
\begin{aligned} 
G &:= G(\mu,\e)=
\begin{pmatrix}
1 +\kappa\mu^2+ r_8(\epsilon^2, \mu^2 \epsilon,\mu^3) &
- \im r_9(\mu \epsilon^2, \mu^2 \epsilon,\mu^3) \\
\im r_9(\mu \epsilon^2, \mu^2 \epsilon,\mu^3) &
\mu \tanh(h \mu) + r_{10}(\mu^2 \epsilon,\mu^3)
\end{pmatrix} . 
\end{aligned}
\end{equation}
In order to deduce the expansion \eqref{F}-\eqref{G} of the matrices $F$, $G$ we exploit also the following lemma:
\begin{lemma} \label{Fmu0 eq 0 Gmu0}
At \( \epsilon = 0 \) the matrices are \( F(\mu, 0) = 0 \) and 
\(
G(\mu, 0) = \begin{pmatrix} 1+\kappa\mu^2 & 0 \\ 0 & \mu \tanh(\tth\mu) \end{pmatrix}.
\)
\end{lemma}

\begin{proof}
By Lemma \ref{fmu0 eq f0} and since $\mathcal{B}_{\mu,0} := \left[\begin{smallmatrix}
1-\kappa\left(\pa_x+\im \mu\right)^2 & -\chk \partial_x \\
\chk\pa_x &|D + \mu| \tanh(\tth|D + \mu|)  
\end{smallmatrix}\right]$,  we obtain 
$\mathcal{B}_{\mu,0} f_0^+(\mu,0) = (1+\kappa\mu^2)f_0^+$ and 
$\mathcal{B}_{\mu,0} f_0^-(\mu,0) = \mu \tanh(\tth\mu) f_0^-$,
for any \( \mu \). The lemma follows recalling \eqref{Bmuepsilon matrix representation} and the fact that 
\( f_1^+(\mu, 0) \) and \( f_1^-(\mu, 0) \) have zero space average by Lemma \ref{fmu0 eq f0}
\end{proof}

In view of Lemma \ref{Fmu0 eq 0 Gmu0} we deduce that the matrices \eqref{F 2nd} and \eqref{G 2nd} have the form \eqref{F} and \eqref{G}. This completes the proof of Lemma \ref{B decomposition}.

\section{Block-Decoupling}\label{sec:BD}
In this section we block-decouple the $4 \times 4$ Hamiltonian matrix $\mathsf{L}_{\mu,\epsilon} = \mathsf{J}_4 \mathsf{B}_{\mu,\epsilon}$, where $\mathsf{B}_{\mu,\epsilon}$ obtained in Lemma \ref{B decomposition}.
We first perform a singular symplectic and reversibility-preserving change of coordinates.
The proof is a simple computation and we skip it.
\begin{lemma} [Singular symplectic rescaling] The conjugation of the Hamiltonian and reversible matrix $\mathsf{L}_{\mu,\epsilon} = \mathsf{J}_4 \mathsf{B}_{\mu,\epsilon}$ obtained in Lemma \ref{B decomposition} through the symplectic and reversibility-preserving $4 \times 4$ matrix
\begin{align}
    Y := \begin{pmatrix} Q & 0 \\ 0 & Q \end{pmatrix} 
\quad \text{with} \quad 
Q := \begin{pmatrix} \mu^{\frac{1}{2}} & 0 \\ 0 & \mu^{-\frac{1}{2}} \end{pmatrix}, \quad \mu > 0,
\end{align}
yields the Hamiltonian and reversible matrix
\begin{align} \label{L1mue YLY}
    \mathsf{L}_{\mu,\epsilon}^{(1)} := Y^{-1} \mathsf{L}_{\mu,\epsilon} Y 
= \mathsf{J}_4 \mathsf{B}_{\mu, \epsilon}^{(1)} 
= \begin{pmatrix}
\mathsf{J}_2 E^{(1)} & \mathsf{J}_2 F^{(1)} \\
\mathsf{J}_2 [F^{(1)}]^* & \mathsf{J}_2 G^{(1)}
\end{pmatrix}
\end{align}
where $\mathsf{B}_{\mu, \epsilon}^{(1)}$ is a self-adjoint and reversibility-preserving $4 \times 4$ matrix
\begin{align}
    \mathsf{B}_{\mu, \epsilon}^{(1)} = 
\begin{pmatrix}
E^{(1)} & F^{(1)} \\
[F^{(1)}]^* & G^{(1)}
\end{pmatrix}, 
\quad 
E^{(1)} = [E^{(1)}]^*, \quad G^{(1)} = [G^{(1)}]^*,
\end{align}
where the $2 \times 2$ reversibility-preserving matrices $E^{(1)}$, $F^{(1)}$, and $G^{(1)}$ extend analytically at $\mu = 0$ with the following expansion
\begin{align} \label{E1}
E^{(1)} &= 
\begin{pmatrix}
\mathsf{e}_{11} \mu \e^2 (1 + \hat{r}_1(\e, \mu \epsilon)) - \mathsf{e}_{22} \frac{\mu^3}{8} (1 + \check{r}_1(\e, \mu)) & \im\, \left( \frac{1}{2} \mathsf{e}_{12} \mu + r_2(\mu \epsilon^2, \mu^2 \epsilon, \mu^3) \right) \\
-\im\, \left( \frac{1}{2} \mathsf{e}_{12} \mu + r_2(\mu \epsilon^2, \mu^2 \epsilon, \mu^3) \right) & -\mathsf{e}_{22} \frac{\mu}{8} (1 + r_5(\epsilon, \mu))
\end{pmatrix},
\\ \label{F1}
F^{(1)} &= 
\begin{pmatrix}
\mathsf{f}_{11} \mu \e + r_3(\mu \epsilon^3, \mu^2 \epsilon^2, \mu^3 \epsilon) & \im\, \mu \epsilon \ch^{-\frac{1}{2}}(1+\kappa)^{\frac{1}{4}} + \im\, r_4(\mu \epsilon^2, \mu^2 \epsilon) \\
\im\, r_6(\mu \epsilon) & r_7(\mu \epsilon)
\end{pmatrix},
\\ \label{G1}
G^{(1)} &= 
\begin{pmatrix}
\mu +\kappa\mu^3+ r_8(\mu \epsilon^2, \mu^3 \epsilon) & -\im\, r_9(\mu \epsilon^2, \mu^2 \epsilon) \\
\im\, r_9(\mu \epsilon^2, \mu^2 \epsilon) & \tanh(\tth \mu) + r_{10}(\mu \epsilon)
\end{pmatrix},
\end{align}
where $\mathsf{e}_{11}, \mathsf{f}_{11}, \mathsf{e}_{12}, \mathsf{e}_{22}, $ are defined in \eqref{e11 f11}, \eqref{def:D}, and \eqref{def:e22}.
\end{lemma} 

\begin{remark}
    The matrix $\mathsf{L}_{\mu, \epsilon}^{(1)}$, a priori defined only for $\mu \neq 0$, extends analytically to the zero matrix at $\mu = 0$. For $\mu \neq 0$ the spectrum of $\mathsf{L}_{\mu, \epsilon}^{(1)}$ coincides with the spectrum of $\mathsf{L}_{\mu, \epsilon}$.
\end{remark} 

\subsection{Non-perturbative Step of Block-Decoupling} 

We state the main result of this section.
\begin{lemma} [Step of block-decoupling]\label{Step of block-decoupling}
Let $(\kappa, \tth) \not\in \mathfrak{D}$, cf.  \eqref{def:fD}.
     There exists a $2 \times 2$ reversibility-preserving matrix $X$, analytic in $(\mu, \epsilon)$, of the form
\begin{equation} \label{XX}
    \begin{aligned}
    X :=& \begin{pmatrix}
x_{11} & \im \,x_{12} \\
\im \,x_{21} & x_{22}
\end{pmatrix}, \quad \text{with } x_{ij} \in \mathbb{R}, \ i, j = 1, 2,\\
=& \begin{pmatrix}
r_{11}(\epsilon) & \im\, r_{12}(\epsilon) \\
-\im\,\tfrac{1}{2} \textup{D}_{\tth,\kappa}^{-1}(\mathsf{e}_{12} \mathsf{f}_{11} + 2\ch^{-\frac{1}{2}}(1+\kappa)^{\frac{1}{4}})\epsilon + \im \, r_{21}(\epsilon^2, \mu \epsilon) & \tfrac{1}{2} \textup{D}_{\tth,\kappa}^{-1}(\ch^{-\frac{1}{2}}(1+\kappa)^{\frac{1}{4}} \mathsf{e}_{12} + 2\tth \mathsf{f}_{11})\epsilon + r_{22}(\epsilon^2, \mu \epsilon)
\end{pmatrix}, 
\end{aligned}
\end{equation}
where $\mathsf{e}_{12}$, $\textup{D}_{\tth,\kappa}$ and $\mathsf{f}_{11}$ are defined in \eqref{def:D} and \eqref{e11 f11}, respectively, such that the following holds true. By conjugating the Hamiltonian and reversible matrix $\mathsf{L}^{(1)}_{\mu, \epsilon}$ defined in \eqref{L1mue YLY}, with the symplectic and reversibility-preserving $4 \times 4$ matrix
\begin{align} \label{S1}
    \exp\left(S^{(1)}\right), \quad \text{where } \quad S^{(1)} := \mathsf{J}_4 \begin{pmatrix}
0 & M \\
M^* & 0
\end{pmatrix}, \quad M := \mathsf{J}_2 X, 
\end{align}
we get the Hamiltonian and reversible matrix
\begin{align} \label{L2mueps}
\mathsf{L}^{(2)}_{\mu, \epsilon} := \exp\left(S^{(1)}\right) \mathsf{L}^{(1)}_{\mu, \epsilon} \exp\left(-S^{(1)}\right) = \mathsf{J}_4 \mathsf{B}^{(2)}_{\mu, \epsilon} = \mathsf{J}_4 \begin{pmatrix}
\mathsf{J}_2 E^{(2)} & \mathsf{J}_2 F^{(2)} \\
\mathsf{J}_2 [F^{(2)}]^* & \mathsf{J}_2 G^{(2)}
\end{pmatrix}, 
\end{align}
where the reversibility-preserving $2 \times 2$ self-adjoint matrix $[E^{(2)}]^* = E^{(2)}$ has the form
\begin{align} \label{E2 in lemma}
E^{(2)} = \begin{pmatrix}
\mu \epsilon^2 \mathsf{e}_{\mathrm{WB}} + r_1'(\mu \epsilon^3, \mu^2 \epsilon^2) - \mathsf{e}_{22} \tfrac{\mu^3}{8}(1 + r_1''(\epsilon, \mu)) & \im\,\left( \frac{1}{2} \mathsf{e}_{12} \mu + r_2(\mu \epsilon^2, \mu^2 \epsilon, \mu^3) \right) \\
-\im\,\left( \frac{1}{2}  \mathsf{e}_{12} \mu + r_2(\mu \epsilon^2, \mu^2 \epsilon, \mu^3) \right) & -\mathsf{e}_{22} \frac{\mu}{8}(1 + r_5(\epsilon, \mu))
\end{pmatrix}, 
\end{align}
where
\begin{align} \label{eWB}
\mathsf{e}_{\mathrm{WB}}:= \mathsf{e}_{\mathrm{WB}}(\kappa, \tth) := \mathsf{e}_{11} - \textup{D}_{\tth,\kappa}^{-1}\left(\ch^{-1}(1+\kappa)^{\frac{1}{2}}  + \tth \mathsf{f}^2_{11} + \mathsf{e}_{12} \mathsf{f}_{11} \ch^{-\frac{1}{2}}(1+\kappa)^{\frac{1}{4}}\right). 
\end{align}
(with constants $\mathsf{e}_{11}, , \mathsf{f}_{11}, \mathsf{e}_{12}$ defined in \eqref{e11 f11}), is the Whitham–Benjamin function defined in \eqref{def:eWB}. The reversibility-preserving $2 \times 2$ self-adjoint matrix $[G^{(2)}]^* = G^{(2)}$ has the form
\begin{align} \label{G2}
G^{(2)} = \begin{pmatrix}
{\mu(1+\kappa \mu^2)} + r_8(\mu \epsilon^2, \mu^3\e ) & -\im r_9(\mu \epsilon^2, \mu^2 \epsilon) \\
\im r_9(\mu \epsilon^2, \mu^2 \epsilon) & \tanh(\tth \mu) + r_{10}(\mu \epsilon)
\end{pmatrix}, 
\end{align}
and
\begin{align} \label{F2}
F^{(2)} = \begin{pmatrix}
r_3(\mu \epsilon^3) & \im r_4(\mu \epsilon^3) \\
\im r_6(\mu \epsilon^3) & r_7(\mu \epsilon^3)
\end{pmatrix}. 
\end{align}
\end{lemma}

The rest of the section is devoted to the proof of Lemma \ref{Step of block-decoupling}. For simplicity let $S = S^{(1)}$.

The matrix $\exp(S)$ is symplectic and reversibility-preserving because the matrix $S$ in \eqref{S1} is Hamiltonian and reversibility-preserving, cf. \cite[Lemma 3.8]{BMV1}. Note that $S$ is reversibility preserving, since $X$ has the form \eqref{XX}.



We now expand in Lie series the Hamiltonian and reversible matrix $
\mathsf{L}_{\mu,\epsilon}^{(2)} = \exp(S) \mathsf{L}_{\mu,\epsilon}^{(1)} \exp(-S)$.
We split $\mathsf{L}^{(1)}_{\mu,\epsilon}$ into its $2 \times 2$-diagonal and off-diagonal Hamiltonian and reversible matrices
\begin{equation} \label{L=D+R}
    \begin{aligned}
        &\mathsf{L}^{(1)}_{\mu,\epsilon} = D^{(1)} + R^{(1)},\\
&D^{(1)} := \begin{pmatrix} D_1 & 0 \\ 0 & D_0 \end{pmatrix}= \begin{pmatrix} \mathsf{J}_2 E^{(1)} & 0 \\ 0 & \mathsf{J}_2 G^{(1)} \end{pmatrix}, \quad 
R^{(1)} := \begin{pmatrix} 0 & \mathsf{J}_2 F^{(1)} \\ \mathsf{J}_2 [F^{(1)}]^* & 0 \end{pmatrix}, 
    \end{aligned}
\end{equation}
and perform the Lie expansion
\begin{equation} \label{L2 mu e}
    \begin{aligned}
        \mathsf{L}^{(2)}_{\mu,\epsilon} = \exp(S) \mathsf{L}^{(1)}_{\mu,\epsilon} \exp(-S) &= D^{(1)} + [S, D^{(1)}] + \frac{1}{2} [S, [S, D^{(1)}]] + R^{(1)} + [S, R^{(1)}]\\
        &+ \frac{1}{2} \int_0^1 (1 - \tau)^2 \exp(\tau S) \mathrm{ad}_S^3(D^{(1)}) \exp(-\tau S) \, d\tau\\
        &+ \int_0^1 (1 - \tau) \exp(\tau S) \mathrm{ad}_S^2(R^{(1)}) \exp(-\tau S) \, d\tau.
    \end{aligned}
\end{equation}

We look for a $4 \times 4$ matrix $S$ as in \eqref{S1} (which is Hamiltonian, reversibility-preserving and off-diagonal as the term $R^{(1)}$ we wish to eliminate) that solves the homological equation 
\[
R^{(1)} + [S, D^{(1)}] = 0,
\]
which, recalling \eqref{L=D+R}, reads
\begin{align} \label{before DX-XD=-JF}
\begin{pmatrix}
0 & \mathsf{J}_2 F^{(1)} + \mathsf{J}_2 M D_0 - D_1 \mathsf{J}_2 M \\
\mathsf{J}_2 [F^{(1)}]^* + \mathsf{J}_2 M^* D_1 - D_0 \mathsf{J}_2 M^* & 0
\end{pmatrix} = 0. 
\end{align}

Note that the equation $\mathsf{J}_2 F^{(1)} + \mathsf{J}_2 M D_0 - D_1 \mathsf{J}_2 M = 0$ implies also $\mathsf{J}_2 [F^{(1)}]^* + \mathsf{J}_2 M^* D_1 - D_0 \mathsf{J}_2 M^* = 0$ and vice versa, since $\mathsf{J}^*_2=-\mathsf{J}_2$ and $\mathsf{J}_2(\mathsf{J}_2 F^{(1)} + \mathsf{J}_2 M D_0 - D_1 \mathsf{J}_2 M)=(\mathsf{J}_2 [F^{(1)}]^* + \mathsf{J}_2 M^* D_1 - D_0 \mathsf{J}_2 M^*)^*\mathsf{J}^*_2$. Thus, writing 
\[
M = \mathsf{J}_2 X, \quad \text{namely} \quad X = -\mathsf{J}_2 M,
\]
the equation \eqref{before DX-XD=-JF} amounts to solving the Sylvester equation
\begin{align} \label{DX-XD=JF}
D_1 X - X D_0 = -\mathsf{J}_2 F^{(1)}. 
\end{align}

We write the matrices $E^{(1)}$, $F^{(1)}$, $G^{(1)}$ in \eqref{L1mue YLY} as
\begin{align} \label{E1 F1 G1}
    E^{(1)} = \begin{pmatrix} E^{(1)}_{11} & \im E^{(1)}_{12} \\ -\im E^{(1)}_{12} & E^{(1)}_{22} \end{pmatrix}, \quad 
F^{(1)} = \begin{pmatrix} F^{(1)}_{11} & \im F^{(1)}_{12} \\ \im F^{(1)}_{21} & F^{(1)}_{22} \end{pmatrix}, \quad G^{(1)} = \begin{pmatrix} G^{(1)}_{11} & \im G^{(1)}_{12} \\ -\im G^{(1)}_{12} & G^{(1)}_{22} \end{pmatrix},
\end{align}
where the real numbers $E_{ij}^{(1)}, F_{ij}^{(1)}, G_{ij}^{(1)}, i, j = 1, 2,$ have the expansion in \eqref{E1}, \eqref{F1}, \eqref{G1}. Thus, by \eqref{L=D+R}, \eqref{XX} and \eqref{E1 F1 G1}, the equation \eqref{DX-XD=JF} amounts to solve the $4 \times 4$ real linear system

\begin{align} \label{Ax=f}
\underbrace{
\begin{pmatrix}
a & b & c & 0 \\
d & a & 0 & -c \\
e & 0 & a & -b \\
0 & -e & -d & a
\end{pmatrix}
}_{=: \mathsf{A}}
\underbrace{
\begin{pmatrix}
x_{11} \\ x_{12} \\ x_{21} \\ x_{22}
\end{pmatrix}
}_{=: \vec{x}}
=
\underbrace{
\begin{pmatrix}
- F_{21}^{(1)} \\ F_{22}^{(1)} \\ -F_{11}^{(1)} \\ F_{12}^{(1)}
\end{pmatrix}
}_{=: \vec{f}}.
\end{align}
where by \eqref{E1}, \eqref{F1}, and \eqref{G1} and since 
$
\tanh(\tth \mu) = \tth \mu + r(\mu^3),
$
\begin{equation} \label{a b c d e}
\begin{aligned}
a &= G^{(1)}_{12} - E^{(1)}_{12} = -\mathsf{e}_{12} \frac{\mu}{2} \big( 1 + r(\epsilon^2, \mu \epsilon, \mu^2) \big), 
\quad b = G^{(1)}_{11} = \mu + r(\mu \epsilon^2, \mu^3), \\
c &= E^{(1)}_{22} = -\mathsf{e}_{22} \frac{\mu}{8} (1 + r_5(\epsilon, \mu)), 
\quad d = G^{(1)}_{22} = \mu \tth + r(\mu \epsilon, \mu^3), \\
e &= E^{(1)}_{11} = r(\mu \epsilon^2, \mu^3),
\end{aligned}
\end{equation}
where $\mathsf{e}_{12}>0$ for any $(\kappa,\tth)\in\mathbb{R}^2_+$ and $\mathsf{e}_{22}$ is defined in \eqref{def:e22}.
Using also  \cite[Lemma 5.5]{BMV3}, one has 
\begin{equation} \label{det A}
\begin{aligned}
\det \mathsf{A} 
= (bd - a^2)^2 - 2ce \Big( a^2 + bd - \tfrac{1}{2} ce \Big) = \mu^4 \textup{D}_{\tth,\kappa}^2 (1 + r(\epsilon, \mu^2)) 
\end{aligned}    
\end{equation}
which is not zero provided  $\mu \neq 0$ and  $(\kappa, \tth) \not\in \mathfrak{D}$ (see \eqref{def:fD}).
The inverse of $\mathsf{A} $ is computed in \cite[formula (5.23)]{BMV3} as 
\begin{align} \label{A inverse}
\mathsf{A}^{-1} = \frac{1}{\det\mathsf{A} }
\begin{pmatrix}
a (a^2 - bd - ce) & b(-a^2 + bd - ce) & -c(a^2 + bd - ce) & -2abc \\[6pt]
d(-a^2 + bd - ce) & a(a^2 - bd - ce) & 2acd & -c(-a^2 - bd + ce) \\[6pt]
-e(a^2 + bd - ce) & 2abe & a(a^2 - bd - ce) & b(a^2 - bd + ce) \\[6pt]
-2ade & -e(-a^2 - bd + ce) & d(a^2 - bd + ce) & a(a^2 - bd - ce)
\end{pmatrix}
\end{align}
and expands as 
\begin{align} \label{mathsf A -1}
\mathsf{A}^{-1} = (1 + r(\epsilon, \mu)) 
\frac{1}{\mu \textup{D}_{\tth,\kappa}^2}\left(
\begin{array}{cccc}
\frac{1}{2} \mathsf{e}_{12} \textup{D}_{\tth,\kappa} & \textup{D}_{\tth,\kappa} & \frac{1}{32} \mathsf{e}_{22}(\mathsf{e}_{12}^2 + 4\tth) & -\frac{1}{8} \mathsf{e}_{12} \mathsf{e}_{22} \\
\tth \textup{D}_{\tth,\kappa} & \frac{1}{2} \mathsf{e}_{12} \textup{D}_{\tth,\kappa} & \frac{1}{8} \mathsf{e}_{12} \mathsf{e}_{22} \tth & -\frac{1}{32} \mathsf{e}_{22} (\mathsf{e}_{12}^2 + 4\tth) \\
r(\epsilon^2, \mu^2) & r(\epsilon^2, \mu^2) & \frac{1}{2} \mathsf{e}_{12} \textup{D}_{\tth,\kappa} & -\textup{D}_{\tth,\kappa} \\
r(\epsilon^2, \mu^2) & r(\epsilon^2, \mu^2) & -\tth \textup{D}_{\tth,\kappa} & \frac{1}{2} \mathsf{e}_{12} \textup{D}_{\tth,\kappa} \\
\end{array}
\right).
\end{align}
Therefore, for any $\mu \neq 0$, there exists a unique solution $\vec{x} = \mathsf{A}^{-1} \vec{f}$ of the linear system \eqref{Ax=f}, namely a unique matrix $X$ which solves the Sylvester equation \eqref{DX-XD=JF}.
Explicitly, 
    by \eqref{Ax=f}, \eqref{mathsf A -1}, \eqref{E1 F1 G1}, \eqref{F1} we obtain, for any $\mu \neq 0$,
    \small
\begin{equation}
\begin{aligned}
\begin{pmatrix}
x_{11} \\ x_{12} \\ x_{21} \\ x_{22}
\end{pmatrix}
&=\textup{D}_{\tth,\kappa}^{-2}
\left(
\begin{array}{cccc}
\frac{1}{2} \mathsf{e}_{12} \textup{D}_{\tth,\kappa} & \textup{D}_{\tth,\kappa} & \frac{1}{32} \mathsf{e}_{22}(\mathsf{e}_{12}^2 + 4\tth) & -\frac{1}{8} \mathsf{e}_{12} \mathsf{e}_{22} \\
\tth \textup{D}_{\tth,\kappa} & \frac{1}{2} \mathsf{e}_{12} \textup{D}_{\tth,\kappa} & \frac{1}{8} \mathsf{e}_{12} \mathsf{e}_{22} \tth & -\frac{1}{32} \mathsf{e}_{22} (\mathsf{e}_{12}^2 + 4\tth) \\
r(\epsilon^2, \mu^2) & r(\epsilon^2, \mu^2) & \frac{1}{2} \mathsf{e}_{12} \textup{D}_{\tth,\kappa} & -\textup{D}_{\tth,\kappa} \\
r(\epsilon^2, \mu^2) & r(\epsilon^2, \mu^2) & -\tth \textup{D}_{\tth,\kappa} & \frac{1}{2} \mathsf{e}_{12} \textup{D}_{\tth,\kappa} \\
\end{array}
\right)
\left(
\begin{array}{c}
r(\epsilon) \\[6pt]
r(\epsilon) \\[6pt]
- \mathsf{f}_{11} \epsilon + r(\epsilon^3, \mu \epsilon^2, \mu^2 \epsilon) \\[6pt]
\ch^{-\frac{1}{2}}(1+\kappa)^{\frac{1}{4}} \epsilon + r(\epsilon^2, \mu \epsilon)
\end{array}
\right)
\big(1 + r(\epsilon, \mu)\big),
\end{aligned}
\end{equation}
\normalsize
which proves \eqref{XX}. In particular each $x_{ij}$ admits an analytic extension at $\mu = 0$.  
Note that, for $\mu = 0$, one has $E^{(2)} = G^{(2)} = F^{(2)} = 0$ and the Sylvester equation reduces to tautology.

Since the matrix $S$ solves the homological equation $R^{(1)}+[S, D^{(1)}]  = 0$, identity \eqref{L2 mu e} simplifies to
\begin{align} \label{L2 mu epsilon}
\mathsf{L}_{\mu,\epsilon}^{(2)} = D^{(1)} + \frac{1}{2} [ S, R^{(1)} ]
+ \frac{1}{2} \int_0^1 (1 - \tau^2)\, \exp(\tau S) \, \mathrm{ad}_S^2 (R^{(1)}) \, \exp(-\tau S) \, d\tau.
\end{align}
This follows by using the Jacobi identity and the homological relation to cancel lower-order terms and to rewrite
\[
\frac{1}{2} [S, [S, D^{(1)}]] + [S, R^{(1)}] = \frac{1}{2} [S, R^{(1)}],
\quad
\operatorname{ad}_S^3 (D^{(1)}) = -\operatorname{ad}_S^2 (R^{(1)}).
\]
The matrix $\frac{1}{2}[S, R^{(1)}]$ is, by \eqref{S1}, \eqref{L=D+R}, the block-diagonal Hamiltonian and reversible matrix
\begin{align} \label{0.5 S R}
\frac{1}{2} [ S, R^{(1)} ]
=
\begin{pmatrix}
\frac{1}{2} \mathsf{J}_2 ( 
M \mathsf{J}_2 [F^{(1)}]^* - F^{(1)} \mathsf{J}_2 M^* ) & 0 \\
0 & \frac{1}{2} \mathsf{J}_2 ( M^* \mathsf{J}_2 F^{(1)} - [F^{(1)}]^* \mathsf{J}_2 M )
\end{pmatrix}
=
\begin{pmatrix}
\mathsf{J}_2 \widetilde{E} & 0 \\
0 & \mathsf{J}_2 \widetilde{G}
\end{pmatrix},
\end{align}
where, recalling \eqref{S1},
\begin{align} \label{E G tilde}
\widetilde{E} := \mathrm{Sym} ( \mathsf{J}_2 X \mathsf{J}_2 [F^{(1)}]^* ), 
\quad
\widetilde{G} := \mathrm{Sym} ( X^* F^{(1)} ),
\end{align}
denoting $\mathrm{Sym}(A) := \frac{1}{2}( A + A^* )$.
\begin{lemma} \label{EG e11 tilde lemma}
    The self-adjoint and reversibility-preserving matrices $\widetilde{E}$, $\widetilde{G}$ in \eqref{E G tilde} have the form
    \begin{equation}\label{tilde EG and tilde e11}
    \begin{aligned} 
\widetilde{E} &= \begin{pmatrix}
\tilde{\mathsf{e}}_{11} \mu \epsilon^2 + \tilde{r}_1 (\mu \epsilon^3, \mu^2 \epsilon^2) & \im\, \tilde{r}_2 (\mu \epsilon^2) \\
- \im \,\tilde{r}_2 (\mu \epsilon^2) & \tilde{r}_5 (\mu \epsilon^2)
\end{pmatrix}, \quad
\widetilde{G} = \begin{pmatrix}
\tilde{r}_8 (\mu \epsilon^2) & \im\, \tilde{r}_9 (\mu \epsilon^2) \\
- \im\, \tilde{r}_9 (\mu \epsilon^2) & \tilde{r}_{10} (\mu \epsilon^2)
\end{pmatrix}, \\
\tilde{\mathsf{e}}_{11} &:= - \textup{D}_{\tth,\kappa}^{-1} \left( \ch^{-1}(1+\kappa)^{\frac{1}{2}} + \tth \mathsf{f}_{11}^2 + \mathsf{e}_{12} \mathsf{f}_{11} \ch^{-\frac{1}{2}}(1+\kappa)^{\frac{1}{4}} \right).
\end{aligned}
 \end{equation}
\end{lemma}
\begin{proof}
    By \eqref{F1}, \eqref{XX}, one has
\begin{equation*} 
\begin{aligned}
\mathsf{J}_2 X \mathsf{J}_2 [F^{(1)}]^* 
&= \begin{pmatrix}
x_{21} F^{(1)}_{12} - x_{22} F^{(1)}_{11} & \im ( x_{21} F^{(1)}_{22} + x_{22} F^{(1)}_{21} ) \\
\im ( x_{11} F^{(1)}_{12} + x_{12} F^{(1)}_{11} ) & - x_{11} F^{(1)}_{22} + x_{12} F^{(1)}_{21}
\end{pmatrix} 
= \begin{pmatrix}
\tilde{e}_{11} \mu \epsilon^2 + r(\mu \epsilon^3, \mu^2 \epsilon^2) & \im r(\mu \epsilon^2) \\
\im r(\mu \epsilon^2) & r(\mu \epsilon^2)
\end{pmatrix},
\end{aligned}
\end{equation*}
with $\tilde{\mathsf{e}}_{11}$ being defined as in \eqref{tilde EG and tilde e11}. The expansion of $\widetilde{E}$ in \eqref{tilde EG and tilde e11} follows in view of \eqref{E G tilde}. Since $X = \mathcal{O}(\epsilon)$ by \eqref{XX} and $F^{(1)} = \mathcal{O}(\mu \epsilon)$ by \eqref{F1} we deduce that $X^* F^{(1)} = \mathcal{O}(\mu \epsilon^2)$ and the expansion of $\widetilde{G}$ in \eqref{tilde EG and tilde e11} follows.
\end{proof}
 
 The last term in \eqref{L2 mu epsilon} is small. 

\begin{lemma} \label{high order terms lemma}
    The $4 \times 4$ Hamiltonian and reversibility matrix
\begin{align} \label{high order terms EG hat}
\frac{1}{2} \int_0^1 (1 - \tau^2) \exp(\tau S) \mathrm{ad}_S^2 (R^{(1)}) \exp(-\tau S) \, d\tau
= \begin{pmatrix}
\mathsf{J}_2 \widehat{E} & \mathsf{J}_2 F^{(2)} \\
\mathsf{J}_2 [F^{(2)}]^* & \mathsf{J}_2 \widehat{G}
\end{pmatrix}.
\end{align}
where the $2 \times 2$ self-adjoint and reversible matrices $\widehat{E}$, $\widehat{G}$ have entries
\begin{align} \label{Eij Gij}
\widehat{E}_{ij} = \widehat{G}_{ij} = r(\mu \epsilon^3), 
\quad i,j = 1,2,
\end{align}
and the $2 \times 2$ reversible matrix $F^{(2)}$ admits an expansion as in \eqref{F2}.
\end{lemma}
\begin{proof}
  Each $\exp(\tau S) \operatorname{ad}_S^2 (R^{(1)}) \exp(-\tau S)$ is Hamiltonian and reversibility-preserving, and formula \eqref{high order terms EG hat} holds. In order to estimate its entries we first compute $\operatorname{ad}_S^2 (R^{(1)})$. Using the form of $S$ in \eqref{S1} and $[S, R^{(1)}]$ in \eqref{0.5 S R} one gets
$
\operatorname{ad}_S^2 (R^{(1)}) = 
\begin{pmatrix}
0 & J_2 \widetilde{F} \\
J_2 \widetilde{F}^* & 0
\end{pmatrix}$ where  
$\widetilde{F} := 2 \left( M J_2 \widetilde{G} - \widetilde{E} J_2 M \right)$ 
and $\widetilde{E}$, $\widetilde{G}$ are defined in \eqref{E G tilde}. Since $\widetilde{E}, \widetilde{G} = \mathcal{O}(\mu \epsilon^2)$ by \eqref{tilde EG and tilde e11}, and $M = \mathsf{J}_2 X = \mathcal{O}(\epsilon)$ by \eqref{XX}, we deduce that $\widetilde{F} = \cO(\mu \epsilon^3)$. Then, for any $\tau \in [0,1]$, the matrix $\exp(\tau S) \operatorname{ad}_S^2 (R^{(1)}) \exp(-\tau S) = \operatorname{ad}_S^2 (R^{(1)})(1 + \mathcal{O}(\mu, \epsilon))$. In particular the matrix $F^{(2)}$ in \eqref{high order terms EG hat} has the same expansion of $\widetilde{F}$, namely $F^{(2)} = \mathcal{O}(\mu \epsilon^3)$, and the matrices $\widehat{E}$, $\widehat{G}$ have entries as in \eqref{Eij Gij}.
\end{proof}

\begin{proof} [Proof of Lemma \ref{Step of block-decoupling}]
    It follows by \eqref{L2 mu epsilon}-\eqref{0.5 S R}, \eqref{L=D+R} and Lemmata \ref{EG e11 tilde lemma} and \ref{high order terms lemma}.  
The matrix $E^{(2)} := E^{(1)} + \widetilde{E} + \widehat{E}$ has the expansion in \eqref{E2 in lemma},  
with $\mathsf{e}_{\mathrm{WB}} = \mathsf{e}_{11} + \tilde{\mathsf{e}}_{11}$ as in \eqref{eWB}.  
Similarly, $G^{(2)} := G^{(1)} + \widetilde{G} + \widehat{G}$ has the expansion in \eqref{G2}. 
\end{proof}

\subsection{Complete Block-Decoupling and Proof of the Main Result}

We now block-diagonalize the $4 \times 4$ Hamiltonian and reversible matrix $\mathsf{L}^{(2)}_{\mu,\epsilon}$ in \eqref{L2mueps}. First we split it into its $2 \times 2$-diagonal and off-diagonal Hamiltonian and reversible matrices
\begin{equation} \label{L2mueps2}
\mathsf{L}^{(2)}_{\mu,\epsilon} = D^{(2)} + R^{(2)}, \qquad 
D^{(2)} := 
\begin{pmatrix}
\mathsf{J}_2 E^{(2)} & 0 \\
0 & \mathsf{J}_2 G^{(2)}
\end{pmatrix}, \qquad
R^{(2)} := 
\begin{pmatrix}
0 & \mathsf{J}_2 F^{(2)} \\
\mathsf{J}_2 [F^{(2)}]^* & 0
\end{pmatrix}.
\end{equation}

\begin{lemma} \label{lemma S2}
There exist a $4 \times 4$ reversibility-preserving Hamiltonian matrix $S^{(2)} := S^{(2)}(\mu,\epsilon)$ of the form \eqref{S1}, analytic in $(\mu, \epsilon)$, of size $\mathcal{O}(\epsilon^3)$, and a $4 \times 4$ block-diagonal reversible Hamiltonian matrix $P := P(\mu, \epsilon)$, analytic in $(\mu, \epsilon)$, of size $\mathcal{O}(\mu\epsilon^6)$ such that
\begin{equation} \label{D2+P}
\exp(S^{(2)})(D^{(2)} + R^{(2)}) \exp(-S^{(2)}) = D^{(2)} + P.
\end{equation}
\end{lemma}

\begin{proof}
We set for brevity $S = S^{(2)}$. The equation \eqref{D2+P} is equivalent to the system
\begin{equation} \label{Pi_D Pi_varnothing}
\begin{aligned}
& \Pi_D \left( e^S (D^{(2)} + R^{(2)}) e^{-S} \right) - D^{(2)} = P  \ , \qquad 
\Pi_\varnothing \left( e^S (D^{(2)} + R^{(2)}) e^{-S} \right) = 0 ,
\end{aligned}
\end{equation}
where $\Pi_D$ is the projector onto the block-diagonal matrices and $\Pi_\varnothing$ onto the block-off-diagonal ones. The second equation in \eqref{Pi_D Pi_varnothing} is equivalent, by a Lie expansion, and since $[S, R^{(2)}]$ is block-diagonal, to
\begin{equation} \label{R2+SD2+R}
R^{(2)} + [S, D^{(2)}] + \underbrace{\Pi_\varnothing \int_0^1 (1 - \tau) e^{\tau S} \mathrm{ad}_S^2 (D^{(2)} + R^{(2)}) e^{-\tau S} d\tau}_{=:\mathcal{R}(S)} = 0.
\end{equation}
The ``nonlinear homological equation'' \eqref{R2+SD2+R},
\begin{equation} \label{SD=-R-R}
[S, D^{(2)}] = -R^{(2)} - \mathcal{R}(S),
\end{equation}
is equivalent to solve the $4 \times 4$ real linear system
\begin{equation} \label{Ax=f f=munu+mug}
\mathsf{A} \vec{x} = \vec{f}(\mu, \epsilon, \vec{x}), \qquad \vec{f}(\mu, \epsilon, \vec{x}) = \mu \vec{\nu}(\mu, \epsilon) + \mu \vec{g}(\mu, \epsilon, \vec{x})
\end{equation}
associated, as in \eqref{Ax=f}, to \eqref{SD=-R-R}. The vector $\mu \vec{\nu}(\mu, \epsilon)$ is associated with $-R^{(2)}$ where $R^{(2)}$ is in \eqref{L2mueps2}. The vector $\mu \vec{g}(\mu, \epsilon, \vec{x})$ is associated with the matrix $-\mathcal{R}(S)$, which is a Hamiltonian and reversible block-off-diagonal matrix (i.e. of the form \eqref{L=D+R}).
The factor $\mu$ is present in $D^{(2)}$ and $R^{(2)}$, see \eqref{E2 in lemma}, \eqref{G2}, \eqref{F2} and the analytic function $\vec{g}(\mu, \epsilon, \vec{x})$ is quadratic in $\vec{x}$ (for the presence of $\mathrm{ad}_S^2$ in $\mathcal{R}(S)$). In view of \eqref{F2} one has
\begin{equation} \label{munu}
\mu \vec{\nu}(\mu, \epsilon) := (-F^{(2)}_{21}, F^{(2)}_{22}, -F^{(2)}_{11}, F^{(2)}_{12})^\top, \qquad F^{(2)}_{ij} = r(\mu \epsilon^3).
\end{equation}
System \eqref{Ax=f f=munu+mug} is equivalent to $\vec{x} = \mathsf{A}^{-1} \vec{f}(\mu, \epsilon, \vec{x})$ and, writing $\mathsf{A}^{-1} = \frac{1}{\mu} \mathcal{B}(\mu, \epsilon)$ (cf. \eqref{mathsf A -1}), to
\[
\vec{x} = \mathcal{B}(\mu, \epsilon) \vec{\nu}(\mu, \epsilon) + \mathcal{B}(\mu, \epsilon) \vec{g}(\mu, \epsilon, \vec{x}).
\]
By the implicit function theorem this equation admits a unique small solution $\vec{x} = \vec{x}(\mu, \epsilon)$, analytic in $(\mu, \epsilon)$, with size $\mathcal{O}(\epsilon^3)$ as $\vec{\nu}$ in \eqref{munu}. Then the first equation of \eqref{Pi_D Pi_varnothing} gives $P = [S, R^{(2)}] + \Pi_D \int_0^1 (1-\tau) e^{\tau S} \text{ad}_S^2 (D^{(2)} + R^{(2)}) e^{-\tau S} d\tau$, and its estimate follows from those of $S$ and $R^{(2)}$ (see \eqref{F2}). 
\end{proof}

\begin{proof} [{Proof of Theorems \ref{Complete BF thm} and \ref{thm:main}.}]
By Lemma \ref{lemma S2} and recalling \eqref{mathcal L= ichmu+mathscr L} the operator $\mathcal{L}_{\mu, \epsilon} : \mathcal{V}_{\mu, \epsilon} \to \mathcal{V}_{\mu, \epsilon}$ is represented by the $4 \times 4$ Hamiltonian and reversible matrix
\[
\im \chk \mu + \exp(S^{(2)}) \mathsf{L}^{(2)}_{\mu, \epsilon} \exp(-S^{(2)}) 
= \im \chk \mu + 
\begin{pmatrix}
\mathsf{J}_2 E^{(3)} & 0 \\
0 & \mathsf{J}_2 G^{(3)}
\end{pmatrix}
=: 
\begin{pmatrix}
\mathsf{U} & 0 \\
0 & \mathsf{S}
\end{pmatrix},
\]
where the matrices $E^{(3)}$ and $G^{(3)}$ expand as in \eqref{E2 in lemma}, \eqref{G2}. Consequently the matrices $\mathsf{U}$ and $\mathsf{S}$ expand as in \eqref{U S}. Theorem \ref{Complete BF thm} is proved. Theorem \ref{thm:main} is a straightforward corollary. The function $\und{\mu}(\epsilon)$ in \eqref{und mu} is defined as the implicit solution of the function $\Delta_{\text{BF}}(\kappa,\tth; \mu, \epsilon)$ in \eqref{BFDF} for $\epsilon$ small enough, depending on $(\kappa,\tth)$. 
\end{proof}

\appendix
\section{Expansion of the Stokes waves in finite depth}\label{sec:App2}

In this Appendix we provide the expansions  \eqref{exp:Sto}-\eqref{expcoef}, \eqref{expfe}, 
\eqref{pino1fd}-\eqref{aino2fd}.
\\[1mm]
\noindent
{\bf Proof  of \eqref{exp:Sto}-\eqref{expcoef}.}
Writing
 \begin{equation}\label{etapsic}
\begin{aligned}
 & \eta_\e(x) = \e \eta_1(x) + \e^2 \eta_2(x) + \cO(\e^3) \, , \\
 &  \psi_\e(x) = \e \psi_1(x) + \e^2 \psi_2(x) + \cO(\e^3) \, ,  
  \end{aligned}
\qquad \quad c_\e = \chk + \e c_1 + \e^2 c_2+ \cO(\e^3) \, ,  
\end{equation}
where  $\eta_i$ is $even(x)$ and $\psi_i$ is $odd(x)$ for $i=1,2$, 
we solve order by order in $ \e $ the equations \eqref{travelingWWstokes},  
that we rewrite  as
\begin{equation}
\label{Sts}
\begin{cases}
-c \, \psi_x +  \eta   + \dfrac{\psi_x^2}{2} - 
\dfrac{\eta_x^2}{2(1+\eta_x^2)} ( c  -  \psi_x )^2-\kappa\left(\frac{\eta_x}{(1+\eta_x^2)^{1/2}}\right)_x  = 0 \\
c \, \eta_x+G(\eta)\psi = 0 	\, ,
\end{cases}
\end{equation}
having substituted $G(\eta)\psi $ with $-c \, \eta_x $  in the first equation.
 We expand the Dirichlet-Neumann operator 
$ G(\eta)=  G_0+ G_1(\eta) + G_2(\eta) + \cO(\eta^3)  $
where, according to \cite[formula (2.14)]{CS}, 
\begin{equation} \label{expDiriNeu}
\begin{aligned}
 G_0 & := D\tanh(\tth D) = |D| \tanh(\tth |D|) \,, \\
 G_1(\eta) & := D \big( \eta -  \tanh(\tth D)\eta \tanh(\tth D) \big)D = -\pa_x \eta \pa_x - |D| \tanh(\tth|D|)\eta|D| \tanh(\tth|D|), \\
 G_2(\eta) & := -\frac12 D \Big( 
 D{\eta}^2 \tanh(\tth D) +\tanh(\tth D){\eta}^2 D - 2\tanh(\tth D)\eta D\tanh(\tth  D)\eta\tanh(\tth D) \Big)D \, .
\end{aligned}
\end{equation}
{\bf First order in $ \e $.}  Substituting the expansions in \eqref{etapsic} into \eqref{Sts}, we get the linear system 
\begin{equation}\label{cB0}
 \left\{\begin{matrix} -\chk (\psi_1)_x + \eta_1 -\kappa (\eta_1)_{xx}= 0 \\
 \chk (\eta_1)_x + G_0\psi_1 =0 \, ,  \end{matrix}\right.
 \quad  \text{i.e.} \, \vet{\eta_1}{\psi_1} \in \text{Ker }\cB_0  \text{ with } \cB_0 := \begin{bmatrix} 1-\kappa\pa_{xx} & -\chk\pa_x \\ \chk\pa_x & G_0  \end{bmatrix},  
\end{equation}
where $\eta_1$ is $even(x)$ and $\psi_1$ is $odd(x)$.
\begin{lemma}\label{lem:B0}
If $(\kappa, \tth) \in \R^2_+ \setminus \fR$,  the kernel of 
 the linear operator $\cB_0$ in \eqref{cB0} is one dimensional and given by
\begin{equation}\label{chk}
\text{Ker }\cB_0= \text{span}\,\Big\{\vet{\cos(x)}{\ch^{-1}(1+\kappa)^{\frac{1}{2}}\sin(x)} \Big\}.
\end{equation}
\end{lemma}

\begin{proof} The action of $\cB_0$ on each subspace spanned by $\footnotesize{\,\Big\{\vet{\cos(kx)}{0}, \vet{0}{\sin(kx)}\Big\}} $, $k\in \N$, is represented by the $2\times 2$ matrix $\footnotesize{ \begin{bmatrix} 1+\kappa k^2 & -\chk k \\ -\chk k & k\tanh(\tth k)  \end{bmatrix}}$. Its determinant is given by
 $$ (1+\kappa k^2)k \tanh(\tth k) - \chk^2 k^2\stackrel{\eqref{exp:Sto}}{=} k^2 \Big(\frac{(1+\kappa k^2)\tanh(\tth k)}{k} -(1+\kappa)\tanh(\tth) \Big)$$
 and, provided $(\kappa, \tth) \in \R^2_+ \setminus \fR$, 
 vanishes  if and only if $k=1$.
Indeed, for $k$ large enough, the determinant goes asymptotically to $+\infty$ so it is uniformly bounded away from zero,  whereas if for some $k \in \N$ it vanishes, it implies that $(\kappa, \tth) \in \fR_k \subset \fR$, absurd.
  For $k=1$ we obtain the kernel of $\cB_0$ given in \eqref{chk}. For $k=0$ it has no kernel since $\psi_1(x)$ is odd.
\end{proof}
We set
$
\eta_1(x) := \cos(x)$, $\psi_1(x) := \ch^{-1}(1+\kappa)^{\frac{1}{2}} \sin(x) 
$
in agreement with 
\eqref{exp:Sto}. 
\\[1mm]
{\bf Second order in $ \e $.} 
By \eqref{Sts}, and since 
$ \chk^2 (\eta_1)_x^2 = (G_0\psi_1)^2  $, we get  the linear system
\begin{equation}\label{syslin2}
 \cB_0 \vet{\eta_2}{\psi_2} = \vet{c_1(\psi_1)_x-\frac12 (\psi_1)_x^2 + \frac12 (G_0\psi_1)^2 }{-c_1(\eta_1)_x - G_1(\eta_1)\psi_1} \, , 
\end{equation}
where  $\cB_0$ is the self-adjoint operator in \eqref{cB0}. 
System \eqref{syslin2} admits a solution if and only if its right-hand term is orthogonal to the Kernel  of $\cB_0$ in \eqref{chk}, namely
\begin{equation}\label{orth1}
 \Big(\vet{c_1(\psi_1)_x-\frac12 (\psi_1)_x^2 + \frac12 (G_0\psi_1)^2 }{-c_1(\eta_1)_x - G_1(\eta_1)\psi_1}\;,\;\vet{\cos(x)}{\ch^{-1}(1+\kappa)^{\frac{1}{2}}\sin(x)}\Big)=0 \, . 
\end{equation}
In view of the first order expansion \eqref{exp:Sto}, \eqref{expDiriNeu} and the identity
 $
  \tanh(2\tth)  = 
  \displaystyle{\frac{2\ch^2}{1+\ch^4}}
 $,
  it results $
 [G_0\psi_1](x)= \ch (1+\kappa)^{\frac{1}{2}}\sin(x)$, $\big[G_1(\eta_1)\psi_1\big](x) 
 =\frac{(1-\ch^4)(1+\kappa)^{\frac{1}{2}}}{\ch(1+\ch^4)}\sin(2x)$
so that \eqref{orth1} implies
 $c_1=0$, in agrement with \eqref{exp:Sto}. Equation
 \eqref{syslin2} reduces to 
\begin{align}\label{sisto2}
\begin{bmatrix} 1-\kappa\pa_{xx} & -\chk\pa_x \\ \chk\pa_x & G_0  \end{bmatrix}
\vet{\eta_2}{\psi_2} 
  = \vet{\frac{1+\kappa}{4}\left(\ch^2-\ch^{-2}\right)-\frac{1+\kappa}{4}\left(\ch^2+\ch^{-2}\right)\cos(2x) }{ -\frac{(1-\ch^4)(1+\kappa)^{\frac{1}{2}}}{\ch(1+\ch^4)}\sin(2x)}.
\end{align}
Setting $ \eta_2 = \eta_2^{[0]} +  \eta_2^{[2]} \cos(2x) $ and  $ \psi_2 =
 \psi_2^{[2]}  \sin (2x) $, system 
\eqref{sisto2} amounts to 
\begin{align*}
\left\{ \begin{matrix} \eta_2^{[0]} +\big( (1+4\kappa)\eta_2^{[2]} -2\chk  \psi_2^{[2]} \big) \cos(2x)  =  \frac{1+\kappa}{4}\left(\ch^2-\ch^{-2}\right)-\frac{1+\kappa}{4}\left(\ch^2+\ch^{-2}\right)\cos(2x)   \\ (-2\chk \eta_2^{[2]}   + 2  \psi_2^{[2]} \tanh(2\tth))\sin(2x)  =   -\frac{(1-\ch^4)(1+\kappa)^{\frac{1}{2}}}{\ch(1+\ch^4)}\sin(2x) \, , \end{matrix}\right.
\end{align*}
which leads to the expansions of $ \eta_2^{[0]} $, $ \eta_2^{[2]} $, $ \psi_2^{[2]} $ 
given in
\eqref{exp:Sto}-\eqref{expcoef}.

\noindent
{\bf Third order in $ \e $.} 
It remains to determine 
$ c_2 $ in \eqref{expc2}.
 We get the linear system 
\begin{equation}\label{syslin3}
 \cB_0 \vet{\eta_3}{\psi_3} = \vet{c_2(\psi_1)_x 
 - (\psi_1)_x (\psi_2)_x - (\eta_1)_x^2 (\psi_1)_x \chk + 
 (\eta_1)_x (\eta_2)_x \chk^2-\frac{3}{2}\kappa (\eta_1)^2_x (\eta_1)_{xx}
 }{-c_2(\eta_1)_x - G_1(\eta_1)\psi_2- G_1(\eta_2)\psi_1 - G_2(\eta_1)\psi_1} \, . 
\end{equation}
System \eqref{syslin3} has 
a solution if and only if the right hand side is orthogonal to the Kernel of
$ \cB_0 $ given in \eqref{chk}. This condition determines uniquely $ c_2 $.
Denoting  $\Pi_1$ the $L^2$-orthogonal projector on span$\, \{\cos(x),\sin(x)\} $, it results 
\begin{align*} 
& c_2 (\psi_1)_x = c_2 \ch^{-1}(1+\kappa)^{\frac{1}{2}} \cos(x)\, , \quad c_2 (\eta_1)_x = -c_2 \sin(x) \, , \quad \Pi_1[ (\psi_1)_x (\psi_2)_x] = 
\psi_2^{[2]} \ch^{-1} (1+\kappa)^{\frac{1}{2}} \cos(x)\, ,\\ 
& \Pi_1 [\chk (\eta_1)_x^2 (\psi_1)_x ] = \tfrac14 (1+\kappa)\cos(x)\, , \quad \Pi_1[\chk^2 (\eta_1)_x (\eta_2)_x] = \eta_2^{[2]}\ch^2(1+\kappa) \cos(x) \, ,\\
& \Pi_1[\frac{3}{2}\kappa(\eta_1)^2_x(\eta_1)_{xx}] = \Pi_1 [-\frac{3}{8}\kappa (-\cos(3x)+\cos(x))]=-\frac{3}{8}\kappa \cos(x) \, ,
\end{align*}
and, in view of \eqref{expDiriNeu}, and \eqref{exp:Sto}, \eqref{expcoef}, 
\begin{align*}
 \Pi_1[ G_1(\eta_1)\psi_2]  &= \psi_2^{[2]}\frac{1-\ch^4}{1+\ch^4} \sin(x) \, ,  \quad 
  \Pi_1[G_2(\eta_1)\psi_1] 
 =   -\frac{\ch(1+\kappa)^{\frac{1}{2}}}{4}\frac{1-3\ch^4}{1+\ch^4} \sin(x) \, , \\
 \Pi_1[G_1(\eta_2)\psi_1] 
&= \ch^{-1}(1+\kappa)^{\frac{1}{2}} 
 \Big( \eta_2^{[0]}(1-\ch^4) +   \frac12 \eta_2^{[2]} (1+\ch^4) \Big)\sin(x) \, . 
\end{align*}
Imposing the orthogonality condition gives $c_2$ as in   \eqref{expc2}.

\noindent{\bf Proof of \eqref{expfe}.} 
We expand  the function $\mathfrak{p}(x)  = \e\mathfrak{p}_1(x) + \e^2 \mathfrak{p}_2(x) + \cO(\e^3)  $ defined by the fixed point equation \eqref{def:ttf}.   We first 
note that  the constant $\mathtt{f}_\e=\cO(\e^2)$  because
 $\eta_1(x) = \cos(x)$ has zero average. 
Then 
$ \mathfrak{p}(x)
 = \frac{\mathfrak H}{\tanh(\tth |D|)} \big[\e\eta_1 +\e^2\big(\eta_2 + 
 (\eta_1)_x \mathfrak{p}_1 \big)+\cO(\e^3)\big] $, 
and, using that $\mathfrak H \cos(kx) = \sin(kx)$, for any  $ k \in \N$, we get 
\begin{align}
 \mathfrak{p}_1(x) 
 & = \frac{\mathfrak H}{\tanh(\tth |D|)}\cos(x)  = \ch^{-2}\sin(x) \, , \label{pfra1} \\
 \mathfrak{p}_2(x) &= \frac{\mathfrak H}{\tanh(\tth |D|)}((\eta_1)_x
 \mathfrak{p}_1 +\eta_2 ) 
 = \frac{\left({\ch}^4+1\right)\,\left({\ch}^4\,\kappa -3\,\kappa +{\ch}^4+3\right)}{8\,{\ch}^4\,\left(\ch^4\,(1+\kappa) -3\,\kappa \right)}\sin(2x) \, . \label{pfra2}
\end{align}
Finally
\begin{align*}
 \tf_\e & = \frac{\e^2}{2\pi} \int_\T \big( \eta_2 + (\eta_1)_x\mathfrak{p}_1 \big) \de x + \cO(\e^3) 
 = \e^2\big(\eta_2^{[0]} -\tfrac1{2} \ch^{-2} \big)+\cO(\e^3 ) \stackrel{\eqref{expcoef}}{=} \e^2 \tf^{[2]}_2+\cO(\e^3 ) \, .
\end{align*}
The expansion \eqref{expfe} is proved.
\\[1mm]
{\bf Proof of Lemma \ref{lem:pa.exp}.} 
In view of \eqref{exp:Sto}-\eqref{expcoef}, the expansions of the functions $B$, $V$  in \eqref{espV} and \eqref{espB} are
\begin{align}
 B&= : \e B_1(x) + \e^2 B_2(x) + \cO(\e^3) 
\notag \\
&= \e\ch (1+\kappa)^{\frac{1}{2}} \sin(x) + \e^2 \Big[\frac{(1+\kappa)^{\frac{1}{2}}\,\left(6\,\kappa -2\,\ch^4\,\kappa -2\,\ch^4+3\right)}{2\,\ch\,\left(\ch^4\,(1+\kappa) -3\,\kappa \right)}\Big]\sin(2x) + \cO(\e^3) \label{espB1}  
  \end{align}
  and
  \begin{align}
 V&= : \e V_1(x) + \e^2 V_2(x)  + \cO(\e^3) \notag\\
 &= \e \ch^{-1} (1+\kappa)^{\frac{1}{2}}\cos(x) + \e^2 \Big[\frac{\ch\,(1+\kappa)^{\frac{1}{2}}}{2}  +\Big(\frac{(1+\kappa )^{\frac{1}{2}}\,\left(9\,\kappa -\ch^8\,(1+\kappa)+ 3\right)}{4\ch^3\,\left(\ch^4\,(1+\kappa) -3\,\kappa\right)}\Big)\cos(2x)\Big]+\cO(\e^3)
  \, .\label{espV1} 
\end{align}
In  view of  \eqref{def:pa}, denoting derivatives w.r.t $x$ with  a prime and suppressing dependence on $x$ when trivial, we have
\begin{align}
\chk+p_\e(x)  &= (\chk+\e^2 c_2 - V(x) - V'(x)\mathfrak{p}(x)+\cO(\e^3)) (1-\mathfrak{p}'(x)+(\mathfrak{p}'(x))^2+\cO(\e^3))\notag \\
&= \chk + \e \underbrace{(- V_1 -\chk \mathfrak{p}_1')}_{=: p_1}+\e^2 \underbrace{\big( c_2 + V_1\mathfrak{p}_1' - V_2 - V_1'\mathfrak{p}_1 -\chk \mathfrak{p}_2' +\chk (\mathfrak{p}_1')^2 \big)}_{=:p_2} + \, \cO(\e^3) \, .\label{pino12imp}
\end{align} 
Similarly by \eqref{def:pa}
\begin{align}
 1+a_\e(&x) : = \frac{1}{1+\mathfrak{p}_x(x)} - (\chk+p_\e(x))B_x(x+\mathfrak{p}(x)) \notag \\
   = &  1+\e\underbrace{\big(-\mathfrak{p}_1'- \chk B_1'\big)}_{=:a_1} +\e^2\underbrace{\big((\mathfrak{p}_1')^2-\mathfrak{p}_2'- \chk B_2'-\chk B_1''\mathfrak{p}_1(x)+ B_1' V_1 + \chk  B_1'\mathfrak{p}_1' \big)}_{=:a_2}+\cO(\e^3)\, . \label{aino12imp}
 \end{align}
By \eqref{espV1}, \eqref{pfra1}, \eqref{exp:Sto}, \eqref{pfra2}, \eqref{espB1} we 
deduce that the functions $p_1 $, $p_2 $, $a_1 $, $a_2 $ in \eqref{pino12imp} and \eqref{aino12imp} have an expansion as in \eqref{pino1fd}-\eqref{aino2fd}.

In view of \eqref{exp:Sto}-\eqref{expcoef}, the expansion of the function $l(x)$ in \eqref{esp l} is 
\begin{align}
 l(x)&= 1 + \e^2 l_2(x)  + \cO(\e^3)= 1+ \e^2  \Big[-\frac{3}{4}+\frac{3}{4}\cos(2x)\Big]+\cO(\e^3)
 \label{espl1} \, . 
\end{align}
In  view of  \eqref{def:Sigma g}, denoting derivatives w.r.t $x$ with  a prime and suppressing dependence on $x$ when trivial, we have
\begin{align}
g_\e(x)  &= (1-\e\mathfrak{p}_1'+\e^2(-\mathfrak{p}_2'+(\mathfrak{p}_1')^2)+\cO(\e^3))(1+\e^2 l_2+\cO(\e^3))\notag \\
&= 1 + \e \underbrace{(-\mathfrak{p}_1')}_{=: g_1}+\e^2 \underbrace{\big(-\mathfrak{p}_2'+(\mathfrak{p}_1')^2+l_2)}_{=:g_2} + \, \cO(\e^3) \, .\label{Sigmano12imp}
\end{align} 
By \eqref{expfe} and \eqref{espl1}, the expansions \eqref{exp:g1}--\eqref{exp:g2} follow.

Finally the operator $\Sigma_\e$ in \eqref{def:Sigma g} expands as
$\Sigma_\e = \pa_{xx} + \e \Sigma_1 + \e^2 \Sigma_2 + \cO(\e^3)$
where
\begin{align}\label{Sigma 1 expansion}
    \Sigma_1:=& (g_{1}-2\mathfrak{p}_1')\,\pa_{xx}+(g_1'-2\,\mathfrak{p}_1'')\pa_{x}-\mathfrak{p}_1''',\\ \notag
    \Sigma_2:=&(g_2-2\mathfrak{p}_2'+3(\mathfrak{p}_1')^2-2g_1\mathfrak{p}_1') \pa_{xx}+(g_2'-2\mathfrak{p}_2''-2g_1\mathfrak{p}_1''-2g_1'\mathfrak{p}_1'+6\mathfrak{p}_1''\mathfrak{p}_1')\pa_{x}\\ \label{Sigma 2}
    &+2(\mathfrak{p}_1'')^2-\mathfrak{p}_2'''-\mathfrak{p}_1''g_1'+3\mathfrak{p}_1'''\mathfrak{p}_1'-g_1\mathfrak{p}_1'''.
\end{align}
Inserting the expansions of $g_1,g_2$ in \eqref{exp:g1}, \eqref{exp:g2} and those of $\mathfrak{p}_1$, $\mathfrak{p}_2$ in \eqref{pfra1}, \eqref{pfra2} proves the claimed expansions 
in \eqref{Sigma 1} with coefficients in \eqref{d1x}--\eqref{h2 h02}.
\qed

\section{Expansion of the Kato Basis}\label{secA1}

In this appendix we prove Lemma \ref{expansion of the basis F}. We provide the expansion of the basis $f_k^{\pm}(\mu,\epsilon) = U_{\mu,\epsilon} f_k^{\pm}$, $k = 0,1$, in \eqref{F basis set and f}, where $f_k^{\pm}$ defined in \eqref{eigenfunc of mathcall L00 2} belong to the subspace $\mathcal{V}_{0,0} := \operatorname{Rg}(P_{0,0})$. We first Taylor-expand the transformation operators $U_{\mu,\epsilon}$ defined in \eqref{U transformation operators}. We denote $\partial_{\epsilon}$ with a prime and $\partial_{\mu}$ with a dot.
The next lemma follows as \cite[Lemma A.1]{BMV1}
\begin{lemma} \label{U mu epsilon P00}
    The first jets of $U_{\mu,\epsilon} P_{0,0}$ are
\begin{align} \label{A1}
    &U_{0,0} P_{0,0} = P_{0,0}, \quad\quad U'_{0,0} P_{0,0} = P'_{0,0} P_{0,0}, \quad\quad \dot{U}_{0,0} P_{0,0} = \dot{P}_{0,0} P_{0,0},\\ \label{A2}
    &\dot{U}'_{0,0} P_{0,0} = \left( \dot{P}'_{0,0} - \frac{1}{2} P_{0,0} \dot{P}'_{0,0} \right) P_{0,0}, 
\end{align}
where
\begin{align} \label{A3}
    P'_{0,0} &= \frac{1}{2\pi \im} \oint_{\Gamma} (\mathscr{L}_{0,0} - \lambda)^{-1} \mathscr{L}'_{0,0} (\mathscr{L}_{0,0} - \lambda)^{-1} d\lambda,  \\ \label{A4}
    \dot{P}_{0,0} &= \frac{1}{2\pi \im} \oint_{\Gamma} (\mathscr{L}_{0,0} - \lambda)^{-1} \dot{\mathscr{L}}_{0,0} (\mathscr{L}_{0,0} - \lambda)^{-1} d\lambda, 
\end{align}
and
\begin{align} \label{A5a}
    \dot{P}'_{0,0} &= -\frac{1}{2\pi \im} \oint_{\Gamma} (\mathscr{L}_{0,0} - \lambda)^{-1} \dot{\mathscr{L}}_{0,0} (\mathscr{L}_{0,0} - \lambda)^{-1} \mathscr{L}'_{0,0} (\mathscr{L}_{0,0} - \lambda)^{-1} d\lambda \\ \label{A5b}
    &\quad -\frac{1}{2\pi \im} \oint_{\Gamma} (\mathscr{L}_{0,0} - \lambda)^{-1} \mathscr{L}'_{0,0} (\mathscr{L}_{0,0} - \lambda)^{-1} \dot{\mathscr{L}}_{0,0} (\mathscr{L}_{0,0} - \lambda)^{-1} d\lambda \\ \label{A5c}
    &\quad + \frac{1}{2\pi \im} \oint_{\Gamma} (\mathscr{L}_{0,0} - \lambda)^{-1} \dot{\mathscr{L}}'_{0,0} (\mathscr{L}_{0,0} - \lambda)^{-1} d\lambda.
\end{align}
The operators $\mathscr{L}'_{0,0}$ and $\dot{\mathscr{L}}_{0,0}$ are
\begin{align} \label{A6}
    \mathscr{L}'_{0,0}=\begin{bmatrix}
        \partial_x\circ p_1(x)  &0\\
        -a_1(x)+\kappa \Sigma_1 &p_1(x)\circ \partial_x
    \end{bmatrix},~~\dot{\mathscr{L}}_{0,0}=\begin{bmatrix}
        0 &\mathrm{sgn}(D)m(D)\\
        2\im \kappa \partial_{x} &0
    \end{bmatrix},
\end{align}
where $\mathrm{sgn}(D)$ is defined in \eqref{D+mu}, \eqref{sgn D} and $m(D)$ is the real, even operator
\begin{align} \label{mD}
    m(D):=\tanh(\tth|D|)+\tth|D|(1-\tanh^2(\tth|D|))
\end{align}
and $a_1(x)$, $p_1(x)$, $\Sigma_1$ are given in Lemma \ref{lem:pa.exp}.

The operator $\dot{\mathscr{L}}'_{0,0}$ is
\begin{align} \label{A8}
    \dot{\mathscr{L}}'_{0,0}=\begin{bmatrix}
        \im p_1(x) & 0\\
        \im \kappa\left(2 d_1(x) \pa_x + e_1(x)\right)        & \im p_1(x)
    \end{bmatrix}
\end{align}
with $d_1(x), e_1(x)$ in \eqref{d1x}, \eqref{d2x}. 
\end{lemma}

By the Lemma \ref{U mu epsilon P00}, we have the Taylor expansion
\begin{equation} \label{the expandsion of f sigma mu}
    \begin{aligned}
    f_k^\sigma (\mu, \epsilon) &= f_k^\sigma + \epsilon P'_{0,0} f_k^\sigma + \mu \dot{P}_{0,0} f_k^\sigma  + \mu \epsilon \big( \dot{P}'_{0,0} - \frac{1}{2} P_{0,0} \dot{P}'_{0,0} \big) f_k^\sigma + \mathcal{O}(\mu^2, \epsilon^2).
\end{aligned}
\end{equation}
In order to compute the vectors $P'_{0,0} f_k^\sigma$ and $\dot{P}_{0,0} f_k^\sigma$ using \eqref{A3} and \eqref{A4}, it is useful to know the action of $(\mathscr{L}_{0,0} - \lambda)^{-1}$ on the vectors
\begin{equation} \label{A10}
    \begin{aligned}
        f^+_k&:=\vet{(\ch/\sqrt{1+\kappa})^{1/2}\cos(kx)}{(\ch/\sqrt{1+\kappa})^{-1/2}\sin(kx)},~~f^-_k:=\vet{-(\ch/\sqrt{1+\kappa})^{1/2}\sin(kx)}{(\ch/\sqrt{1+\kappa})^{-1/2}\cos(kx)},\\
        f^+_{-k}&:=\vet{(\ch/\sqrt{1+\kappa})^{1/2}\cos(kx)}{-(\ch/\sqrt{1+\kappa})^{-1/2}\sin(kx)},~~f^-_{-k}:=\vet{(\ch/\sqrt{1+\kappa})^{1/2}\sin(kx)}{(\ch/\sqrt{1+\kappa})^{-1/2}\cos(kx)},~~k\in\mathbb{N}.
    \end{aligned}
\end{equation}

\begin{lemma}
    The space $Y=H^2(\mathbb{T},\mathbb{C})\times H^1(\mathbb{T},\mathbb{C})$ decomposes as $Y = \mathcal{V}_{0,0} \oplus \mathcal{U} \oplus \mathcal{W}_{Y}$, with
\begin{equation*}
    \mathcal{W}_{Y} = \overline{\bigoplus_{k=2}^{\infty} \mathcal{W}_k}^{Y}
\end{equation*}
where the subspaces $\mathcal{V}_{0,0}, \mathcal{U}$, and $\mathcal{W}_k$, defined below, are invariant under $\mathscr{L}_{0,0}$ and the following properties hold:

\begin{itemize}
\item[(i)] $\mathcal{V}_{0,0} = \text{span}\{f_1^+, f_1^-, f_0^+, f_0^-\}$ is the generalized kernel of $\mathscr{L}_{0,0}$. For any $\lambda \neq 0$ the operator $\mathscr{L}_{0,0} - \lambda : \mathcal{V}_{0,0} \to \mathcal{V}_{0,0}$ is invertible and
    \begin{equation} \label{A11-0}
        (\mathscr{L}_{0,0} - \lambda)^{-1} f_1^+ = -\frac{1}{\lambda} f_1^+, \quad (\mathscr{L}_{0,0} - \lambda)^{-1} f_1^- = -\frac{1}{\lambda} f_1^-,
    \end{equation}
    \begin{equation} \label{A11}
        (\mathscr{L}_{0,0} - \lambda)^{-1} f_0^- = -\frac{1}{\lambda} f_0^-,
    \end{equation}
    \begin{equation} \label{A12}
        (\mathscr{L}_{0,0} - \lambda)^{-1} f_0^+ = -\frac{1}{\lambda} f_0^+ + \frac{1}{\lambda^2} f_0^- .
    \end{equation}

    \item[(ii)] $\mathcal{U} := \text{span}\{ f_{-1}^+, f_{-1}^- \}$. For any $\lambda \neq \pm \im \,2\chk$ the operator $\mathscr{L}_{0,0} - \lambda : \mathcal{U} \to \mathcal{U}$ is invertible and
    \begin{equation} \label{A13}
    \begin{aligned}
        (\mathscr{L}_{0,0} - \lambda)^{-1} f_{-1}^+ &= \frac{1}{\lambda^2 + 4\chk^2} \left(-\lambda f_{-1}^+ + 2\chk f_{-1}^-\right),\\
        (\mathscr{L}_{0,0} - \lambda)^{-1} f_{-1}^- &= \frac{1}{\lambda^2 + 4\chk^2} \left(-2\chk f_{-1}^+ - \lambda f_{-1}^-\right).
    \end{aligned}
    \end{equation}

    \item[(iii)] Each subspace $\mathcal{W}_k := \text{span}\{ f_k^+, f_k^-, f_{-k}^+, f_{-k}^- \}$ is invariant under $\mathscr{L}_{0,0}$. Let
    \begin{equation*}
        \mathcal{W}_{L^2} = \overline{\bigoplus_{k=2}^{\infty} \mathcal{W}_k}^{ L^2}.
    \end{equation*}
    For any $|\lambda| < \delta(\tth)$ small enough, the operator $\mathscr{L}_{0,0} - \lambda : \mathcal{W}_{Y} \to \mathcal{W}_{L^2}$ is invertible and for any $f \in \mathcal{W}_{L^2}$,
    \begin{equation} \label{A14}
        (\mathscr{L}_{0,0} - \lambda)^{-1} f = \left( \chk^2 \partial^2_{x} + |D| \tanh(\tth|D|)(1-\kappa \pa^2_{x}) \right)^{-1} \begin{bmatrix} \chk \partial_x & -|D| \tanh(\tth|D|) \\ 1-\kappa \pa^2_{x} & \chk \partial_x \end{bmatrix} f
        + \lambda \varphi_f(\lambda, x),
    \end{equation}
    for some analytic function $\lambda \mapsto \varphi_f(\lambda, \cdot) \in Y=H^2(\mathbb{T},\mathbb{C})\times H^1(\mathbb{T},\mathbb{C})$.
\end{itemize}
\end{lemma}

\begin{proof}
    By inspection the spaces $\mathcal{V}_{0,0}, \mathcal{U}$ and $\mathcal{W}_k$ are invariant under $\mathscr{L}_{0,0}$ and, by Fourier series, they decompose $Y=H^2(\mathbb{T},\mathbb{C})\times H^1(\mathbb{T},\mathbb{C})$. Formulas \eqref{A11}-\eqref{A12} follow using that $f_1^+, f_1^-, f_0^-$ are in the kernel of $\mathscr{L}_{0,0}$, and $\mathscr{L}_{0,0} f_0^+ = -f_0^-$. Formula \eqref{A13} follows using that $\mathscr{L}_{0,0} f_{-1}^+ = -2\chk f_{-1}^-$ and $\mathscr{L}_{0,0} f_{-1}^- = 2\chk f_{-1}^+$. Let us prove item $(iii)$. Let $\mathcal{W} := \mathcal{W}_{Y}$. The operator $(\mathscr{L}_{0,0} - \lambda \mathrm{Id})|_{\mathcal{W}}$ is invertible for any 
$$\lambda \notin \{ \pm i k\chk \pm i\sqrt{|k|\tanh(\tth| k|)\left(1+\kappa k^2\right)}, k \geq 2, k \in \mathbb{N} \}$$
and 
$$
(\mathscr{L}_{0,0}|_{\mathcal{W}})^{-1} = \left( \chk^2 \partial_{xx} + |D| \tanh(\tth|D|)(1-\kappa \pa_{xx}) \right)^{-1} \begin{bmatrix} \chk \partial_x & -|D| \tanh(\tth|D|) \\ 1-\kappa \pa_{xx} & \chk \partial_x \end{bmatrix} \Big|_{\mathcal{W}}.
$$
By Neumann series, for any $\lambda$ such that 
$$|\lambda| \| (\mathscr{L}_{0,0}|_{\mathcal{W}})^{-1} \|_{\mathcal{L}(\mathcal{W}, Y)} < 1$$
we have 
$$
(\mathscr{L}_{0,0}|_{\mathcal{W}} - \lambda)^{-1} = (\mathscr{L}_{0,0}|_{\mathcal{W}})^{-1} (\text{Id} - \lambda (\mathscr{L}_{0,0}|_{\mathcal{W}})^{-1})^{-1} 
= (\mathscr{L}_{0,0}|_{\mathcal{W}})^{-1} \sum_{k \geq 0} ((\mathscr{L}_{0,0}|_{\mathcal{W}})^{-1} \lambda)^k.
$$
Formula \eqref{A14} follows with 
$$\varphi_f(\lambda, x) := (\mathscr{L}_{0,0}|_{\mathcal{W}})^{-1} \sum_{k \geq 1} \lambda^{k-1} [(\mathscr{L}_{0,0}|_{\mathcal{W}})^{-1}]^k f.$$ 
\end{proof}

To prove Lemma \ref{expansion of the basis F}, we shall also use the following formulas obtained by \eqref{A6}, \eqref{mD}, and \eqref{eigenfunc of mathcall L00 2}:

\begin{equation} \label{A15}
    \begin{aligned}
        &\mathscr{L}'_{0,0} f^+_1=\vet{2\left(\ch/\sqrt{1+\kappa}\right)^{-\frac{1}{2}}\sin(2x)}{b^{[1]}_{\tth,\kappa}+b^{[2]}_{\tth,\kappa}\cos(2x)},\\
        &\mathscr{L}'_{0,0} f^-_1=\vet{2\left(\ch/\sqrt{1+\kappa}\right)^{-\frac{1}{2}}\cos(2x)}{-b^{[2]}_{\tth,\kappa}\sin(2x)},\\
        &\mathscr{L}'_{0,0} f^+_0=\vet{2\ch^{-1}(1+\kappa)^{\frac{1}{2}}\sin(x)}{\left(\ch^2+\ch^{-2}\right)(1+\kappa)\cos(x)}, \quad \mathscr{L}'_{0,0}f^-_0=\vet{0}{0},\\
        &\dot{\mathscr{L}}_{0,0} f^+_1=\vet{-\im b^{[3]}_{\tth,\kappa}\cos(x)}{-2\im\kappa(\ch/\sqrt{1+\kappa})^{\frac{1}{2}}\sin(x)}, \quad \dot{\mathscr{L}}_{0,0} f^-_1=\vet{\im b^{[3]}_{\tth,\kappa} \sin(x)}{-2\im\kappa(\ch/\sqrt{1+\kappa})^{\frac{1}{2}}\cos(x)},\\
        &\dot{\mathscr{L}}_{0,0} f^+_0=\vet{0}{0}, \quad \dot{\mathscr{L}}_{0,0}f^-_0=\vet{0}{0},
    \end{aligned}
\end{equation}
where
\begin{equation}  \label{b1 b2 b3}
\begin{aligned}
b^{[1]}_{\tth,\kappa}&:=\frac{1}{2}(1+\kappa)^{\frac{3}{4}}\ch^{\frac{5}{2}}(1-\ch^{-4}), \ \ 
    b^{[2]}_{\tth,\kappa}:=\frac{(1+\kappa)\ch^4+5\kappa-1}{2\ch^{\frac{3}{2}}(1+\kappa)^{\frac{1}{4}}}, \ \ 
b^{[3]}_{\tth,\kappa}:=\left(\frac{\ch}{\sqrt{1+\kappa}}\right)^{-\frac{1}{2}}\left(\ch^2+\tth(1-\ch^4)\right).
\end{aligned}
\end{equation}
We finally calculate $P'_{0,0} f^\sigma_k$ and $\dot{P}_{0,0} f^\sigma_k$.

\begin{lemma} One has
    \begin{equation} \label{A16}
        \begin{aligned}
    P'_{0,0}f^+_1&=\vet{\alpha_{\tth,\kappa}\cos(2x)}{\beta_{\tth,\kappa}\sin(2x)},\quad P'_{0,0}f^-_1=\vet{-\alpha_{\tth,\kappa}\sin(2x)}{\beta_{\tth,\kappa}\cos(2x)}, \quad 
            P'_{0,0}f^+_0= \delta_{\tth,\kappa}f^+_{-1}, \\
            P'_{0,0}f^-_0& =\vet{0}{0}, \quad \dot{P}_{0,0}f^+_0=\vet{0}{0}, \quad \dot{P}_{0,0}f^-_0=\vet{0}{0}, \quad 
            \dot{P}_{0,0}f^+_1=\im\frac{\gamma_{\tth, \kappa}}{4}f^-_{-1}, \quad  \dot{P}_{0,0}f^-_1=\im\frac{\gamma_{\tth, \kappa}}{4}f^+_{-1},
        \end{aligned}
    \end{equation}
where $\alpha_{\tth,\kappa}$,  $\beta_{\tth,\kappa}$, ${\gamma_{\tth, \kappa}}$ and $\delta_{\tth, \kappa}$ are defined in \eqref{47 coefficients}.
\end{lemma}
\begin{proof}
    We first calculate $P'_{0,0} f^+_1$. By \eqref{A3}, \eqref{A11-0}, \eqref{A11}, and \eqref{A15} we deduce
    \begin{align*}
        P'_{0,0} f^+_1=-\frac{1}{2\pi \im}\oint_{\Gamma} \frac{1}{\lambda} \left(\mathscr{L}_{0,0}-\lambda\right)^{-1} \vet{2\left(\frac{\ch}{\sqrt{1+\kappa}}\right)^{-\frac{1}{2}}\sin(2x)}{b^{[1]}_{\tth,\kappa}+b^{[2]}_{\tth,\kappa}\cos(2x)} d\lambda.
    \end{align*}
We note that 
\begin{equation*}
    \vet{2\left(\frac{\ch}{\sqrt{1+\kappa}}\right)^{-\frac{1}{2}}\sin(2x)}{b^{[1]}_{\tth,\kappa}+b^{[2]}_{\tth,\kappa}\cos(2x)}=b^{[1]}_{\tth,\kappa}f^-_0+\mathcal{W} , \qquad b^{[1]}_{\tth,\kappa} \mbox{ in } \eqref{b1 b2 b3} \ . 
\end{equation*}
Therefore by \eqref{A11} and \eqref{A14} there is an analytic function $\lambda \mapsto \varphi(\lambda,\cdot)\in H(\mathbb{T})$ so that, by exploiting the identity $\tanh(2\tth)=\frac{2\ch^2}{1+\ch^4}$ in applying \eqref{A14}, we obtain
\begin{align*}
    P'_{0,0} f^+_1=-\frac{1}{2\pi \im} \oint_\Gamma \frac{1}{\lambda}\left(-\frac{b^{[1]}_{\tth,\kappa}}{\lambda}f^-_0+\vet{-\alpha_{\tth,\kappa}\cos(2x)}{-\beta_{\tth,\kappa}\sin(2x)}+\lambda\varphi(\lambda)\right) d\lambda.
\end{align*}
Thus, by means of the redidue theorem we obtain the first identity in \eqref{A16}. Similarly, one may calculate $P'_{0,0} f^-_1$. By \eqref{A3}, \eqref{A11}, and \eqref{A15}, one has $P'_{0,0}f^-_0=0$. Next, we calculate $P'_{0,0}f^+_0$. By \eqref{A3}, \eqref{A11}, \eqref{A12}, and \eqref{A15} we get 
\[
P'_{0,0} f^+_0 = -\frac{1}{2\pi \im} \oint_{\Gamma} \frac{1}{\lambda} (\mathscr{L}_{0,0} - \lambda)^{-1} \vet{2\ch^{-1}(1+\kappa)^{\frac{1}{2}}\sin(x)}{\left(\ch^2+\ch^{-2}\right)(1+\kappa)\cos(x)} d\lambda.
\]

Next, we decompose
\begin{equation}
    \begin{aligned}
       \vet{2\ch^{-1}(1+\kappa)^{\frac{1}{2}}\sin(x)}{\left(\ch^2+\ch^{-2}\right)(1+\kappa)\cos(x)}&= \underbrace{\frac{1}{2}(1+\kappa)^{\frac{3}{4}}\ch^{\frac{1}{2}}\left(\ch^2+3\ch^{-2}\right)}_{=: \alpha} f^-_{-1} + \underbrace{\frac{1}{2}(1+\kappa)^{\frac{3}{4}}\ch^{\frac{1}{2}}\left(\ch^2-\ch^{-2}\right)}_{=: \beta} f^-_1. 
    \end{aligned}
\end{equation}
By \eqref{A15} and \eqref{A13} we get
\[
P'_{0,0} f^+_0 = -\frac{1}{2\pi \im} \oint_{\Gamma} 
\left( -\frac{2\alpha \chk}{\lambda (\lambda^2 + 4\chk^2)} f^+_{-1} 
- \frac{\alpha}{\lambda^2 + 4\chk^2} f^-_{-1} 
- \frac{\beta}{\lambda^2} f^-_{1} \right) d\lambda.
\]
Applying the residue theorem, we obtain
\[
P'_{0,0} f^+_0 = \frac{\alpha}{2 \chk} f^+_{-1}.
\]
Thus, we obtain the third identity in \eqref{A16}. Now, we calculate $\dot{P}_{0,0} f^+_1$. First, we have 
\begin{align*}
    \dot{P}_{0,0} f^+_1=\frac{1}{2\pi }\oint_\Gamma \frac{1}{\lambda}\left(\mathscr{L}_{0,0}-\lambda\right)^{-1} \vet{b^{[3]}_{\tth,\kappa} \cos(x)}{2\kappa\left(\frac{\ch}{\sqrt{1+\kappa}}\right)^{\frac{1}{2}}\sin(x)} d\lambda,
\end{align*}
where $b^{[3]}_{\tth,\kappa}$ is in \eqref{A15}, and then, writing
\begin{align*}
    &\vet{\cos(x)}{0}=\frac{1}{2}\left(\frac{\ch}{\sqrt{1+\kappa}}\right)^{-\frac{1}{2}}\left(f^+_1+f^+_{-1}\right), \\
    &\vet{0}{\sin(x)}=\frac{1}{2}\left(\frac{\ch}{\sqrt{1+\kappa}}\right)^{\frac{1}{2}}\left(f^+_1-f^+_{-1}\right),
\end{align*}
and using \eqref{A13}, we conclude using again the residue theorem 
\begin{align*}
    \dot{P}_{0,0}f^+_1=\im\left(\frac{1}{4}b^{[3]}_{\tth,\kappa}\ch^{-\frac{3}{2}}(1+\kappa)^{-\frac{1}{4}}-\frac{1}{2}\kappa(1+\kappa)^{-1}\right)  f^-_{-1}.
\end{align*}
Similarly, we have
\begin{align*}
    \dot{P}_{0,0}f^-_1=\im\left(\frac{1}{4}b^{[3]}_{\tth,\kappa}\ch^{-\frac{3}{2}}(1+\kappa)^{-\frac{1}{4}}-\frac{1}{2}\kappa(1+\kappa)^{-1}\right) f^+_{-1}.
\end{align*}
Finally, in view of \eqref{A15}, we have
\begin{equation*}
    \begin{aligned}
        \dot{P}_{0,0}f^+_0=&\frac{1}{2\pi \im}\oint_\Gamma \left(\mathscr{L}_{0,0}-\lambda\right)^{-1}\dot{\mathscr{L}}_{0,0}\left(\frac{1}{\lambda^2}f^-_0-\frac{1}{\lambda}f^+_0\right) d\lambda=0,\\
        \dot{P}_{0,0}f^-_0=&\frac{1}{2\pi \im}\oint_\Gamma \left(\mathscr{L}_{0,0}-\lambda\right)^{-1}\dot{\mathscr{L}}_{0,0}\left(\frac{-1}{\lambda}f^-_0\right) d\lambda=0.
    \end{aligned}
\end{equation*}
In conclusion, all the formulas in \eqref{A16} are proven.
\end{proof}

So far we have obtained the linear terms of the expansions \eqref{43 f+1}, \eqref{44 f-1}, \eqref{45 f+0}, and \eqref{46 f-0}. We now provide further information about the expansion of the basis at $\mu = 0$. 

\begin{lemma} 
     The basis $\{ f_k^{\sigma}(0,\e), \quad k = 0,1, \quad \sigma = \pm \}$ is real. For any $\epsilon$ it results $f_0^-(0,\epsilon) \equiv f_0^-$. The property \eqref{48 f-0=01} holds.
\end{lemma}

\begin{proof}
    The reality of the basis $f_k^\sigma(0,\epsilon)$ is a consequence of Lemma \ref{properties of U and P}-(iii). Since, recalling \eqref{mathscr L mu e} and \eqref{mathcal B mu e}, $\mathscr{L}_{0,\epsilon} f_0^- = 0$ for any $\epsilon$ (cf. \eqref{237}), we deduce $(\mathscr{L}_{0,\epsilon} - \lambda)^{-1} f_0^- = -\frac{1}{\lambda} f_0^-$ and then, using also the residue theorem,
\[
P_{0,\epsilon} f_0^- = -\frac{1}{2\pi \im} \oint_\Gamma (\mathscr{L}_{0,\epsilon} - \lambda)^{-1} f_0^- \, d\lambda = f_0^-.
\]
In particular $P_{0,\epsilon} f_0^- = P_{0,0} f_0^-$, for any $\epsilon$ and we get, by \eqref{U transformation operators}, $f_0^-(0,\epsilon) = U_{0,\epsilon} f_0^- = f_0^-$, for any $\epsilon$.

Let us prove property \eqref{48 f-0=01}. In view of \eqref{Parity structure} and since the basis is real, we know that
\[
f_k^+(0,\epsilon) = \begin{bmatrix}\text{even}(x)\\ \text{odd}(x)\end{bmatrix}, \quad f_k^-(0,\epsilon) = \begin{bmatrix}\text{odd}(x)\\ \text{even}(x)\end{bmatrix},
\]
for any $k = 0, 1$. By Lemma \ref{F is symplectic and reversible} the basis $\{f_k^\sigma(0,\epsilon)\}$ is symplectic (cf. \eqref{basis is symplectic} and, since
\[
\mathcal{J} f_0^-(0,\epsilon) = \mathcal{J} f_0^- = \begin{bmatrix}1\\0\end{bmatrix},
\]
for any $\epsilon$, we get
\[
0 = (\mathcal{J} f_0^-(0,\epsilon),\, f_1^+(0,\epsilon)) = \left(\begin{bmatrix}1\\0\end{bmatrix}, f_1^+(0,\epsilon)\right), \qquad 
1 = (\mathcal{J} f_0^-(0,\epsilon),\, f_0^+(0,\epsilon)) = \left(\begin{bmatrix}1\\0\end{bmatrix}, f_0^+(0,\epsilon)\right).
\]
Thus the first component of both $f_1^+(0,\epsilon)$ and $f_0^+(0,\epsilon) - \begin{bmatrix}1\\0\end{bmatrix}$ has zero average, proving \eqref{48 f-0=01}.
\end{proof}

\begin{lemma} \label{fmu0 eq f0} 
For any small $\mu$, we have $f_0^+(\mu,0) \equiv f_0^+$ and $f_0^-(\mu,0) \equiv f_0^-$. Moreover, the vectors $f_1^+(\mu,0)$ and $f_1^-(\mu,0)$ have both components with zero space average.
\end{lemma}
\begin{proof}
The operator $\mathscr{L}_{\mu,0} = 
\begin{bmatrix}
\chk\partial_x & |D+\mu|\tanh(\tth|D+\mu|) \\
-1+\kappa(\pa_{xx}+2\im \mu\pa_x-\mu^2) & \chk\partial_x
\end{bmatrix}$ 
leaves invariant the subspace $\mathcal{Z} := \text{span}\{ f_0^+, f_0^- \}$ since 
$\mathscr{L}_{\mu,0} f_0^+ = -(1+\kappa\mu^2)f_0^-$ and $\mathscr{L}_{\mu,0} f_0^- = \mu\tanh(\tth\mu) f_0^+$. 
The operator $\mathscr{L}_{\mu,0}|_{\mathcal{Z}}$ has the two eigenvalues $\pm \im \sqrt{(1+\kappa\mu^2)\mu\tanh(\tth\mu)}$, 
which, for small $\mu$, lie inside the loop $\Gamma$ around $0$ in \eqref{Projection P mu e}. Then, by \eqref{spectrum separated by Gamma}, 
we have $\mathcal{Z} \subseteq \mathcal{V}_{\mu,0} = \text{Rg}(P_{\mu,0})$ and 

\[
P_{\mu,0} f_0^\pm = f_0^\pm, \quad 
f_0^\pm (\mu,0) = U_{\mu,0} f_0^\pm = f_0^\pm, 
\quad \text{for any $\mu$ small.}
\]

The basis $\{ f_k^\sigma (\mu,0) \}$ is symplectic (cf. \eqref{basis is symplectic}). Then, since 
$\mathcal{J} f_0^+ = \begin{bmatrix} 0 \\ -1 \end{bmatrix}$ 
and $\mathcal{J} f_0^- = \begin{bmatrix} 1 \\ 0 \end{bmatrix}$, we have

\[
\begin{aligned}
0 &= \big( \mathcal{J} f_0^+ (\mu,0), f_1^\sigma (\mu,0) \big) 
= \Big( \begin{bmatrix} 0 \\ -1 \end{bmatrix}, f_1^\sigma (\mu,0) \Big), \quad 
0 = \big( \mathcal{J} f_0^- (\mu,0), f_1^\sigma (\mu,0) \big) 
= \Big( \begin{bmatrix} 1 \\ 0 \end{bmatrix}, f_1^\sigma (\mu,0) \Big),
\end{aligned}
\]
namely both the components of $f_1^\pm (\mu,0)$ have zero average.
\end{proof}

We finally consider the $\mu \epsilon$ term in the expansion \eqref{the expandsion of f sigma mu}.

\begin{lemma} 
The derivatives $ (\partial_\mu \partial_\epsilon f_k^{\sigma}(0,0) = \left( \dot{P}_{0,0}^{\prime} - \frac{1}{2} P_{0,0} \dot{P}_{0,0}^{\prime} \right) f_k^{\sigma}$ satisfy
\begin{equation} \label{A17}
\begin{aligned}
(\partial_\mu \partial_\epsilon f_1^+(0,0)) &= \im \begin{bmatrix} \mathit{odd}(x) \\ \mathit{even}(x) \end{bmatrix}, & (\partial_\mu \partial_\epsilon f_1^-(0,0)) &= \im \begin{bmatrix} \mathit{even}(x) \\ \mathit{odd}(x) \end{bmatrix}, \\
(\partial_\mu \partial_\epsilon f_0^+(0,0)) &= \im \begin{bmatrix} \mathit{odd}(x) \\ \mathit{even}_0(x) \end{bmatrix}, & (\partial_\mu \partial_\epsilon f_0^-(0,0)) &= \im \begin{bmatrix} \mathit{even}_0(x) \\ \mathit{odd}(x) \end{bmatrix}.
\end{aligned}
\end{equation}
\end{lemma}

 \begin{proof}
 We prove that $ \dot{P}_{0,0}^{\prime} = \eqref{A5a} + \eqref{A5b} + \eqref{A5c}$ is purely imaginary. This follows since the operators in \eqref{A5a}, \eqref{A5b}, and \eqref{A5c} are purely imaginary because $\dot{\mathscr{L}}_{0,0}$ is purely imaginary, $\mathscr{L}_{0,0}^{\prime}$ in \eqref{A6} is real, and $\dot{\mathscr{L}}_{0,0}^{\prime}$ in \eqref{A8} is purely imaginary (as argued in Lemma \ref{properties of U and P}-(iii) of \cite{BMV1}). Then, when applied to the real vectors $f_k^{\sigma}$, for $k = 0,1$ and $\sigma = \pm$, they produce purely imaginary vectors.

The property \eqref{Parity structure} implies that $ (\partial_\mu \partial_\epsilon f_k^{\sigma}(0,0)) $ has the claimed parity structure in \eqref{A17}. We shall now prove that $ (\partial_\mu \partial_\epsilon f_0^+(0,0)) $ has zero average. We have, by \eqref{A12} and \eqref{A15},
\begin{equation}
\eqref{A5a} f_0^+ := \frac{1}{2\pi i} \oint_{\Gamma} \frac{1}{\lambda}(\mathscr{L}_{0,0} - \lambda)^{-1} \dot{\mathscr{L}}_{0,0}(\mathscr{L}_{0,0} - \lambda)^{-1} \vet{2\ch^{-1}(1+\kappa)^{\frac{1}{2}}\sin(x)}{\left(\ch^2+\ch^{-2}\right)(1+\kappa)\cos(x)} d \lambda.
\end{equation}

Since the operators $(\mathscr{L}_{0,0} - \lambda)^{-1}$ and $\mathscr{L}_{0,0}$ are both Fourier multipliers, they preserve the absence of average in the vectors. Thus, \eqref{A5a} $f_0^+$ has zero average.
Next, \eqref{A5b} $f_0^+ = 0$ since $\dot{\mathscr{L}}_{0,0} f_0^{\pm} = 0$, cf. \eqref{A15}. Finally, by \eqref{A12} and \eqref{A8}, where $p_1(x) = p_1^{[1]} \cos(x)$ and $e_1(x)=e^{[1]}_1\sin(x)$,
\begin{equation}
\eqref{A5c} f_0^+ := \frac{1}{2 \pi \im} \oint_{\Gamma} (\mathscr{L}_{0,0} - \lambda)^{-1} \left(\frac{-1}{\lambda}\vet{\im p_1(x)}{\im \kappa e_1(x)} +\frac{1}{\lambda^2}\vet{0}{\im p_1(x)}\right) d \lambda
\end{equation}
is a vector with zero average. We conclude that $\dot{P}_{0,0}^{\prime} f_0^+$ is an imaginary vector with zero average, as well as $(\partial_\mu \partial_\epsilon f_0^+(0,0))$ since $P_{0,0}$ sends zero average functions into zero average functions. Finally, by \eqref{Parity structure}, $(\partial_\mu \partial_\epsilon f_0^+(0,0))$ has the claimed structure in \eqref{A17}.

We finally consider $(\partial_\mu \partial_\epsilon f_0^-(0,0))$. By \eqref{A11} and $\mathscr{L}_{0,0}^{\prime} f_0^- = 0$ (cf. \eqref{A15}), it results
\begin{equation}
\eqref{A5a} f_0^- = \frac{1}{2\pi i} \oint_{\Gamma} \frac{1}{\lambda}(\mathscr{L}_{0,0} - \lambda)^{-1} \dot{\mathscr{L}}_{0,0} (\mathscr{L}_{0,0} - \lambda)^{-1} \mathscr{L}_{0,0}^{\prime} f_0^- d \lambda = 0.
\end{equation}
Next, by \eqref{A11} and $\dot{\mathscr{L}}_{0,0} f_0^- = 0$, we get \eqref{A5b} $f_0^- = 0$. Finally, by \eqref{A11} and \eqref{A8},
\begin{equation}
\eqref{A5c} f_0^- = -\frac{1}{2\pi \im} \oint_{\Gamma} (\mathscr{L}_{0,0} - \lambda)^{-1} \frac{1}{\lambda} \begin{bmatrix} 0 \\ \im p_1(x)\end{bmatrix} d \lambda
\end{equation} has zero average since $(\mathscr{L}_{0,0} - \lambda)^{-1}$ is a Fourier multiplier (and thus preserves average absence).
 \end{proof}

This completes the proof of Lemma \ref{expansion of the basis F}.

\section{$\textup{D}_{\tth,\kappa}<0$ under Bond number condition }\label{sec:D.sign}
In this appendix we prove a sign condition on the function $\textup{D}_{\tth,\kappa}$ under the Bond number condition. 
\begin{lemma} \label{Bond number Dhk<0}
Let $(\kappa, \tth)$ satisfy the Bond number condition $ \frac{\kappa}{\tth^2}>\frac{1}{3}$, then the function $\textup{D}_{\tth,\kappa}$ in \eqref{def:D} satisfies $ \textup{D}_{\tth,\kappa}<0$ and $\lim\limits_{\tth\rightarrow 0^+} \textup{D}_{\tth,\kappa}=0$.
\end{lemma}
\begin{proof}
Recalling \eqref{e11 f11}, we write $\textup{D}_{\tth,\kappa}=(\sqrt{\tth}+\frac{1}{2}\mathsf{e}_{12})(\sqrt{\tth}-\frac{1}{2}\mathsf{e}_{12})=f_1(\tth,\kappa)f_2(\tth,\kappa)$ whose first factor is positive for $\tth>0$ and $\frac{\kappa}{\tth^2}>\frac{1}{3}$. We claim that the second factor is negative. We define 
\begin{equation*}
    \begin{aligned}
    &f_1(\kappa,\tth):=\sqrt{\tth}+\frac{1}{2}\mathsf{e}_{12}(\kappa,\tth)=\sqrt{\tth}+\frac{1}{2}(\ch^{-1}(1+\kappa)^{\frac{1}{2}} (\ch^2 + (1 - \ch^4)\tth)+\kappa(1+\kappa)^{-\frac{1}{2}}\ch,\\
    &f_2(\kappa,\tth):=\sqrt{\tth}-\frac{1}{2}\mathsf{e}_{12}(\kappa,\tth)=\sqrt{\tth}-\frac{1}{2}(\ch^{-1}(1+\kappa)^{\frac{1}{2}} (\ch^2 + (1 - \ch^4)\tth)-\kappa(1+\kappa)^{-\frac{1}{2}}\ch.
\end{aligned}
\end{equation*}
One may observe that $f_1(\kappa,\tth)>0$ for $\tth>0$ and $\kappa>\frac{\tth^2}{3}$. Also, since the functions $(1+\kappa)^{\frac{1}{2}}$ and $\kappa(1+\kappa)^{-\frac{1}{2}}$ are increasing for any $\kappa>0$, we only need to verify that $f_2(\frac{\tth^2}{3},\tth)<0$ for any $\tth>0$. 

Let us consider $f_0(\tth):=f_1(\frac{\tth^2}{3},\tth)f_2(\frac{\tth^2}{3},\tth)$. A straightforward computation reveals that
\begin{equation} \label{f<0}
    \begin{aligned}
        f_0(\tth)=-\frac{q(\tth)}{24(\cosh(2\tth)+1)^2\tanh(\tth)(\tth^2+3)},
    \end{aligned}
\end{equation}
where
\begin{equation} \label{qqq}
    \begin{aligned}
q(\tth)=&9\cosh(4\tth)-36\tth\sinh(4\tth)+18\tth^2\cosh(4\tth)+9\tth^4\cosh(4\tth)\\
&+72\tth^3\sinh(2\tth)-12\tth^3\sinh(4\tth)+24\tth^5\sinh(2\tth)+54\tth^2+39\tth^4+8\tth^6-9.
\end{aligned}
\end{equation}
Our goal here is to show $q(\tth)>0$ for $\tth>0$. An explicit computation reveals that
\begin{equation*}
\begin{aligned}
    &q(0)=\frac{d q}{d \tth}(0)=\frac{d^2 q}{d \tth^2}(0)=\frac{d^3 q}{d \tth^3}(0)=\frac{d^4 q}{d \tth^4}(0)=\frac{d^5 q}{d \tth^5}(0)=0, \\
    &\frac{d^6 q}{d \tth^6}(0)=23040, \quad \frac{d^7 q}{d \tth^7}(0)=0, \quad \frac{d^8 q}{d \tth^8}(0)=1806336,
\end{aligned}
\end{equation*}
and
\begin{equation*}
    \begin{aligned}
\frac{d^9 q}{d \tth^9}(\tth)=&8128512\sinh(2\tth)+5455872\sinh(4\tth)+2248704\tth^3\cosh(2\tth)+18087936\tth^3\cosh(4\tth)\\
&+12288\tth^5\cosh(2\tth)+8239104\tth^2\sinh(2\tth)+47185920\tth^2\sinh(4\tth)+276480\tth^4\sinh(2\tth)\\
&+2359296\tth^4\sinh(4\tth)+13602816\tth\cosh(2\tth)+43646976\tth\cosh(4\tth)>0,
\end{aligned}
\end{equation*}
for $\tth>0$, whence it follows from the mean value theorem that
\begin{equation} \label{q>0}
    \begin{aligned}
        q(\tth)&=\tth_5\tth_4\tth_3\tth_2\tth_1\tth \frac{d^6 q}{d \tth^6}(\tth_6)=\tth_5\tth_4\tth_3\tth_2\tth_1\tth \left(\tth_6\frac{d^7 q}{d \tth^7}(\tth_7)+23040\right)\\
        &=\tth_5\tth_4\tth_3\tth_2\tth_1\tth \left(\tth_7\tth_6\frac{d^8 q}{d \tth^8}(\tth_8)+23040\right)\\
        &=\tth_5\tth_4\tth_3\tth_2\tth_1\tth \left(\tth_8\tth_7\tth_6\frac{d^9 q}{d \tth^9}(\tth_9)+ 1806336\tth_7\tth_6 +23040\right)\\
        &= \tth_8\tth_7\tth_6\tth_5\tth_4\tth_3\tth_2\tth_1\tth\frac{d^9 q}{d \tth^9}(\tth_9)+ 1806336\tth_7\tth_6\tth_5\tth_4\tth_3\tth_2\tth_1\tth +23040 \tth_5\tth_4\tth_3\tth_2\tth_1\tth>0,
    \end{aligned}
\end{equation}
for some $\tth_j$, $j=1,\ldots,9$, such that $0<\tth_9<\ldots<\tth_1<\tth$.

Therefore, combining equations \eqref{q>0}, \eqref{f<0}, and \eqref{qqq}, we conclude that $\textup{D}_{\tth,\kappa}=\tth-\frac{1}{4}\mathsf{e}_{12}^2<0$. Finally, for any fixed $\kappa>0$, one checks easily that $\lim\limits_{\tth\rightarrow 0^+} \textup{D}_{\tth,\kappa}=0$.
\end{proof}

\footnotesize

\vspace{1em}
\noindent{\bf Acknowledgments:}
T.-Y. Hsiao and A. Maspero are  supported by   the European Union  ERC CONSOLIDATOR GRANT 2023 GUnDHam, Project Number: 101124921. A. Maspero is also supported by 
PRIN 2022 (2022HSSYPN).
 Views and opinions expressed are however those of the authors only and do not necessarily reflect those of the European Union or the European Research Council. Neither the European Union nor the granting authority can be held responsible for them.

\end{document}